\documentclass[10pt]{amsart}

\usepackage{amssymb}
\usepackage{pifont}
\usepackage{amscd}
\usepackage{amsmath}
\usepackage[dvips]{graphicx}
\usepackage{dsfont}
\usepackage[all]{xy}
\usepackage{xcolor}
\usepackage{amsfonts}

\usepackage[colorlinks=true, linkcolor=blue]{hyperref}

\usepackage{tikz-cd}
\usetikzlibrary{arrows.meta}
\usepackage{bm}  

\theoremstyle{definition}

\newtheorem{Thm}{Theorem}[section]
\newtheorem{Lem}[Thm]{Lemma}
\newtheorem{Def}[Thm]{Definition}

\newtheorem{Prop}[Thm]{Proposition}
\newtheorem{Ex}[Thm]{Example}
\newtheorem{Rmk}[Thm]{Remark}

\newtheorem{Claim}[Thm]{Claim}

\newcommand{\xto}{\xrightarrow}

\numberwithin{equation}{section}
\title[$A_{\infty}$-structures on Tate–Hochschild cohomology]%
      {$A_{\infty}$-structures on the additive decomposition of the Tate-Hochschild cohomology of a finite group algebra}

\author[X. Bian, L. Li, Y. Liu, T. Wang, Z. Wang, G. Zhou]%
       {Xiuli Bian, Longfei Li, Yuming Liu, Tianyun Wang, Zhengfang Wang, Guodong Zhou}

\address{Xiuli Bian, Guodong Zhou
\newline School of Mathematical Sciences
 \newline Key Laboratory of MEA (Ministry of Education)
   \newline  Shanghai Key laboratory of PMMP
\newline  East China Normal University
\newline 200241, Shanghai
\newline P. R. China}
\email{52205500018@stu.ecnu.edu.cn}
\email{gdzhou@math.ecnu.edu.cn}

\address{Longfei Li
\newline  Department of Mathematics 
 \newline  Kansas State University 
\newline   66502, Manhattan 
\newline USA}
\email{longfeili@ksu.edu}

\address{Yuming Liu
\newline School of Mathematical Sciences
\newline Laboratory of Mathematics and Complex Systems
\newline Beijing Normal University
\newline 100875, Beijing
\newline P. R. China  }
\email{ymliu@bnu.edu.cn}

\address{Tianyun Wang
\newline Capital Normal University High School
\newline 100048, Beijing 
\newline P. R. China  }
\email{850623539@163.com}

\address{Zhengfang Wang
\newline School of Mathematics
\newline Nanjing University
\newline 210093, Nanjing
\newline P. R. China  }
\email{zhengfangw@nju.edu.cn}

\date{\today}

\newcommand{\ot}{{\otimes}}

\newcommand{\lra}{\longrightarrow}

\newcommand{\ra}{\rightarrow}
\newcommand{\sdp}{\times\kern-.2em\vrule height1.1ex depth-.05ex}
\newcommand{\epi}{\lra \kern-.8em\ra}

 \normalbaselines
\begin{document}

\begin{abstract} 	

Firstly, for a finite group algebra, we provide a computational framework $\widehat{m}_n$ for  the Tate-Hochschild cochain complex in terms of the additive decomposition, by decomposing each planar n-ary tree into local two children and local three children. Secondly, we give  all $\widehat{m}_2$ formulas of the  Tate-Hochschild cochain complex in terms of the additive decomposition. Thirdly, we give explicit $A_{\infty}$-multiplication formulas for both the Hochschild cochain complex and the Hochschild chain complex under additive decompositions. Finally, we give $A_{\infty}$-multiplication formulas in the context of abelian groups.

\end{abstract}

\maketitle

\tableofcontents  

\section{Introduction}\

For an associative algebra $ A $, the concept of the Hochschild cohomology group $ \mathrm{HH}^*(A, A) $ was introduced by Hochschild in 1945~\cite{Hoc45}. This cohomology group is defined via the Hochschild cochain complex $ C^*(A, A) $, where $ C^n(A, A) $ is the space of linear maps from $ A^{\otimes n} $ to $ A $.  In~\cite{Ger63}, while studying the deformation theory of associative algebras in 1963, Gerstenhaber discovered that $ \mathrm{HH}^*(A, A) $ possesses a remarkably rich algebraic structure, now known as a \emph{Gerstenhaber algebra}. Specifically:
\begin{itemize}
    \item[(i)] $ \mathrm{HH}^*(A, A) $ is a graded-commutative associative algebra under the \emph{cup product};
    \item[(ii)] $ \mathrm{HH}^*(A, A) $ carries a \emph{graded Lie bracket of degree $-1$} (now called the \emph{Gerstenhaber bracket}), giving it the structure of a graded Lie algebra;
    \item[(iii)] The Gerstenhaber bracket and the cup product are \emph{compatible via the graded Leibniz rule}.
\end{itemize}

Subsequently, based on the complete resolution of a module, Tate introduced a theory of Tate cohomology~\cite{Tate52} that allows the definition of the cohomology groups $ \mathrm{H}^n $ for both positive and negative integers $ n $. Similarly, as an extension of the Hochschild cohomology to include the negative part, the \emph{Tate-Hochschild cohomology groups} $ \widehat{\mathrm{HH}}^*(A, A) $ were first defined by Buchweitz in \cite{Buch86}. In\cite{Wang21}, Wang constructed a new complex called \emph{singular Hochschild cochain complex}, which is used to calculate the Tate-Hochschild cohomology $ \widehat{\mathrm{HH}}^*(A, A) $ of an algebra $ A $. Using this complex, it was shown that $ \widehat{\mathrm{HH}}^*(A, A) $ admits a \emph{Gerstenhaber algebra} structure, which extends the Gerstenhaber structure in Hochschild cohomology $ \mathrm{HH}^*(A, A) $. Furthermore, in~\cite{RW19} Rivera and Wang extended the results of Tradler and Menichi, proving that when $ A $ is a symmetric algebra, the Gerstenhaber structure on its Tate-Hochschild cohomology can be extended to a \emph{Batalin–Vilkovisky (BV) algebra structure}. This result deepens the connection between higher homological algebraic structures and the symmetry of algebras.

In~\cite{RW19}, from the perspective of \lq\lq string topolog\rq\rq, the authors studied the Tate-Hochschild cohomology of finite-dimensional differential graded (dg) symmetric algebras. By constructing a Tate-Hochschild cochain complex $ \mathcal{D}^*(A, A) $, they realized a method to compute the Tate-Hochschild cohomology of a symmetric algebra $ A $. This complex possesses the following structure:

\begin{itemize}
  \item \textbf{Negative degrees:} Corresponding to the Hochschild chain complex $ C_*(A, A) $, with
  \[
  \mathcal{D}^{-m-1}(A, A) = C_m(A, A), \quad (m \geq 0);
  \]
  
  \item \textbf{Non-negative degrees:} Corresponding to the Hochschild cochain complex $ C^*(A, A) $, with
  \[
  \mathcal{D}^m(A, A) = C^m(A, A), \quad (m \geq 0);
  \]
  
  \item \textbf{Differential operator $ \tau $:} Defined in degree $-1$ as $ \tau \colon C_0(A, A) \to C^0(A, A) $, induced by the Casimir element $ \sum_i e_i \otimes f_i $, where
  \[
  a \mapsto \sum_i e_i a f_i.
  \]
\end{itemize}

In 1960, Stasheff introduced the notion of $ A_\infty $-algebras. In recent years, $ A_\infty $ structures have been applied in representation theory. For example, Keller~\cite{Keller01} used $ A_\infty $ structures to reconstruct complexes from homology groups. Inspired by Deligne's conjecture and Kontsevich's work on deformation quantization \cite{Kon03}, the focus has changed from the cohomology groups themselves to the higher structures in the complexes of the cohomology groups. One interesting problem is how to find non-trivial $ A_\infty $-structures. It has been pointed out in~\cite{LV12} that the construction of homotopy deformation retracts is one method to obtain $ A_\infty $-structures. Given a homotopy deformation retract, the homotopy transfer theorem can be used to lift an ordinary multiplication to an $ A_\infty $-multiplication. For example, it was shown in~\cite{RW19} that the complex $ \mathcal{D}^*(A, A) $ admits an $ A_\infty $-algebra structure $(m_1 = \partial, m_2, m_3, \dots)$ with $m_i = 0$  for $i > 3$.

More specifically, for a finite group $ G $, the group algebra $ kG $ over a field $ k $ is a typical example of a symmetric algebra, and serves as the main object of study in this paper. The cup product formula on $ \mathcal{D}^*(kG, kG) $ was given in~\cite{Wang21}, and the formula for $ m_3 $ was later obtained by Liu, Wang and Zhou in~\cite{LWZ21}. One concrete approach to constructing non-trivial $ A_\infty $-structures is using the additive decomposition of the Hochschild cohomology algebra of group algebras. It has been shown and proven in~\cite{SW99} that the Hochschild cohomology ring of $ kG $ admits the following additive decomposition:
\[
\mathrm{HH}^*(kG, kG) \cong \bigoplus_{x \in X} \mathrm{H}^*(C_G(x), k),
\]
where $ X $ is a complete set of representatives of the conjugacy classes of $ G $, and $ C_G(x) $ denotes the centralizer subgroup of $ x \in G $. Furthermore, in~\cite{LZ16} Liu and Zhou lifted the above decomposition to the complex level as follows:
\begin{equation*}
\xymatrix@C=0.000000001pc{
 		\ar@(lu, dl)_-{s^*}&C^*(kG,  kG) \ar@<0.5ex>[rrrrr]^-{\iota^*}&&&&&\bigoplus\limits_{x\in G}  C^*(G,  k). \ar@<0.5ex>[lllll]^-{\rho^*}}
 \end{equation*}

 Furthermore, it was shown in~\cite{LZ16} that the Tate-Hochschild cohomology group $ \widehat{\mathrm{HH}}^*(kG,kG) $, as a $ k $-linear space, admits the following additive decomposition:
\begin{equation*}
    \widehat{\mathrm{HH}}^*(kG,kG) \cong \bigoplus_{x \in X} \widehat{\mathrm{H}}^*(C_G(x),k),
\end{equation*}
where $ \widehat{\mathrm{H}}^*(C_G(x),k) $ denotes the Tate cohomology group of the centralizer subgroup $ C_G(x) $. In~\cite{LWZ21}, Liu, Wang and Zhou used the Tate-Hochschild cochain complex $ \mathcal{D}^*(kG,kG) $ for the group algebra and the Tate cochain complex $ \widehat{C}^*(C_G(x),k) $ for the finite group to give the additive decomposition of the Tate-Hochschild cohomology of the group algebra explicitly at the complex level, it can be  realized as the following homotopy deformation retraction:
\begin{equation*}
  	\xymatrix@C=0.000000001pc{
  		\ar@(lu, dl)_-{\hat{s}}& \mathcal{D}^{*}(k G,  k G)\ar@<0.5ex>[rrrrr]^-{\hat{\rho^*}}&&&&& \underset{x \in X}{\bigoplus} \widehat{C}^{*}\left(C_{G}(x),  k\right) .\ar@<0.5ex>[lllll]^-{\hat{\iota^*}}}
  \end{equation*}

This homotopy deformation retraction at the complexe level provides examples of non-trivial $ A_\infty $-algebra structures. In 2020~\cite{Li20}, Li used a similar homotopy deformation retraction for the additive decomposition of the Hochschild cochain complex:
 \begin{equation*}
 \xymatrix@C=0.000000001pc{
    		\ar@(lu, dl)_-{s^*}&C^*(kG,  kG) \ar@<0.5ex>[rrrrr]^-{\iota^*}&&&&&\bigoplus\limits_{x\in G}  C^*(G,  k), \ar@<0.5ex>[lllll]^-{\rho^*}}
    \end{equation*}
and, by categorizing planar binary trees into two local transformation types ($\alpha$-type and $\beta$-type), established the formula for the $ A_\infty $-multiplication on the right-hand side. That is, we have the following result: 

The expansion of the $ A_\infty $-multiplication formulas for the additive decomposition of the Hochschild cohomology algebra of $ kG $ is, up to a sign, determined by the following equation:
Let
\[
\widehat{\varphi}_i : C_G(x)^{\times n_i} \to k, \quad i = 1,2,\dots,n,
\]
then the multiplication $ m_n $ is given by:
\[
m_n(\widehat{\varphi}_1,\dots,\widehat{\varphi}_n)(g_{1,1},\dots,g_{1,i_1},g_{2,1},\dots,g_{2,i_2},\dots,g_{n,1},\dots,g_{n,i_n})
\]
\[
= \sum \pm \widehat{\varphi}_1(h_{1,1},\dots,h_{1,j_1}) \widehat{\varphi}_2(h_{2,1},\dots,h_{2,j_2}) \cdots \widehat{\varphi}_n(h_{n,1},\dots,h_{n,j_n}),
\]
where the tuples $ (h_{1,1},\dots,h_{1,j_1}),(h_{2,1},\dots,h_{2,j_2}),\dots,(h_{n,1},\dots,h_{n,j_n}) $ are determined by the $ \alpha $- and $ \beta $-type transformations of a planar binary tree.

In actual computations, we found that the classification of local decompositions of planar binary trees in \cite{Li20} does not cover all cases arising in the computation of the $ A_\infty $-algebra. Therefore, building upon the existing work on local decompositions of planar binary trees, we further refine the $\alpha$- and $\beta$-type transformations of planar binary trees, resulting in a more specific and operational computational procedure.

Furthermore, based on the following homotopy deformation retraction at the complex level of the Tate-Hochschild cochain complex of group algebras:
\begin{equation*}
		\xymatrix@C=0.000000001pc{
			\ar@(lu, dl)_-{\hat{s}}& \mathcal{D}^{*}(k G,  k G)\ar@<0.5ex>[rrrrr]^-{\hat{\rho^*}}&&&&& \underset{x \in X}{\bigoplus} \widehat{C}^{*}\left(C_{G}(x),  k\right) , \ar@<0.5ex>[lllll]^-{\hat{\iota^*}}}
	\end{equation*}
we may decompose planar $ n $-ary trees locally to compute the $ A_\infty $-multiplication formulas for both the Hochschild and Tate-Hochschild cohomology algebras of the group algebra $ kG $ under additive decomposition.

The layout of this paper is as follows.

In Section~\ref{Section: preliminaries}, we recall the basic definitions related to $A_{\infty}$-algebra and the homotopy transfer theorem, the definition of Tate-Hochschild cohomology together with its chain complex, the Tate cohomology complex of a finite group, as well as the additive decomposition of the Tate-Hochschild cohomology complex and the corresponding homotopy deformation retracts.

In Section~\ref{Section:mn-algorithm}, we present a method for computing the operation $\widehat{m}_n$ on the additive decomposition of the Tate-Hochschild cohomology at the complex level. The key idea is to perform a local decomposition on each $PT_n$, splitting it into locally two-branched and locally three-branched graphs. The corresponding local algorithms are $\alpha_{i\pm,j\pm,\pm}$, $\beta_{i\pm,j\pm,\pm}$ for the two-branched graphs, and $\alpha_{i+,j-,k+,+}$, $\alpha_{i-,j+,k-,-}$, $\beta_{i+,j-,k+,+}$, $\beta_{i-,j+,k-,-}$ for the three-branched graphs, where  $i,j\in\{0,1\}$. Through this local analysis, we obtain an explicit computational procedure for $\widehat{m}_n$. In other words, any planar $n$-ary tree can be divided into a composition of the algorithms $\alpha_{i\pm,j\pm,\pm}$, $\alpha_{i\pm,j\pm,k\pm,\pm}$, $\beta_{i\pm,j\pm,\pm}$, and $\beta_{i\pm,j\pm,k\pm,\pm}$. The corresponding algorithmic flowchart is given in Figure~\ref{figure:n-ary-algorithm}.

In Section~\ref{Section: computation for alpha and beta}, through explicit computations of the operations $\alpha_{1\pm,1\pm,\pm}$ and $\beta_{1\pm,1\pm,\pm}$, we obtain explicit formulas for cup product $\widehat{m}_2$ on the additive decomposition of the Tate-Hochschild cohomology at the complex level, stated as Theorem~\ref{Thm: Main 1}. Compared with the formulas given in \cite{LZ16, LWZ21}, our result further includes the cup product formula $\widehat{m}_2$ between $\mathcal{D}^{\ge 0}(kG, kG)$ and $\mathcal{D}^{< 0}(kG, kG)$, which was not covered in these earlier works.

In Section~\ref{section:Ainfinity-on-the-additive-decomposition}, by carrying out explicit computations for the algorithms corresponding to the  locally two-branched graphs $\alpha_{i+,j+,+}$ and $\beta_{i+,j+,+}$, we define the operations $\alpha_{i+,j+,+}$ and $\beta_{i+,j+,+}$. This enables us to derive all expressions of $\widehat{m}_n$ on the additive decomposition of the Hochschild cohomology at the complex level, as stated in Theorem~\ref{Theorem: main theorem for cohomology}. We also point out that this result provides a refined and corrected version of the formulas for $\widehat{m}_n$ given in \cite{Li20}.

In Section~\ref{section:Ainfinity-on-Hoch-homo}, by performing explicit computations for the algorithms corresponding to the  locally two-branched graphs $\alpha_{i-,j-,-}$ and $\beta_{i-,j-,-}$, we define the operations $\alpha_{i-,j-,-}$ and $\beta_{i-,j-,-}$. This allows us to derive all expressions of $\widehat{m}_n$ on the additive decomposition of the Hochschild homology at the complex level, as stated in Theorem~\ref{thm:mn-on-homo}.

In Section~\ref{section:abelian-group-case}, we further specialize the finite group $G$ under discussion to the case where $G$ is abelian. Building on the results of \cite[Corollary 4.11]{LWZ21}, we obtain explicit expression for the $A_{\infty}$-operations $\widehat{m}_n$ on the additive decomposition of the Tate-Hochschild cohomology of a finite abelian group $G$ at the complex level, as stated in Theorem~\ref{theorem:Ainfity-finite-group}. Finally, we compute the $A_{\infty}$-algebra structures explicitly for several examples of abelian groups.

\textbf{Conventions:} Throughout $k$ denotes a fixed field and all algebraic structures discussed in this paper will be defined over $k$. That is, a vector space will be a $k$-vector space, an algebra will be a $k$-algebra, and so on. We shall write $\otimes$ for $\otimes_{k}$, the tensor product over the field $k$.

\section{Preliminaries}\label{Section: preliminaries}

\subsection{\texorpdfstring{Homotopy transfer theorem for $A_\infty$-algebras}{Homotopy transfer theorem for A-infinity algebras}}\

From~\cite{LV12}, if $V$ is a chain complex homotopy equivalent to a differential graded (dg) algebra $A$, then the dg algebra structure on $A$ can be transferred to an $A_{\infty}$-algebra structure on $V$. More generally, an $A_{\infty}$-algebra structure on $A$ can be transferred to an $A_{\infty}$-algebra structure on $V$. We now proceed to a detailed review of the above content.

\begin{Def} {(\cite{Sta63})}
An \textbf{$ A_\infty $-algebra} over $k$ is a $\mathbb{Z}$-graded vector space $A=\bigoplus\limits_{p\in \mathbb{Z}}A_p$, endowed with graded maps
 \[
 m_n: A^{\otimes n}\rightarrow A\quad (n\geq 1),
 \]
of degree $2-n$, satisfying the following identities:
\begin{enumerate}
    \item For $n=1$, 
    \[
    m_1m_1=0,
    \]
    so $(A,m_1)$ is a cochain complex.

    \item For $n=2$,
    \[
    m_1m_2=m_2(m_1\ot 1_A+1_A\ot m_1),
    \]
    which means that $m_1$ acts as a derivation of $m_2$.

    \item For $n=3$,
    \[
    m_2(1_A\ot m_2-m_2\ot1_A)=m_1 m_3+m_3(m_1\ot 1_A\ot1_A+1_A\ot m_1\ot1_A+1_A\ot 1_A\ot m_1),
    \]
    expressing that $m_2$ is associative up to the homotopy provided by $m_3$.

    \item More generally, for every $n\ge 1$, 
    \[
 \sum\limits_{\begin{array}{c} n=r+s+t\\ r,t\geq 0, s\geq 1\end{array}}(-1)^{r+st} m_{r+1+t}(1^{\otimes r} \otimes m_s \otimes 1^{\otimes t})=0.
 \]
\end{enumerate}
\end{Def}

In particular, if $m_n=0$ with $n\geq 3$, then $(A, m_1, m_2)$ is just a dg algebra.

\begin{Def}
Given two chain complexes $(W, d_W)$ and $(V, d_V)$, along with two chain maps $p: W \to V$ and $i: V \to W$, if the following homotopy deformation retraction diagram holds:
\begin{equation*}
\xymatrix@C=0.000000001pc{
\ar@(lu,dl)_-{h} & (W, d_W) \ar@<0.5ex>[rrrrr]^-{p} &&&&& (V, d_V) \ar@<0.5ex>[lllll]^-{i}}
\end{equation*}
that is, if the identities $1_W - i p = d_W h + h d_W$ and $1_V = p i$ are satisfied,  
then $V$ is called a \textbf{homotopy deformation retraction core} of $W$.  
In particular, if in addition $h^2 = 0$, $h i = 0$, and $p h = 0$, then $V$ is called a \textbf{strong homotopy deformation retraction core} of $W$.
\end{Def}

We remark that there are some related notations of the above definition, see for example \cite[Definition 1.1]{CLZ}. 
\emph{In the following discussion of this paper, unless otherwise specified, all homotopy deformation retraction cores we considered are strong homotopy deformation retraction cores.} Before introducing the homotopy transfer theorem, we first present the relevant definitions and notations concerning planar rooted trees.

A \textbf{rooted tree} is an undirected connected graph without cycles. For a precise definition, see \cite{LV12}. We denote by $PT_n$ the set of planar trees with $n$ leaves. For example:
\vspace{2mm}
\[PT_1:=\{|\},\
 PT_2 := \left\{
\begin{tikzpicture}[scale=0.3, baseline={(0,-0.1)}, every path/.style={line width=0.6pt}]
    \tikzstyle{every node}=[{inner sep=0pt, minimum size=0pt, outer sep=0pt}]
    \node (0) at (0,0)  {};
    \node (1) at (0,-1) {};
    \node (2) at (-1,1) {};
    \node (3) at (1,1) {};
    \node (a0) at (0,1.4)  {};
    \node (a1) at (0,-1.6)  {};
    \node (a2) at (-1.5,-0.1)  {};
    \node (a3) at (1.5,-1.6)  {};
    \draw  (0)--(1);
    \draw  (0)--(2);
    \draw  (0)--(3);
\end{tikzpicture}
\right\},\ PT_3 := \left\{
\begin{tikzpicture}[scale=0.28, baseline={(2.5,0.05)}, every path/.style={line width=0.6pt}]
    \tikzstyle{every node}=[{inner sep=0pt, minimum size=0pt, outer sep=0pt}]
    \node (0) at (0,0)  {};
    \node (1) at (0,-1) {};
    \node (2) at (-1,1) {};
    \node (3) at (2,2) {};
    \node (4) at (-2,2) {};
    \node (5) at (0,2) {};
    \node (0-1) at (5,0)  {};
    \node (1-1) at (5,-1) {};
    \node (2-1) at (6,1) {};
    \node (3-1) at (7,2) {};
    \node (4-1) at (3,2) {};
    \node (5-1) at (5,2) {};
    \node (0-2) at (10,0)  {};
    \node (1-2) at (10,-1) {};
    \node (2-2) at (10,2) {};
    \node (3-2) at (12,2) {};
    \node (4-2) at (8,2) {};
    \node (b1) at (2.5,0) {,};
    \node (b2) at (7.5,0) {,};
    \node (a0) at (-2.5,0)  {};
    \node (a1) at (12.5,0)  {};
    \node (a2) at (0,-1.5)  {};
    \node (a3) at (0,2.5)  {};
    \draw  (0)--(1);
    \draw  (0)--(2);
    \draw  (0)--(3);
    \draw  (2)--(4);
    \draw  (2)--(5);
    \draw  (0-1)--(1-1);
    \draw  (0-1)--(4-1);
    \draw  (0-1)--(2-1);
    \draw  (2-1)--(5-1);
    \draw  (2-1)--(3-1);
    \draw  (0-2)--(1-2);
    \draw  (0-2)--(4-2);
    \draw  (0-2)--(2-2);
    \draw  (0-2)--(3-2);
\end{tikzpicture}
\right\},\]
\vspace{2mm}
\[PT_4 := \left\{
\begin{tikzpicture}[scale=0.28, baseline={(2.5,0.2)}, every path/.style={line width=0.6pt}]
    \tikzstyle{every node}=[{inner sep=0pt, minimum size=0pt, outer sep=0pt}]
    \node (0) at (0,0)  {};
    \node (1) at (0,-1) {};
    \node (2) at (-1,1) {};
    \node (3) at (-2,2) {};
    \node (4) at (-3,3) {};
    \node (5) at (-1,3) {};
    \node (6) at (1,3) {};
    \node (7) at (3,3) {};
    \node (0-1) at (9.5,0)  {};
    \node (1-1) at (9.5,-1) {};
    \node (2-1) at (8.5,1) {};
    \node (3-1) at (8.5,3) {};
    \node (4-1) at (6.5,3) {};
    \node (5-1) at (10.5,3) {};
    \node (6-1) at (12.5,3) {};
    \node (0-2) at (19,0)  {};
    \node (1-2) at (19,-1) {};
    \node (2-2) at (16,3) {};
    \node (3-2) at (18,3) {};
    \node (4-2) at (20,3) {};
    \node (5-2) at (22,3) {};
    \node (b1) at (3.7,0) {,};   
    \node (b2) at (4.8,0) {$\cdots$};   
    \node (b3) at (5.8,0) {,};
    \node (b4) at (13.2,0) {,};   
    \node (b5) at (14.3,0) {$\cdots$};   
    \node (b6) at (15.3,0) {,};
    \node (a0) at (-3.5,0)  {};
    \node (a1) at (22.5,0)  {};
     \node (a2) at (0,3.5)  {};
    % \node (a3) at (0,-1.5)  {};
    \draw  (0)--(1);
    \draw  (0)--(2);
    \draw  (0)--(7);
    \draw  (2)--(3);
    \draw  (2)--(6);
    \draw  (3)--(4);
    \draw  (3)--(5);
    \draw  (0-1)--(1-1);
    \draw  (0-1)--(2-1);
    \draw  (0-1)--(6-1);
    \draw  (2-1)--(3-1);
    \draw  (2-1)--(4-1);
    \draw  (2-1)--(5-1);
    \draw  (0-2)--(1-2);
    \draw  (0-2)--(2-2);
    \draw  (0-2)--(3-2);
    \draw  (0-2)--(4-2);
    \draw  (0-2)--(5-2);
\end{tikzpicture}
\right\}.\]
\vspace{2mm}

In particular, a rooted tree in which each vertex has at most two leaves is called a \textbf{planar binary tree}, abbreviated as $PBT$. For example:
\vspace{2mm}
\[PBT_1:=\{\textbf{|}\},\
 PBT_2 := \left\{
\begin{tikzpicture}[scale=0.3, baseline={(0,-0.1)}, every path/.style={line width=0.6pt}]
    \tikzstyle{every node}=[{inner sep=0pt, minimum size=0pt, outer sep=0pt}]
    \node (0) at (0,0)  {};
    \node (1) at (0,-1) {};
    \node (2) at (-1,1) {};
    \node (3) at (1,1) {};
    \node (a0) at (0,1.4)  {};
    \node (a1) at (0,-1.6)  {};
    \node (a2) at (-1.5,-0.1)  {};
    \node (a3) at (1.5,-1.6)  {};
    \draw  (0)--(1);
    \draw  (0)--(2);
    \draw  (0)--(3);
\end{tikzpicture}
\right\},\ PBT_3 := \left\{
\begin{tikzpicture}[scale=0.25, baseline={(2.5,0.05)}, every path/.style={line width=0.6pt}]
    \tikzstyle{every node}=[{inner sep=0pt, minimum size=0pt, outer sep=0pt}]
    \node (0) at (0,0)  {};
    \node (1) at (0,-1) {};
    \node (2) at (-1,1) {};
    \node (3) at (2,2) {};
    \node (4) at (-2,2) {};
    \node (5) at (0,2) {};
    \node (0-1) at (6,0)  {};
    \node (1-1) at (6,-1) {};
    \node (2-1) at (7,1) {};
    \node (3-1) at (8,2) {};
    \node (4-1) at (4,2) {};
    \node (5-1) at (6,2) {};
    \node (b1) at (3,0) {,};
    \node (a0) at (-2.5,0)  {};
    \node (a1) at (8.5,0)  {};
    \node (a2) at (0,-1.5)  {};
    \node (a3) at (0,2.5)  {};
    \draw  (0)--(1);
    \draw  (0)--(2);
    \draw  (0)--(3);
    \draw  (2)--(4);
    \draw  (2)--(5);
    \draw  (0-1)--(1-1);
    \draw  (0-1)--(4-1);
    \draw  (0-1)--(2-1);
    \draw  (2-1)--(5-1);
    \draw  (2-1)--(3-1);
\end{tikzpicture}
\right\},\]
\vspace{2mm}
\[PBT_4 := \left\{
\begin{tikzpicture}[scale=0.28, baseline={(2.5,0.2)}, every path/.style={line width=0.6pt}]
    \tikzstyle{every node}=[{inner sep=0pt, minimum size=0pt, outer sep=0pt}]
    \node (0) at (0,0)  {};
    \node (1) at (0,-1) {};
    \node (2) at (-1,1) {};
    \node (3) at (-2,2) {};
    \node (4) at (-3,3) {};
    \node (5) at (-1,3) {};
    \node (6) at (1,3) {};
    \node (7) at (3,3) {};
    \node (0-1) at (8,0)  {};
    \node (1-1) at (8,-1) {};
    \node (2-1) at (7,1) {};
    \node (3-1) at (5,3) {};
    \node (4-1) at (7,3) {};
    \node (5-1) at (9,3) {};
    \node (6-1) at (11,3) {};
    \node (7-1) at (8,2) {};
    \node (0-2) at (16,0)  {};
    \node (1-2) at (16,-1) {};
    \node (2-2) at (14,2) {};
    \node (3-2) at (18,2) {};
    \node (4-2) at (13,3) {};
    \node (5-2) at (15,3) {};
    \node (6-2) at (17,3) {};
    \node (7-2) at (19,3) {};
    \node (0-3) at (24,0)  {};
    \node (1-3) at (24,-1) {};
    \node (2-3) at (25,1) {};
    \node (3-3) at (21,3) {};
    \node (4-3) at (23,3) {};
    \node (5-3) at (25,3) {};
    \node (6-3) at (27,3) {};
    \node (7-3) at (24,2) {};
    \node (0-4) at (32,0)  {};
    \node (1-4) at (32,-1) {};
    \node (2-4) at (33,1) {};
    \node (3-4) at (29,3) {};
    \node (4-4) at (31,3) {};
    \node (5-4) at (33,3) {};
    \node (6-4) at (35,3) {};
    \node (7-4) at (34,2) {};
    \node (b1) at (4,0) {,};   
    \node (b2) at (12,0) {,};  
    \node (b3) at (20,0) {,}; 
    \node (b4) at (28,0) {,}; 
    \node (a0) at (-3.5,0)  {};
    \draw  (0)--(1);
    \draw  (0)--(2);
    \draw  (0)--(7);
    \draw  (2)--(3);
    \draw  (2)--(6);
    \draw  (3)--(4);
    \draw  (3)--(5);
    \draw  (0-1)--(1-1);
    \draw  (0-1)--(2-1);
    \draw  (0-1)--(6-1);
    \draw  (2-1)--(3-1);
    \draw  (2-1)--(5-1);
    \draw  (4-1)--(7-1);
    \draw  (0-2)--(1-2);
    \draw  (0-2)--(4-2);
    \draw  (0-2)--(7-2);
    \draw  (2-2)--(5-2);
    \draw  (3-2)--(6-2);
    \draw  (0-3)--(1-3);
    \draw  (0-3)--(3-3);
    \draw  (0-3)--(6-3);
    \draw  (2-3)--(4-3);
    \draw  (5-3)--(7-3);
    \draw  (0-4)--(1-4);
    \draw  (0-4)--(3-4);
    \draw  (0-4)--(6-4);
    \draw  (2-4)--(4-4);
    \draw  (5-4)--(7-4);
\end{tikzpicture}
\right\}.\]
\vspace{2mm}

The homotopy transfer theorem for  $A_\infty$-structures is stated as follows:
\begin{Thm} {(\cite{LV12})}
	Let $ (W,d_W) $ be an $A_{\infty}$-algebra and we have the following homotopy retract:
\begin{equation*}
\xymatrix@C=0.000000001pc{
\ar@(lu,dl)_-{h} & (W, d_W) \ar@<0.5ex>[rrrrr]^-{p} &&&&& (V, d_V) \ar@<0.5ex>[lllll]^-{i}}
\end{equation*}
with $1_W - i p = d_W h + h d_W$ and $1_V = p i$, then $(V,d_V)$ inherits an $A_{\infty}$-algebra structure $\{m'_n\}_{n\ge 1}$ from $ (W,d_W) $. Specifically,
\end{Thm}
\[
\begin{tikzpicture}
	\tikzstyle{every node}=[thick, minimum size=4pt, inner sep=1pt]
    \tikzset{dot/.style={inner sep=0pt, minimum size=0pt, outer sep=0pt}}

	\node(0) at (0, 0) {$m_2$};
    \node(1) at (0, -1.25) {};
	\node(3) at (-1.5, 1) {$m_3$};
    \node(4) at (1.5, 1){$ m_2 $};
    \node(5) at (-4, 0.7){$ m_n'=\sum\limits_{\mathrm{PT_n}}\pm $};
    \node(6) at (4,0.7) {.};
	\node[dot] (a1) at (-2.5, 2){};
	\node[dot] (a2) at (-1.5, 2){};
    \node[dot] (a3) at (-0.5, 2){};
    \node[dot] (a4) at (0.5, 2){};
    \node(a5) at (2.5,2) {$m_2$};
    \node[dot] (b1) at (-2.5, 2.8){};
	\node[dot] (b2) at (-1.5, 2.8){};
    \node[dot] (b3) at (-0.5, 2.8){};
    \node[dot] (b4) at (0.5, 2.8){};
    \node[dot] (b5) at (1.5, 2.8){};
    \node[dot] (b6) at (3.5, 2.8){};
    \node[dot] (c1) at (1.5, 3.4){};
    \node[dot] (c2) at (3.5, 3.4){};
    
	\draw (0)-- node[pos=0.4, above, yshift=3pt] {$h$} (3);
    \draw[->] (0)-- node[pos=0.5,right,xshift=2pt] {$p$} (1);
	\draw (0)-- node[pos=0.4, above, yshift=3pt] {$h$} (4);
    \draw (4)-- node[pos=0.4, above, yshift=3pt] {$h$} (a5);
    \draw (4)--(a4);
	\draw (3)--(a1);
	\draw (3)--(a2);
    \draw (3)--(a3);
    \draw[->] (b1)-- node[pos=0.45, right, xshift=1pt] {$i$} (a1);
	\draw[->] (b2)-- node[pos=0.45, right, xshift=1pt] {$i$} (a2);
    \draw[->] (b3)-- node[pos=0.45, right, xshift=1pt] {$i$} (a3);
    \draw[->] (b4)-- node[pos=0.45, right, xshift=1pt] {$i$} (a4);
    \draw (a5)--(b5);
    \draw (a5)--(b6);
    \draw[->] (c1)-- node[pos=0.45, right, xshift=1pt] {$i$} (b5);
    \draw[->] (c2)-- node[pos=0.45, right, xshift=1pt] {$i$} (b6);
\end{tikzpicture}
\]

In fact,
for any $t\in PT_n$,  the  $n$-ary operation $m'_t$ is obtained  by putting $i$ on the leaves, $m_x$ on the vertice if this vertice has $x$ leaves, $h$ on the internal edges and $p$ on the root. The summation is taken over every planar tree in $PT_n$.

In particular, if $(W, d_W)$ is a differential graded associative algebra, then
\[
m'_n = \sum_{t \in PBT_n} \pm m_t,
\]
where each $m_t$ is defined by placing $i$ at the leaves of the tree, $m_2$ at the vertices, $h$ on the internal edges, and $p$ at the root.

\begin{Ex}
    Formulas for $m'_2$ and $m'_3$:
    \vspace{4mm}
    
\begin{tikzpicture}[scale=0.8]
	\tikzstyle{every node}=[thick,minimum size=6pt, inner sep=1pt]
    
    \node (a0) at (-5.5,-0.5)  {$m_2'$};
    \node (a1) at (-4.5,-0.5)  {$=$};
    \node (a2) at (0.5,-0.5)  {$=$};
    \node (a3) at (2,-0.5)  {$pm_2(i,i)$,};
	\node (b0) at (-2,-0.5)  {$m_2$};
	\node (b1) at (-3.5,1)  [minimum size=0pt]{$V$};
	\node (b3) at (-0.5,1)  [minimum size=0pt]{$V$};
	\node (b4) at (-2,-2)  [minimum size=0pt]{};
    
	\draw         (b0)-- node[pos=0.35, above, yshift=3pt] {$i$}(b1);
	\draw         (b0)-- node[pos=0.35, above, yshift=3pt] {$i$}(b3);
	\draw         (b0)-- node[pos=0.5, left, xshift=-3pt] {$p$}(b4);
\end{tikzpicture}
\vspace{4mm}

\begin{tikzpicture}[scale=0.7]
	\tikzstyle{every node}=[thick, minimum size=4pt, inner sep=1pt]

    \node(a0) at (-4,0) {$m_3'$};
    \node(a1) at (-3,0) {$=$};
    \node(a2) at (2.5,0) {$-$};
    \node(a3) at (7.5,0) {$+$};
    \node(a4) at (-3,-2.5) {$=$};
    \node(a5) at (-0.5,-2.5) {$pm_3(i,i,i)$};
    \node(a6) at (2.5,-2.5) {$-$};
    \node(a7) at (5,-2.5) {$pm_2(hm_2(i,i),i)$};
    \node(a8) at (7.5,-2.5) {$+$};
    \node(a9) at (10,-2.5) {$pm_2(i,hm_2(i,i)).$};
	\node(b1) at (0, 0) {$m_3$};
    \node(b0) at (0,-1.5) {};
	\node(b3) at (-2, 2){$ V $};
    \node(b4) at (0, 2){$ V $};
	\node(b5) at (2, 2){$ V $};
    \node(c1) at (5, 0) {$m_2$};
    \node(c0) at (5,-1.5) {};
	\node(c3) at (3, 2){$ V $};
    \node(c4) at (5, 2){$ V $};
	\node(c5) at (7, 2){$ V $};  
    \node(c6) at (4, 1){$ m_2 $};
    \node(d1) at (10, 0) {$m_2$};
    \node(d0) at (10,-1.5) {};
	\node(d3) at (8, 2){$ V $};
    \node(d4) at (10, 2){$ V $};
	\node(d5) at (12, 2){$ V $};
    \node(d6) at (11, 1) {$m_2$};
    
    \draw (b0)-- node[pos=0.45, left, xshift=-3pt] {$p$} (b1);
	\draw (b1)-- node[pos=0.35, above, yshift=3pt] {$i$} (b3);
	\draw (b1)-- node[pos=0.35, above, xshift=-3pt] {$i$} (b4);
	\draw (b1)-- node[pos=0.35, above, yshift=3pt] {$i$} (b5);
    \draw (c0)-- node[pos=0.45, left, xshift=-3pt] {$p$} (c1);
	\draw (c1)-- node[pos=0.7, below, yshift=-3pt] {$h$} (c6);
    \draw (c3)-- node[pos=0.6, above, xshift=3pt] {$i$} (c6);
	\draw (c6)-- node[pos=0.65, below, yshift=-3pt] {$i$} (c4);
	\draw (c1)-- node[pos=0.35, above, yshift=3pt] {$i$} (c5);
    \draw (d0)-- node[pos=0.45, left, xshift=-3pt] {$p$} (d1);
	\draw (d1)-- node[pos=0.35, above, yshift=3pt] {$i$} (d3);
	\draw (d6)-- node[pos=0.5, below, xshift=-3pt] {$i$} (d4);
	\draw (d6)-- node[pos=0.35, above, yshift=3pt] {$i$} (d5);
    \draw (d1)-- node[pos=0.5, below, xshift=3pt] {$h$} (d6);
    
\end{tikzpicture}
\end{Ex}

\subsection{Hochschild (co)homology and Tate-Hochschild (co)homology}

\subsubsection{Hochschild (co)homology for algebras}\

Hochschild introduced the cohomology theory of  associative algebras in \cite{Hoc45}. Given a $k$-algebra $A$, its Hochschild cohomology groups are defined as 
\[
\mathrm{HH}^n(A,A)\cong
\mathrm{Ext}^n_{A^e}(A, A),
\]
and its Hochschild homology groups are defined as
\[
\mathrm{HH}_n(A,A)\cong
\mathrm{Tor}_n^{A^e}(A, A),
\]
where $n\geq 0$ and $A^e=A \otimes A^{\mathrm{op}}$ is the enveloping algebra of $A$. There exists a projective resolution of $A$ as $A^e$-module, the so called  \textbf{normalized bar resolution} $\mathrm{Bar}_*(A)$  which is given by 
\[
\mathrm{Bar}_{-1}=A,
\]
\[
\mathrm{Bar}_n(A)=A\otimes \bar{A}^{\otimes n}\otimes A,\quad n\ge 1,
\]
where $\bar{A}=A/(k\cdot 1_A)$, that is,
\begin{eqnarray*}
\mathrm{Bar}_*(A) \colon \cdots \to A\otimes
\bar{A}^{\otimes n}\otimes A\xto{d_{n}} A\otimes
\bar{A}^{\otimes n-1}\otimes A\to \cdots \to A\otimes
\bar{A}\otimes A\xto{d_1}
A^{\otimes 2}\xto{d_0} A,
\end{eqnarray*}
where the map $d_0: A\otimes A\to A$ is the multiplication of $A$, and for $n\ge 1$,
\begin{eqnarray*}
d_n(a_0\otimes \overline{a_1}\otimes\cdots\otimes\overline{a_n}\otimes a_{n+1})&=& a_0a_1\otimes \overline{a_2}\otimes\cdots\otimes\overline{a_n}\otimes a_{n+1}\\
&+&\sum_{i=1}^{n-1}(-1)^{i} a_0 \otimes\overline{a_1}\otimes
\cdots\otimes \overline{a_{i-1}}\otimes \overline{a_ia_{i+1}}\otimes \overline{a_{i+2}}\otimes \cdots\otimes \overline{a_n}\otimes a_{n+1}\\
&+&(-1)^n a_0\otimes \overline{a_1}\otimes\cdots\otimes\overline{a_{n-1}}\otimes a_n a_{n+1}.
\end{eqnarray*}

For convenience, we write $\overline{a_{i,j}}:=\overline{a_i}\otimes\overline{a_{i+1}}\otimes\cdots\otimes\overline{a_j}\ (i\le j)$, when $n=0$, $\bar{A}^n:=k$.

The \textbf{Hochschild cochain complex} is $C^*(A,A)=\mathrm{Hom}_{A^e}({\mathrm{Bar}}_*(A), A)$. Note that
\[
C^n(A,A)=\mathrm{Hom}_{A^e}(A\otimes
{\bar{A}}^{\otimes n}\otimes A, A)\cong\mathrm{Hom}({\bar{A}}^{\otimes n}, A)
\]
for each $n\geq 1$. We also identify $C^0(A,A)$ with $A$. Thus, $C^*(A,A)$ has the following form:
\[
C^*(A,A) \colon A \xto{\delta^0} \mathrm{Hom}({\bar{A}}, A) \to \cdots \to \mathrm{Hom}({{\bar{A}}}^{\otimes n}, A)
\xto{\delta^n} \mathrm{Hom}({\bar{A}}^{\otimes (n+1)}, A)\to \cdots .
\]
It is not difficult to give the definition of $\delta^{*}$, $\delta^0: A\to \mathrm{Hom}(\bar{A},A)$ is defined as follows: 
\[
\delta^0(x)(\bar{a})=ax-xa,\ x\in A,\ \overline{a}\in \bar{A}.
\]
For any $f$ in $\mathrm{Hom}({\bar{A}}^{\otimes n}, A),\ n\ge 1$, the map $\delta^n(f)$ is defined by sending $\overline{a_{1,n+1}}$ to
\[
\delta^n(f)(\overline{a_{1,n+1}})= a_1 \cdot f(\overline{a_{2,n+1}})+\sum_{i=1}^n(-1)^{i}f(\overline{a_{1,i-1}}\otimes \overline{a_ia_{i+1}}\otimes \overline{a_{i+1,n+1}})+(-1)^{n+1}f(\overline{a_{1,n}})a_{n+1}.
\]

Recall that the \textbf{Hochschild chain complex} $(C_*(A,A),\partial_*)$ is defined as follows:
\[
C_{n}(A, A)=A \otimes _{A^{e} }\mathrm{Bar}_{n}(A) \simeq A \otimes \bar{A}^{\otimes n},  n \ge 0,
\]
and, for $n\ge 2$, the differential $\partial_n:A \otimes \bar{A}^{\otimes n}\to A \otimes \bar{A}^{\otimes n-1}$ sends $a_{0} \otimes \overline{a_{1, n}}$ to
\[
a_{0} a_{1} \otimes \overline{a_{2, n}}+\sum_{i=1}^{n-1}(-1)^{i} a_{0} \otimes \overline{a_{1, i-1}} \otimes \overline{a_{i} a_{i+1}} \otimes \overline{a_{i+2, n}}+(-1)^{n} a_{n} a_{0} \otimes \overline{a_{1, n-1}},
\]
and in degree $n=1$, the differential $\partial_1: A \otimes \bar{A} \rightarrow A$ is given by
\[
\partial_{1}\left(a_{0} \otimes \overline{a_{1}}\right)=a_{0} a_{1}-a_{1} a_{0}  \text {( for } a_{0} \in A \text { and } \overline{a_{1}} \in \bar{A} \text { ). }
\]

\subsubsection{Tate-Hochschild cohomology}\label{subsubsec: Tate-Hoch}\

According to \cite{Buch86}, in this section we first recall the definition of the $n$-th Tate–Hochschild cohomology group $\widehat{\mathrm{HH}}^{n}(A, A)$ of a self-injective algebra $A$.

\begin{Prop}\label{Prop:Tate-Hoch of self-injective}\cite[Corollary 6.4.1]{Buch86}
Let $A$ be a self-injective algebra. Denote $\mathrm{Hom}_{A^{e}}\left(A, A^{e}\right)$ by $A^{\vee}$. Then
\begin{itemize}
\item [(i)] $\widehat{\mathrm{HH}}^{n}(A, A) \simeq \mathrm{HH}^{n}(A, A)$ for all $n>0$,

\item [(ii)]  $\widehat{\mathrm{HH}}^{n}(A, A) \simeq \mathrm{HH}_{-n-1}\left(A^{\vee}, A\right)$ for all $n<-1$,

\item [(iii)] $\widehat{\mathrm{HH}}^{0}(A, A) \simeq \underline{\mathrm{Hom}}_{A^{e}}(A, A)$, $\widehat{\mathrm{HH}}^{-1}(A, A) \simeq \underline{\mathrm{Hom}}_{A^{e}}\left(A, \Omega_{A^{e}}(A)\right)$, and there is an exact sequence
\[
0 \rightarrow \widehat{\mathrm{HH}}^{-1}(A, A) \rightarrow A^{\vee} \otimes _{A^{e}} A \xrightarrow{\sigma} \mathrm{Hom}_{A^{e}}(A, A) \rightarrow \widehat{\mathrm{HH}}^{0}(A, A) \rightarrow 0,
\]
where $\sigma$ is given by $\sigma(f \otimes a)\left(a^{\prime}\right)=f\left(a^{\prime}\right) \cdot a$ for $a, a^{\prime} \in A$ and $f \in A^{\vee}$. Here  $\underline{\mathrm{Hom}}_{A^{e}}(-,-)$ denotes the homomorphism space in the stable category $A^{e}$-Mod and $\Omega_{A^{e}}$ is the syzygy functor over $A^{e}$-Mod.
\end{itemize}
\end{Prop}

Now we specialize $A$ to be a symmetric algebra.

\begin{Def}
    A finite dimensional $k$-algebra $A$ is a symmetric algebra if there is a symmetric nondegenerate associative bilinear form $\langle\cdot, \cdot\rangle: A \times A \rightarrow k$, or equivalently, $A \simeq A^{*}=\mathrm{Hom}_{k}(A, k)$ as $A$-$A$-bimodules.
\end{Def}

Note that we can choose an $A$-$A$-bimodule isomorphism (denote by $t$) as follows: $t(a)=\langle a, \cdot\rangle$ for $a\in A$. This isomorphism $t$ induces the following isomorphism
\[
\begin{array}{l} 
	t \otimes \mathrm{id}: A  \otimes_{k} A \rightarrow A^{*} \otimes_{k} A \simeq \operatorname{End}_{k}(A) \\
	\quad \quad \quad \quad  a \otimes b\ \mapsto t(a) \otimes b \ \mapsto(x \mapsto t(a)(x) b) .
\end{array}
\]

Following Broué (see \cite{Brou09}), we call the element $(t \otimes \mathrm{id})^{-1}(\mathrm{id}):=\sum_{i} e_{i} \otimes f_{i} \in A \otimes_{k} A$ the Casimir element of $A$. It follows from \cite[Proposition 3.3]{Brou09} that the Casimir element induces an $A$-$A$-bimodule isomorphism
\[
A \simeq A^{\vee}=\operatorname{Hom}_{A^{e}}\left(A, A^{e}\right), \quad a \mapsto \sum_{i} e_{i} a \otimes f_{i},
\]
where we identify $\mathrm{Hom}_{A^{e}}\left(A, A^{e}\right)$ as
\[
(A \otimes A)^{A}:=\left\{\sum_{i} a_{i} \otimes b_{i} \in A \otimes_{k} A | \sum_{i} a a_{i} \otimes b_{i}=\sum_{i} a_{i} \otimes b_{i} a \text { for any } a \in A \right\}. 
\]

According to Proposition~\ref{Prop:Tate-Hoch of self-injective}, we now give a description of the $n$-th Tate-Hochschild cohomology group $\widehat{\mathrm{HH}}^{n}(A, A)$ when $A$ is a symmetric algebra.

\begin{itemize}
    \item [(i)] $\widehat{\mathrm{HH}}^{n}(A, A) \simeq \mathrm{HH}^{n}(A, A)$ for all $n>0$,

    \item [(ii)] $ \widehat{\mathrm{HH}}^{n}(A, A) \simeq  \mathrm{HH}_{-n-1}(A, A)$ for all $n<-1$, 

    \item [(iii)] $\mathrm{HH}_{0}(A, A)=A /[A, A]$, $\mathrm{HH}^{0}(A, A)=Z(A)$, and there is an exact sequence
    \[
    0 \rightarrow \widehat{\mathrm{HH}}^{-1}(A, A) \rightarrow \mathrm{HH}_{0}(A, A) \xrightarrow{\tau}  \mathrm{HH}^{0}(A, A) \rightarrow \widehat{\mathrm{HH}}^{0}(A, A) \rightarrow 0,
    \]
    where the map $\tau$ is defined as follows:
    \[
    \tau: \mathrm{HH}_{0}(A, A)=A /[A, A] \rightarrow \mathrm{HH}^{0}(A, A)=Z(A),\quad  a+[A, A] \mapsto \sum_{i} e_{i} a f_{i}.
    \]
\end{itemize}

Therefore, $\widehat{\mathrm{HH}}^{*}(A, A)$ is a "combination" of the Hochschild cohomology $\mathrm{HH}^{*}(A, A)$ and the Hochschild homology  $\mathrm{HH}_{*}(A, A)$. We can summarize the above results by means of the following diagram:
\[
\xymatrix@R=1.2pc{
	& & & & \mathrm{HH}^0\ar@{->>}[d] &\mathrm{HH}^1\ar@<-0.5ex>@{=}[d] & \mathrm{HH}^2\ar@<-0.5ex>@{=}[d]& \cdots\\
	\cdots & \widehat{\mathrm{HH}}^{-3}\ar@<-1ex>@{=}[d]& \widehat{\mathrm{HH}}^{-2}\ar@<-1ex>@{=}[d] & \widehat{\mathrm{HH}}^{-1}\ar@<-0.5ex>@{_{(}->}[d] & \widehat{\mathrm{HH}}^{0} & \widehat{\mathrm{HH}}^{1} & \widehat{\mathrm{HH}}^{2} & \cdots \\
	\cdots & \mathrm{HH}_{2} & \mathrm{HH}_{1} &\mathrm{HH}_{0}\ar@{->}[uur]_{\tau} & & & &}
\]

In~\cite[Section 6.4]{Wang21}, the author constructed a complex (called \textbf{Tate-Hochschild cochain complex})
\[
\mathcal{D}^*(A, A):=\left(\cdots \xrightarrow{\partial_{2}} C_{1}(A, A) \xrightarrow{\partial_{1}} C_{0}(A, A) \xrightarrow{\tau} C^{0}(A, A) \xrightarrow{\delta^{0}} C^{1}(A, A) \xrightarrow{\delta^{1}} \cdots\right),
\]
to compute $\widehat{\mathrm{HH}}^{*}(A, A)$ for a symmetric algebra $A$, where $\partial_*$ (resp. $\delta^*$) is the differtial of $C_{*}(A, A)$ (resp. $C^{*}(A, A)$), and $\tau(x)=\sum_{i} e_{i} x f_{i}$. Here $\sum_{i}\left(e_{i} \otimes f_{i}\right)$ is Casimir element.

\subsection{Tate-Hochschild cohomology of a group algebra}\label{Section: Tate-Hochschild Cohomology of a Group Algebra}\

The content of this section is based entirely on \cite[Section 2]{LWZ21}; for further details, the reader is referred to that reference.

Let $k$ be a field, $G$ a finite group, and $kG$ the group algebra. Recall that $kG$ is a symmetric algebra with the symmetrizing form:
\[
\langle g, h\rangle=1 , \text{ if } gh=1 \text{ and } \langle g, h\rangle=0 \text{ otherwise}  
\]
for all $g,h\in G$. In particular, $\sum_{g \in G} g^{-1} \otimes g$ is a Casimir element of $kG$. Thus from Section~\ref{subsubsec: Tate-Hoch}, we have the Tate-Hochschild cohomology $\widehat{\mathrm{HH}}^{*}(k G, k G)$ is a "combination" of the Hochschild cohomology $\mathrm{HH}^{*}(k G, k G)$ and the Hochschild homology $\mathrm{HH}_{*}(k G, k G)$.

For convenience, we first introduce the following notation:

For a set $X$, we denote by $k[X]$ the $k$-vector space spanned by the elements in $X$. In particular, we have $kG = k[G]$. Note that $\overline{kG}$ can be identified with the $k$-vector space $k[\overline{G}]$, where $\overline{G} = G \setminus \{1\}$. When $n = 0$, the product $\overline{G}^{\times n}$ is understood as a one-point set, and we set $k[\overline{G}^{\times n}] := k$. For simplicity, we write $(g_1, g_2, \cdots, g_n) \in G^{\times n}$ as $(g_{1,n})$.

The normalized bar resolution $\left(\operatorname{Bar}_{*}(kG), d_*\right)$ of the group algebra $kG$ has the following form (here we write only the maps on the basis elements):
\[
\operatorname{Bar}_{-1}(kG) = kG, \quad \text{and for } n \geqslant 0, \quad \operatorname{Bar}_{n}(kG) = k\left[G \times \overline{G}^{\times n} \times G\right],
\]
\[
d_0\colon \operatorname{Bar}_0(kG) = k[G \times G] \rightarrow kG, \quad (g_0, g_1) \mapsto g_0 g_1,
\]
and for $n \geqslant 1$,
\[
\begin{aligned}
d_n\colon\ &\operatorname{Bar}_n(kG) \rightarrow \operatorname{Bar}_{n-1}(kG), \\
&(g_0, \overline{g_{1,n}}, g_{n+1}) \mapsto \sum_{i=0}^n (-1)^i (g_0, \overline{g_1}, \ldots, \overline{g_i g_{i+1}}, \ldots, \overline{g_n}, g_{n+1}).
\end{aligned}
\]

Here, $k\left[G \times \overline{G}^{\times n} \times G\right]$ denotes the $k$-vector space spanned by the elements of the Cartesian product $G \times \overline{G}^{\times n} \times G$. For convenience, we just write $g$ for its image $\bar{g}$ in $\overline{G}$.

\begin{Def}
   The \textbf{Hochschild cochain complex $\left(C^{*}(k G, k G), \delta^{*}\right)$} is defined as follows:
    \begin{itemize}
        \item [(i)] for $n\ge 0$, 
        \[
        C^{n}(k G, k G)=\mathrm{Hom}_{(k G)^{e}}\left(\mathrm{Bar}_{n}(k G), k G\right) \simeq \mathrm{Hom}_{k}\left(k\left[\overline{G}^{\times n}\right], k G\right) \simeq \mathrm{Map}\left(\overline{G}^{\times n}, k G\right),
        \]
        where $\mathrm{Map}\left(\overline{G}^{\times n}, k G\right)$ denotes the set of maps from $\overline{G}^{\times n}$ to $k G$, and 

        \item [(ii)] the differential is given by
        \[
        \delta^{n}: \mathrm{Map}\left(\overline{G}^{\times n}, k G\right) \rightarrow \mathrm{Map}\left(\overline{G}^{\times(n+1)}, k G\right),  \varphi \mapsto \delta^{n}(\varphi),
	  \]
	 where $\delta^{n}(\varphi)$ sends $g_{1, n+1} \in \overline{G}^{(n+1)}$ to
	 \[
	 g_{1} \varphi\left(g_{2, n+1}\right)+\sum_{i=1}^{n}(-1)^{i} \varphi\left(g_{1, i-1}, g_{i} g_{i+1}, g_{i+2, n+1}\right)+(-1)^{n+1} \varphi\left(g_{1, n}\right) g_{n+1} .
	 \]
	In degree $0$, the differential map $\delta^{0}: k G \rightarrow \mathrm{Map}(\overline{G}, k G)$ is given by
	\[
	\left.\delta^{0}(x)(g)=g x-x g \   (\text{for}\  x \in k G\   \text{and}\  g \in \overline{G}\right) .
	\]
    \end{itemize}
\end{Def}

\begin{Def}
The \textbf{Hochschild chain complex $\left(C_{*}(k G, k G), \partial_{*}\right)$} is defined as follows:
\begin{itemize}
    \item [(i)] for $n\ge 0$, 
    \[
    C_{n}(k G, k G)=k G \otimes_{(k G)^{e}} \mathrm{Bar}_{n}(k G) \simeq k\left[G \times \overline{G}^{\times n}\right],
    \]
    where $k\left[G \times \overline{G}^{\times n}\right]$ denotes the $k$-vector space spanned by the elements in $G \times \overline{G}^{\times n}$, and the differential is given by,

    \item [(ii)] for $n>1$, 
    \[
    \partial_{n}: k\left[G \times \overline{G}^{\times n}\right] \rightarrow k\left[G \times \overline{G}^{\times(n-1)}\right],
    \]
    \[
    \left(g_{0}, g_{1, n}\right) \mapsto\left(g_{0} g_{1}, g_{2, n}\right)+\sum_{i=1}^{n-1}(-1)^{i}\left(g_{0}, g_{1, i-1}, g_{i} g_{i+1}, g_{i+2, n}\right)+(-1)^{n}\left(g_{n} g_{0}, g_{1, n-1}\right).
    \]
    In degree $1$, the differential map $\partial_{1}: k[G \times \overline{G}] \rightarrow k G$ is given by
    \[
    \left.\partial_{1}\left(g_{0}, g_{1}\right)=g_{0} g_{1}-g_{1} g_{0}\    (\text{for}\ g_{0} \in G\  \text{and}\ g_{1} \in \overline{G}\right). 
    \]  
\end{itemize}
\end{Def}

From Section~\ref{subsubsec: Tate-Hoch}, the Tate-Hochschild cohomology $\widehat{\mathrm{HH}}^*(kG, kG)$ can be computed by the Tate-Hochschild complex $\mathcal{D}^{*}(k G, k G)$.

\begin{Def}
    The \textbf{Tate-Hochschild complex $(\mathcal{D}^{*}(k G, k G),d^*)$} is defined as follows:
    \[
	\cdots \xrightarrow{\partial_{2}} k[G \times \overline{G}] \xrightarrow{\partial_{1}} k G \xrightarrow{\tau} k G \xrightarrow{\delta^{0}} \mathrm{Map}(\overline{G}, k G) \xrightarrow{\delta^{1}} \cdots,
	\]
    where we have
    \begin{itemize}
        \item [(i)] $\mathcal{D}^{n}(k G, k G)=C^{n}(k G, k G)$ and $d^n=\delta^n$ for $n\ge 0$,

        \item [(ii)] $\mathcal{D}^{n}(k G, k G)=C_{-1-n}(k G, k G)$ for $n\le -1$ and $d^n=\partial_{-n-1}$ for $n\le -2$,

        \item [(iii)] $d^{-1}=\tau:  k G \rightarrow  kG$ (from degree $-1$ to degree $0$ component) is defined to be the trace map
        \[
	   \tau(x)=\sum_{g \in G} g x g^{-1}.
	  \]
    \end{itemize}
\end{Def}

From \cite{Wang21} and \cite{RW19}, there is an $A_{\infty}$-algebra structure $(m_1,m_2,m_3,\cdots)$ on $\mathcal{D}^*(kG,kG)$ with $m_2=\cup$ defined in \cite{Wang21}.

\begin{Def}
Let $\alpha \in \mathcal{D}^n(k G, k G)$ and $\beta \in \mathcal{D}^m(k G, k G)$. Then the (generalised) cup product $\alpha \cup \beta$ is defined by the following six cases:
	
\emph{Case 1}. $n \ge 0, m \ge 0$. Then $\alpha \in C^n(k G, k G), \beta \in C^m(k G, k G)$, and the cup product $\alpha \cup \beta \in C^{n+m}(k G, k G)=\mathcal{D}^{n+m}(k G, k G)$ is the same as the usual cup product on $C^{*}(k G, k G)$:
\[
\alpha \cup \beta: \overline{G}^{\times n+m} \rightarrow k G,\quad  g_{1, n+m} \mapsto \alpha\left(g_{1, n}\right) \beta\left(g_{n+1, n+m}\right).
\]

\emph{Case 2}. $n \le-1, m \le -1$.  Then $\alpha=(g_{0}, g_{1, s}) \in C_s(k G, k G)$ with $s=-n-1 \ge 0$, $\beta= (h_{0}, h_{1, t}) \in C_{t}(k G, k G)$ with $t=-m-1\ge 0$, and the cup product $\alpha \cup \beta \in C_{s+t+1}(k G, k G)=\mathcal{D}^{n+m}(k G, k G)$ is defined by
\[
\alpha \cup \beta=\sum_{g \in G}\left(g h_{0}, h_{1, t}, g^{-1} g_{0}, g_{1, s}\right) \in k\left[G \times \overline{G}^{\times s+t+1}\right].
\]
	
\emph{Case 3}. $n \ge 0, m \le -1$ and $n+m \le -1$. Then $\alpha \in C^n(k G, k G)$, $\beta=(h_{0}, h_{1, t}) \in C_t(k G, k G)$ with $t=-m-1\ge 0$, and the cup product $\alpha \cup \beta \in C_{t-n}(k G, k G)=\mathcal{D}^{n+m}(kG,kG)$ is the same as the usual cap product $\cap$ (which induces an action of Hochschild cohomology on Hochschild homology):
\[
\alpha \cup \beta=\left(\alpha\left(h_{t-n+1, t}\right) h_{0}, h_{1, t-n}\right) \in k\left[G \times \overline{G}^{\times t-n}\right].
\]
 
\emph{Case 4}. $n \ge 0, m \le -1$ and $n+m \ge 0$. Then $\alpha \in C^n(kG, kG)$, $\beta=(g_0, g_{1, t}) \in C_t(kG, kG)$ with $t=-m-1\ge 0$, and the cup product $\alpha \cup \beta \in C^{n-t-1}(kG, kG)=\mathcal{D}^{n+m}(kG, kG)$ is defined as the following generalized cap product:
\[
\alpha \cup \beta: \overline{G}^{\times n-t-1} \rightarrow kG,\quad  h_{1, n-t-1} \mapsto \sum_{g \in G} \alpha\left(h_{1, n-t-1}, g^{-1}, g_{1, t}\right) g_{0} g.
\]
 
\emph{Case 5}. $n \le -1, m \ge 0$ and $n+m \le -1$. Then $\alpha=(g_{0}, g_{1, s}) \in C_s(kG, kG)$ with $s=-n-1\ge 0$, $\beta \in C^m(kG, kG)$, and the cup product $\alpha \cup \beta \in C_{s-m}(k G, k G)=\mathcal{D}^{n+m}(k G, k G)$ is the following cap product $\cap$ from the right side:
\[
\alpha \cup \beta=\left(g_{0} \beta\left(g_{1, m}\right), g_{m+1, s}\right) \in k\left[G \times \overline{G}^{\times s-m}\right].
\]

\emph{Case 6}. $n \le -1, m \ge 0$ and $n+m \ge 0$. Then $\alpha=(g_{0}, g_{1, s}) \in C_{s}(kG, kG)$ with $s=-n-1\ge 0$, $ \beta \in C^m(kG, kG)$, and the cup product $\alpha \cup \beta \in C^{m-s-1}(kG, kG)=\mathcal{D}^{n+m}(kG, kG)$ is defined as the following generalized cap product from the right side:
\[
\alpha \cup \beta: \overline{G}^{\times m-s-1} \rightarrow k G,\quad  h_{1, m-s-1} \mapsto \sum_{g \in G} g g_{0} \beta\left(g_{1, s}, g^{-1}, h_{1, m-s-1}\right).
\]
\end{Def}

\begin{Rmk}
    Since the sign convention for the cup product $\cup$ used in this paper differs from that in \cite{Wang21}, in order to make the following identity still hold in $\mathcal{D}^{*}(k G, k G)$,
    \[
	\partial(\alpha \cup \beta)=\partial(\alpha) \cup \beta+(-1)^{m} \alpha \cup \partial(\beta),\   \text{for}\ \alpha \in \mathcal{D}^m(kG, kG)\  \text{and}\  \beta \in \mathcal{D}^n(kG, kG),
	\]
	we have to change the signs of the differential in the negative part $\mathcal{D}^{<0}(kG, kG)$. That is, the new differential $\partial'$ on $\mathcal{D}^*(kG, kG)$ is given as follows:
    \[
    \partial'_{m}(\alpha)=\left\{\begin{array}{ll}
 	(-1)^{m+1} \partial_{-m-1}(\alpha) & \text{for}\ \alpha \in \mathcal{D}^m(kG, kG)\ \text{and}\ m<-1,\\
 	\tau(\alpha)& \text{for}\ \alpha \in \mathcal{D}^{-1}(k G, k G), \\
 	\delta^m(\alpha) & \text{for}\ \alpha \in \mathcal{D}^m(kG, kG)\ \text{and}\ m\ge 0.
 \end{array}\right.
    \]
\end{Rmk}

From \cite[Theorem 6.3]{RW19}, it follows that the cup product extends to an $A_{\infty}$-algebra structure $(m_1,m_2,m_3,\cdots)$ on $\mathcal{D}^*(kG,kG)$ with $m_1=\partial'$, $m_2=\cup$ amd $m_i=0$ for $i>3$. The formula for $m_3$ is described as follows:
\begin{itemize}
    \item [(i)] If either $\phi, \varphi, \psi \in C^*(kG,kG)$ or $\phi, \varphi, \psi \in C_*(kG,kG)$, then $m_3(\phi, \varphi, \psi)=0$.

    \item [(ii)] If $\alpha, \beta \in C_*(kG,kG)$ and $\phi \in C^*(kG,kG)$, then $m_3(\alpha, \beta, \phi)=0=m_3(\phi, \alpha, \beta)$.

    \item [(iii)] If  $\alpha \in C_*(kG,kG)$ and $\phi, \varphi \in C^*(kG,kG)$, then $m_3(\phi, \varphi, \alpha)=0=m_3(\alpha, \phi, \varphi)$.

    \item [(iv)] For $\phi \in C^m(kG,kG), \varphi \in C^n(kG,kG)$ and $\alpha=(g_{0}, g_{1},  \cdots, g_r) \in C_r(kG,kG)$,
    \begin{itemize}
        \item [$\bullet$] if $r+2\le m+n$, then $m_3(\phi, \alpha, \varphi) \in C^{m-r+n-2}(kG,kG)$ is defined by
         \begin{align*}
  		m_3(\phi, \alpha, \varphi)(h_{1}, \cdots, h_{m-r+n-2})=\sum_{g \in G} \sum_{j=\max\{1,r+2-m\}}^{\min \{n, r+1\}}(-1)^{m+r+j-1} \quad\quad\quad\quad\quad\quad\quad\quad\quad\quad\quad \\
  	\phi(h_{1, m-r+j-2}, g, g_{j, r}) g_{0} \varphi(g_{1, j-1}, g^{-1}, h_{m-r+j-1, m-r+n-2}). 
      \end{align*}

      \item [$\bullet$] if $r+2>m+n$, then $m_{3}(\phi, \alpha, \varphi)=0$.
    \end{itemize}

    \item [(v)] For $\alpha=(g_0, g_{1, r}) \in C_r(kG, kG), \beta=(h_0, h_{1, s}) \in C_s(kG, kG)$ and $\phi \in C^m(kG, kG)$, 
    \begin{itemize}
        \item [$\bullet$] if $m-1\le r+s$, 
        \begin{align*}
		m_3(\alpha, \phi, \beta)=\sum_{g \in G} \sum_{j=\max\{0,s+1-m\}}^{\min \{s,  r-m+s+1\}}(-1)^{m+r+s-j}\quad\quad\quad\quad\quad\quad\quad\\
        (g_{0}\phi(g_{1, m-s+j-1}, g, h_{j+1, s}) h_{0}, h_{1, j}, g^{-1}, g_{m-s+j, r}),
	\end{align*}

    \emph{Note}: In the original text of \cite{LWZ21}, it is written as
\begin{align*}
	&m_3(\alpha, \phi, \beta) = \sum_{g \in G} \sum_{j=0}^{s} (-1)^{n-j} \\ 
    &\quad\quad\quad\quad\quad\quad(g_{0} \phi(g_{1, m-s+j-1}, g, h_{j+1, s}) h_{0}, h_{1, j}, g^{-1}, g_{m-s+j, r}).
\end{align*}
However, in actual computations, we found a sign error in the exponent of $(-1)$. In all computations in this paper, it has been corrected to $(-1)^{m + r + s - j}$.

\item [$\bullet$] if $m - 1 > r + s$, then $m_3(\alpha, \phi, \beta) = 0$.
    \end{itemize}
\end{itemize}

\subsection{Reminder on cohomology and Tate cohomology of finite groups}\

In this section, we recall some notions on Tate cohomology of finite groups. For the details, we refer the reader to \cite[Chapter VI]{Brown82}.

Let $k$ be a field, $G$ a finite group, and $kG$ the group algebra. Let $M$ be a left $kG$-module. Then the cohomology of $G$ with coefficients in $M$ is defined to be
\[
 \mathrm{H}^p(G, M):=\mathrm{Ext}_{kG}^p(k, M), p \ge 0, 
\]
and the homology of $G$ with coefficients in $M$ is defined to be
\[
 \mathrm{H}_p(G, M)=\mathrm{Tor}_p^{kG}(k, M),  p \ge 0,
 \]
where $k$ is the left trivial $kG$-module in $\mathrm{Ext}_{kG}^p(k, M)$ and is the right trivial $kG$-module in $\mathrm{Tor}_p^{kG}(k, M)$. Note that the complex $P_*:=\mathrm{Bar}_*(kG) \otimes_{kG} k$ is the \emph{standard resolution} of the trivial $kG$-module $k$. So there exist canonical complexes computing group (co)homology.

Recall that the \textbf{group cohomology complex $(C^*(G,  M),  \delta^*)$} is defined as follows:
\[
C^n(G, M)=\mathrm{Hom}_{kG}(\mathrm{Bar}_n(kG) \otimes_{kG} k, M) \simeq \mathrm{Hom}_{kG}(k\left[\overline{G}^{\times n}\right], M) \simeq \mathrm{Map}(\overline{G}^{\times n}, M G),\ for\ n\ge 0,
\]
and the differential is given by
 \[
 \delta^{n}: \mathrm{Map}\left(\overline{G}^{\times n},  M\right) \rightarrow \mathrm{Map}\left(\overline{G}^{\times(n+1)},  M\right),\quad   \varphi \mapsto \delta^{n}(\varphi), 
 \]
 where $\delta^n(\varphi)$ sends $g_{1,n+1}\in \overline{G}^{n+1}$ to:
 \[
 g_{1} \varphi\left(g_{2,  n+1}\right)+\sum_{i=1}^{n}(-1)^{i} \varphi\left(g_{1, i-1}, g_{i} g_{i+1}, g_{i+2, n+1}\right)+(-1)^{n+1} \varphi\left(g_{1,  n}\right).
 \]
 In degree $0$, the differential map $\delta^0: M \rightarrow \mathrm{Map}(\overline{G}, M)$ is given by:
 \[
 \delta^{0}(x)(g)=g x-x \quad (\text {for}\ x \in M \text { and}\ g \in \overline{G}) .
 \]

 We can consider $M$ as a right $kG$-module via  $x \cdot g=g^{-1} x,  x \in M,  g \in G$. Then $\mathrm{Tor}_*^{k G}(k, M) \cong \mathrm{Tor}_*^{k G}(M, k)$, where we use the right $kG$-module $M$ in $\mathrm{Tor}_*^{k G}(M, k)$. Notice that $\mathrm{Tor}_*^{k G}(M, k)$ can be computed by the \textbf{group homology complex $(C_*(G,  M), \partial_*)$}, which is defined as follows:
 \[
 C_n(G, M)=M \otimes_{k G} \mathrm{Bar}_n(k G) \otimes_{k G} k \simeq M \otimes k\left[\overline{G}^{\times n}\right], \quad \text {for } n \ge 0,
 \]
and the differential $\partial_n: M \otimes k\left[\overline{G}^{\times n}\right] \rightarrow M \otimes k\left[\overline{G}^{\times(n-1)}\right],  n \ge 2$ is given by
 \[
 x \otimes g_{1, n} \mapsto x \cdot g_1 \otimes\left(g_{2, n}\right)+\sum_{i=1}^{n-1}(-1)^i x \otimes\left(g_{1, i-1}, g_i g_{i+1}, g_{i+2, n}\right)+(-1)^n x \otimes\left(g_{1, n-1}\right)
 \]
 and in degree $1$, the differential map $\partial_1: M \otimes k[\overline{G}] \rightarrow M$ is given by
 \[
 \partial_{1}\left(x \otimes g_{1}\right)=x \cdot g_{1}-x\quad \left(\text {for } x \in M \text { and } g_{1} \in \overline{G}\right)
 \]

 Let $U$ be any left $kG$-module, we recall the definition in \cite{LWZ21} of the \textbf{Tate cochain complex $(\widehat{C}^*(G, U), \delta'_n)$} of finite group $G$
\begin{itemize}
    \item [(i)] $\widehat{C}^{\ge 0}(G, U)=C^*(G, U),\quad \delta'_{\ge 0}=\delta^n$.

    \item [(ii)] For each $n\le -1$ (let $s=-n-1\ge 0$), 
    \[
    \widehat{C}^{n<0}(G, U)=C_*(G, U)
    \]
    and the differential is given by $\delta'_n=\partial_{s}$ for all $s\ge 1$.

    \item [(iii)] For $n=1$ (or $s=0$), the differential $\delta'_{-1}: C_0(G, U)=U \rightarrow U=C^0(G, U)$ is given by $u\mapsto (\sum_{g \in G} g) u$ for $u \in U$.
\end{itemize}

\subsection{Additive decomposition of the Tate-Hochschild cohomology at the complex level}\label{section:additive-decomposition-TT}\

In \cite{LWZ21}, the authors constructed an additive decomposition at the level of the Tate-Hochschild cochain complex $\mathcal{D}^*(kG, kG)$ of the group algebra $kG$. In the following, we briefly review this construction. The main result of our paper is to obtain a new $A_{\infty}$-algebra structure under this additive decomposition framework.

\emph{Let $X$ be a set of representatives of the conjugacy classes of elements in $G$. For each $x \in X$, define $C_{x} = \left\{g x g^{-1} \mid g \in G\right\}$ as the conjugacy class of $x$, and define the centralizer subgroup as $C_G(x) = \left\{g \in G \mid g x g^{-1} = x\right\}$.}

Fix a decomposition of $G$ into right cosets of $C_G(x)$:
\[
G = C_G(x)\gamma_{1, x} \cup C_G(x)\gamma_{2, x} \cup \cdots \cup C_G(x)\gamma_{n_x, x},
\]
where $n_x$ is the number of elements in the conjugacy class $C_x$. Then the conjugacy class $C_x$ can be written as:
\[
C_x = \left\{ \gamma_{1, x}^{-1} x \gamma_{1, x}, \cdots, \gamma_{n_x, x}^{-1} x \gamma_{n_x, x} \right\}.
\]
We denote $x_i = \gamma_{i, x}^{-1} x \gamma_{i, x}$ and, without loss of generality, let $\gamma_{1, x} = 1$, so $x_1 = x$.

Define:
\[
\mathcal{H}^{x,0} = k[C_x],\text{ and for } n\ge 1,
\]
\[
\mathcal{H}^{x,n} = \left\{ \varphi: \overline{G}^{\times n} \longrightarrow kG \ \middle|\ \varphi(g_1, \dots, g_n) \in k[g_1 \cdots g_n C_x] \subset kG, \ \forall g_1, \dots, g_n \in \overline{G} \right\},
\]
where $g_1 \cdots g_n C_x$ denotes the subset of $G$ obtained by multiplying $g_1 \cdots g_n$ on $C_x$, and $k[g_1 \cdots g_n C_x]$ is the $k$-subspace of $kG$ spanned by this set.

Note that we have the equality $g_1 \cdots g_n C_x = C_x g_1 \cdots g_n$, and hence $k[g_1 \cdots g_n C_x] = k[C_x g_1 \cdots g_n]$. Define $\mathcal{H}^{x,*} = \bigoplus_{n \geq 0} \mathcal{H}^{x,n}$, which forms a subcomplex of $C^*(kG, kG)$, and we have:
\[
C^*(kG, kG) = \bigoplus_{x \in X} \mathcal{H}^{x,*}.
\]

We now recall the definition of the cochain complex $(C^*(C_G(x), k), \delta^n)$:
\[
\widehat{C}^*(C_G(x), k) =
\left(
\cdots \xrightarrow{\partial_{2}} C_{1}(C_G(x), k)
\xrightarrow{\partial_{1}} C_{0}(C_G(x), k)
\xrightarrow{\tau} C^{0}(C_G(x), k)
\xrightarrow{\delta^{0}} C^{1}(C_G(x), k)
\xrightarrow{\delta^{1}} \cdots
\right).
\]
\begin{itemize}
    \item [(i)] $\widehat{C}^{n\ge 0}(C_G(x), k)=C^n(C_G(x), k)= \mathrm{Map}(	\overline{C_G(x)}^{\times n}, k),\ \delta_{\ge 0}=\delta^n$, specificly, the differential is given by 
    \[
    \delta^0(\lambda)(g_1)=0, \text{ for } \lambda\in k \text{ and } g_1\in C_G(x),
    \]
    and for $n\ge 1$, $\delta^n(\varphi)$ map $g_{1, n+1} \in \overline{C_{G}(x)}^{\times {n+1}} $ to
    \[
    \varphi\left(g_{2,  n+1}\right)+\sum_{i=1}^{n}(-1)^{i} \varphi\left(g_{1,  i-1},  g_{i} g_{i+1},  g_{i+2,  n+1}\right)+(-1)^{n+1} \varphi\left(g_{1,  n}\right).
    \]

    \item [(ii)] For each $n\le -1$ (let $s=-n-1\ge 0$), $\widehat{C}^n(C_G(x), k)=C_s(C_G(x), k)= k\left[\overline{C_{G}(x)} ^{\times s}\right]$, and the differential $\delta_n=\partial_s:k\left[\overline{C_{G}(x)} ^{\times s}\right] \rightarrow k\left[\overline{C_{G}(x)} ^{\times (s-1)}\right]$ is given by 
    \[
    g_{1, s} \mapsto (g_{2, s})+\sum_{i=1}^{s-1}(-1)^{i} (g_{1,  i-1},  g_{i} g_{i+1},  g_{i+2,  s})+(-1)^{s} (g_{1, s-1})
    \]
    for all $s\ge 2$, and for $s=1$, the differential $\partial_1: k\left[\overline{C_G(x)}\right] \rightarrow k$ is defined by $\partial_1( g_1)=0$.

    \item [(iii)] For $n=-1$ (or $s=0$), the differential $\delta_{-1}=\tau: C_0(C_G(x), k)=k \rightarrow k=C^0(C_G(x), k)$ is defined by $\tau(1) =|C_G(x)|$.
\end{itemize}

In \cite{LZ16,LWZ21}, a lifting of the additive decomposition of the Hochschild cohomology of the group algebra $kG$ at the complex level was established.

\begin{Lem}\cite[Theorem 6.3]{LZ16}\cite[Theorem 4.3]{LWZ21}\label{Lemma: add Hochschild cohomology} 
Let $k$ be a field and $G$ a finite group. Consider the additive decomposition of Hochschild cohomology algebra of the group algebra $kG$:
\[ 
\mathrm{HH}^{*}(k G,  k G) \simeq \bigoplus_{x \in X} \mathrm{H}^{*}\left(C_{G}(x),  k\right) 
\]
The additive decomposition can lift to a homotopy deformation retract of complexes
\begin{equation*}\xymatrix@C=0.000000001pc{
		\ar@(lu, dl)_-{s^*}&C^*(kG,  kG) \ar@<0.5ex>[rrrrr]^-{\iota^*}&&&&&\bigoplus\limits_{x\in X}  C^*(C_G(x),  k). \ar@<0.5ex>[lllll]^-{\rho^*}}
\end{equation*}
where $\iota^n=\sum_{x\in X} \iota^{x,n}, \ \rho^n=\sum_{x\in X} \rho^{x,n}$, and $s^n=\sum_{x\in X}s^{x,n}$, for $n\ge 0$.

For $n=0$, for any $\alpha_x \in k[C_x]$, the maps $\iota^{x,0}$, $\rho^{x,0}$, $s^{x,0}$ are defined as follows:
\begin{align*}
	&	\iota^{x,0}: \mathcal{H}^{x,0}=k[C_x] \rightarrow k,  
	\quad\alpha_x=\sum_{i=1}^{n_x}\lambda_i x_i \mapsto \lambda_1, \\
	&	\rho^{x,0}:  k  \rightarrow\mathcal{H}^{x,0}=k[C_x], \quad  \lambda \mapsto \alpha_x=\sum_{i=1}^{n_x}\lambda x_i, \\
	&s^{x,0}(\alpha_x)=0
\end{align*}

For $n=1$, the homotopy $s^{x,1}$ is defined by: for $(\varphi_x: \overline{G}  \rightarrow k[gC_x])\in \mathcal H^{x,1}$, define $s^{x,1}(\varphi_x) \in  k[C_x]$ as
\[
s^{x,1}(\varphi_x)= \sum_{i=1}^{n_x} a^1_i x_i,
\]
where $a^1_i$ determined by $\varphi_{x}(\gamma _{i, x})=\sum_{k=1}^{n_x}a^k_i x_k \gamma _{i, x}$.

For $n\ge 1$, the map $\iota^{x,n}$ is given by
\begin{equation*}
	\begin{split}
		\iota^{x,n}: \mathcal{H}^{x,n} \rightarrow C^n(C_G(x),  k),  \quad
		[\varphi_x: \overline{G}^{\times n}\rightarrow kG]\mapsto [\widehat{\varphi}_x:
		\overline{C_G(x)}^{\times n}\rightarrow k],
	\end{split}
\end{equation*}
\[
\text{with }\widehat{\varphi}_x(h_{1, n})=a_{1, x} \text{ where }\varphi_x(h_{1, n})h_n^{-1}\cdots
h_1^{-1}=\sum_{i=1}^{n_x}a_{i, x}x_i\in kC_x.
\]
In other words, $\widehat{\varphi}_x(h_{1,n})$ is just the coefficient of $x$ in $\varphi_x(h_{1, n})h_{n}^{-1}\cdots h_1^{-1}\in kC_x$. The map $\rho^{x,n}$ is given by
\begin{equation*}
	\begin{split}
		\rho^{x,n}:  C^n(C_G(x),  k)&\rightarrow \mathcal H^{x,n} , \quad
		[\widehat{\varphi}_x:
		\overline{C_G(x)}^{\times n}\rightarrow k]\mapsto
		[\varphi_x: \overline{G}^{\times n}\rightarrow kG], \end{split}
\end{equation*}
\[
\text{with }\varphi_x\in\mathcal{H}^{x,n} , \text{ and }\varphi_x(g_{1,n})=\sum_{i=1}^{n_x}\widehat{\varphi}_x(h_{i,1}^x, \cdots, h_{i,n}^x)x_ig_1\cdots g_n,
\]
where $\{h_{i, 1}^x, \cdots,  h_{i, n}^x\}\in \overline{C_G(x)}$ are determined by the sequence $\{g_1, \cdots, g_n\}$ as follows:
\[
\gamma_{i, x}g_1=h_{i,1}^x\gamma_{s_i^1, x},  \quad \gamma_{s_i^1, x}g_2=h_{i,2}^x\gamma_{s_i^2, x},  \quad \cdots,  \quad
\gamma_{s_i^{n-1}, x}g_n=h_{i,n}^x\gamma_{s_i^n, x}.
\]
\emph{We call this process as $\spadesuit$. In this process, the sequence $\{h_{i, 1}^x, \cdots,  h_{i, n}^x\}\in \overline{C_G(x)}^{\times n}$ is determined by the sequence $\{g_1, \cdots,  g_n\}$, $x$ and $i$. We write $\{h_{i, 1}^x, \cdots,  h_{i, n}^x\}=\spadesuit_{x,i}\{g_1, \cdots,  g_n\}$.}

For $n\ge 2$, the homotopy $s^{x,n}$ is given by: for $(\varphi_x: \overline{G}^{\times n} \rightarrow kG)\in \mathcal{H}^{x,n}$, we define $s^{x,n}(\varphi_x) \in \mathcal{H}^{x,n-1}$ as
\[
s^{x,n}(\varphi_x)(g_{1,  n-1})=\sum_{j=0}^{n-1}\sum_{i=1}^{n_x} (-1)^ja^{1}_{i,  j} x_i g_1\cdots g_{n-1}, 
\]
where the coefficients $a_{i,j}^1$ are determined by the following identity (when $j=0$, we set $\gamma_{s_i^0, x}=\gamma_{i, x}$)
\[
\varphi_x(h_{i,1}^x,  \cdots,  h_{i,j}^x,  \gamma_{s_i^j,  x},  g_{j+1},  \cdots, g_{n-1})g_{n-1}^{-1}\cdots g_1^{-1}\gamma_{i, x}^{-1}=\sum_{k=1}^{n_x} a_{i, j}^k x_k,
\]
since we have $h_{i,1}^x h_{i,2}^x \cdots h_{i,j}^x \gamma_{s_i^j,  x}g_{j+1}\cdots g_{n-1}=\gamma_{i,  x}g_1\cdots g_{n-1}$ for any $0\le j\le n-1$.
\end{Lem}

We now discuss the additive decomposition at the level of the Hochschild homology of the group algebra $kG$. We begin by introducing the following notation
\[
\mathcal{H}_{x, 0}=k\left[C_{x}\right], \text{ and for } s\ge 1,
\]  
\[
\mathcal{H}_{x, s}=k\left[\left(g_s^{-1} \cdots g_1^{-1} u, g_{1, s}\right) |\ u \in C_x, g_1, \cdots, g_s \in \overline{G}\right].
\]
Let $\mathcal{H}_{x, *}=\bigoplus_{s \ge 0} \mathcal{H}_{x, s}$. It is easy to verify that $\mathcal{H}_{x, *}$ is a subcomplex of $C_*(k G, k G)$ and $C_*(k G, k G)=\bigoplus_{x\in X} \mathcal{H}_{x, *}$.

\begin{Lem}\cite[Theorem 4.6]{LWZ21}\label{Lemma: add Hochschild homology}
$\mathrm{H}_*(C_G(x), k)$ is the group homology of $C_G(x)$. Consider the additive decomposition of Hochschild homology of the group algebra $kG$:
\[ 
\mathrm{HH}_*(k G,  k G) \simeq \bigoplus_{x \in X} \mathrm{H}_*\left(C_{G}(x), k\right),
\]
this additive decomposition can lift to a homotopy deformation retract of complexes
\begin{equation*}
\xymatrix@C=0.000000001pc{
		\ar@(lu, dl)_-{s_*}&C_*(kG,  kG)=\bigoplus_{x \in X} \mathcal{H}_{x,  *} \ar@<0.5ex>[rrrrr]^-{\rho_*}&&&&&\bigoplus\limits_{x\in X}  C_*(C_G(x),  k), \ar@<0.5ex>[lllll]^-{\iota_*}}
\end{equation*}
where $\iota_n=\sum_{x\in X} \iota_{x,n}, \ \rho_n=\sum_{x\in X} \rho_{x,n}$ and $s_n=\sum_{x\in X}s_{x,n}$, for $n\ge 0$.

For $n=0$, $x=\sum_{i=1}^{n_x}\lambda_i x_i \in k[C_x],\ \lambda\in k$, the injection $\iota_{x,0}$, surjection $\rho_{x,0}$ and homotopy $s_{x,0}$ are given respectively as follows:
\[
\iota_{x,0}(\lambda)=\lambda x_1,
\]
\[
\rho_{x,0}(\sum_{i=1}^{n_x}\lambda_i x_i)=\sum_{i=1}^{n_x}\lambda_i , 
\]
\[
s_{x,0}(\sum_{i=1}^{n_x}\lambda_i x_i)=\sum_{i=1}^{n_x}\lambda_i (\gamma^{-1}_{i, x}x, \gamma_{i, x}).
\]

For $n\ge 1$, the injection $\iota_{x,n}$ is given by
\[
\iota_{x,n}: C_n\left(C_G(x),  k\right)  \stackrel{\sim}{\longrightarrow} \mathcal{H}_{x,  n},\quad\quad\quad\quad\quad\quad 
\]
\[
{\left[\widehat{\alpha}_{x}=\left(h_{1},  \cdots,  h_{n}\right) \in k\left[{\overline{C_{G}(x)}}^{\times n}\right]\right.}  \longmapsto\left[\alpha_{x}=\left(h_{n}^{-1} \cdots h_{1}^{-1}x,  h_{1,  n}\right) \in \mathcal{H}_{x,  n}\right],
\]
and the surjection $\rho_{x,n}$ is given by
\[
\ \quad\quad\rho_{x,n}: \mathcal{H}_{x, n} \longrightarrow C_n\left(C_G(x), k\right), 
\]
\[
{\left[\alpha_x=\left(g_n^{-1} \cdots g_1^{-1} g_0^{-1} x g_0, g_{1, n}\right) \in \mathcal{H}_{x, n}\right] }  \longmapsto\left[\widehat{\alpha}_x=\left(h_{i,1}^x, \cdots, h_{i,n}^x\right) \in k[\overline{C_G(x)} ^{\times n}]\right],
\]
where $h_{i,1}^x, \cdots, h_{i,n}^x\in \overline{C_G(x)}$ are determined by the following sequence:
\[
g_0=h \gamma_{i, x},\quad \gamma_{i, x} g_1=h_{i,1}^x \gamma_{s_i^1, x},\quad  \gamma_{s_i^1, x} g_2=h_{i,2}^x \gamma_{s_i^2, x},\quad  \cdots,\quad   \gamma_{s_i^{n-1}, x} g_n=h_{i,n}^x \gamma_{s_i^n, x}.
\]
The homotopy $s_{x,n}$ (for $n\ge 1$) is given as follows: for $\alpha_x=\left(g_n^{-1} \cdots g_1^{-1} g_0^{-1} x g_0, g_{1, n}\right) \in \mathcal{H}_{x, n}$, 
\begin{align*}
s_{x,n}\left(\alpha_x\right)&=\sum_{j=0}^n(-1)^j\left(g_n^{-1} \cdots g_1^{-1} g_0^{-1} x h,  h_{i,1}^x, \cdots, h_{i,j}^x, \gamma_{s_i^j, x}, g_{j+1}, \cdots, g_n\right) \\
&=\sum_{j=0}^n(-1)^j\left( (h_{i,1}^x \cdots h_{i,j}^x \gamma_{s_i^j,x} g_{j+1} \cdots g_n)^{-1}x,  h_{i,1}^x, \cdots, h_{i,j}^x, \gamma_{s_i^j, x}, g_{j+1}, \cdots, g_n\right) \in \mathcal{H}_{x, n+1},
\end{align*}
when $j=0$, we set $\gamma_{s_i^0, x}=\gamma_{i, x}$.
\end{Lem}

We obtain the following Theorem from Lemma~\ref{Lemma: add Hochschild cohomology} and Lemma~\ref{Lemma: add Hochschild homology}.

\begin{Thm}\label{Thm: add tate Hochschild}
The Tate-Hochschild cochain complex $\mathcal{D}^*(kG, kG)$ of the group algebra $kG$ admits an additive decomposition as follows:
 \begin{equation*}
 \xymatrix@C=0.000000001pc{
 \ar@(lu, dl)_-{\hat{s}^*}& \mathcal{D}^*(k G,  k G)\ar@<0.5ex>[rrrrr]^-{\hat{\rho}^*}&&&&& \underset{x \in X}{\bigoplus} \widehat{C}^*\left(C_G(x),  k\right) .\ar@<0.5ex>[lllll]^-{\hat{\iota}^*}}
 \end{equation*}
where, for $m \ge 0$, we have  
 \[
 \hat{\iota} ^m=\rho^m,\  \hat{\iota} ^{-m-1}=\iota_m;\quad  
 {\hat{\rho}} ^m=\iota^m,\  \hat{\rho} ^{-m-1}=\rho_m ; \quad
 \hat{s} ^m=s^m,\  \hat{s} ^{-m-1}=s_m.
 \]
\end{Thm}

\section{\texorpdfstring{$A_{\infty}$-structure}{A-infinity structure}}\label{section: Ainfty structure}

\subsection{\texorpdfstring{$\widehat{m}_n$ Algorithm}{mn Algorithm}}\label{Section:mn-algorithm}\

Due to the complexity of the graphs corresponding to $PT_n$, in this subsection, we provide a decomposition method for the $PT_n$ graphs. We also explain how to compute the multiplication associated to each $PT_n$ graph via local decompositions, leading to the computation process of $\widehat{m}_n$ under the additive decomposition of the Tate-Hochschild cochain complex of the group algebra.

Specifically, since the higher operations $m_i = 0$ for all $i \ge 4$ on the Tate-Hochschild cochain complex $\mathcal{D}^*(kG, kG)$, the computation of $\widehat{m}_n$ on $\bigoplus_{x \in X} \widehat{C}^*(C_G(x), k)$ only involves analyzing local cases in the $PT_n$ graph with either two or three branches.

We begin by discussing the local graphs with exactly two branches.

\subsubsection{Algorithms on locally two-branched graphs}\

(1) We begin by referring to the classification of planar binary trees introduced in \cite{Li20}, that is, starting from the terminal end, planar binary trees can be classified into two types according to their terminal morphism: $\hat{\rho}^*$ and $\hat{s}^*$. We denote these by $\alpha$-type and $\beta$-type trees, respectively. In particular, the associated branching graphs can be further categorized into the following two distinct forms:
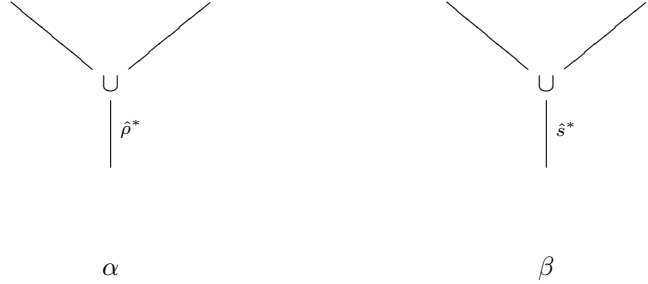
\[
\xymatrix{
	\ar@{-}[rd] & \quad & \ar@{-}[ld] & \quad & \ar@{-}[rd] & \quad & \ar@{-}[ld] \\
	\quad & \cup\ar@{-}[d] ^{\hat{\rho}^*} & \quad & \quad & \quad & \cup\ar@{-}[d]^{\hat{s}^*} & \quad\\
	\quad & \quad & \quad & \quad & \quad & \quad & \quad\\
	\quad & \alpha & \quad & \quad & \quad & \beta & \quad}
\]

(2) In practical computation, we observe that in each of the two local structures described above, the morphisms at the left and the right ends can each be classified into two types:
\[
\hat{\iota}^*:\ \underset{x \in X}{\bigoplus}\widehat{C}^*\left(C_G(x), k\right)\longrightarrow \mathcal{D}^*(k G, k G), 
\]
and identity morphism
\[
\mathrm{id}:\ \mathcal{D}^*(k G, k G)\longrightarrow \mathcal{D}^*(k G, k G).
\]
To distinguish these more precisely, we further refine the classification. For $\alpha$-type structures, we divide them into the following four subclasses,

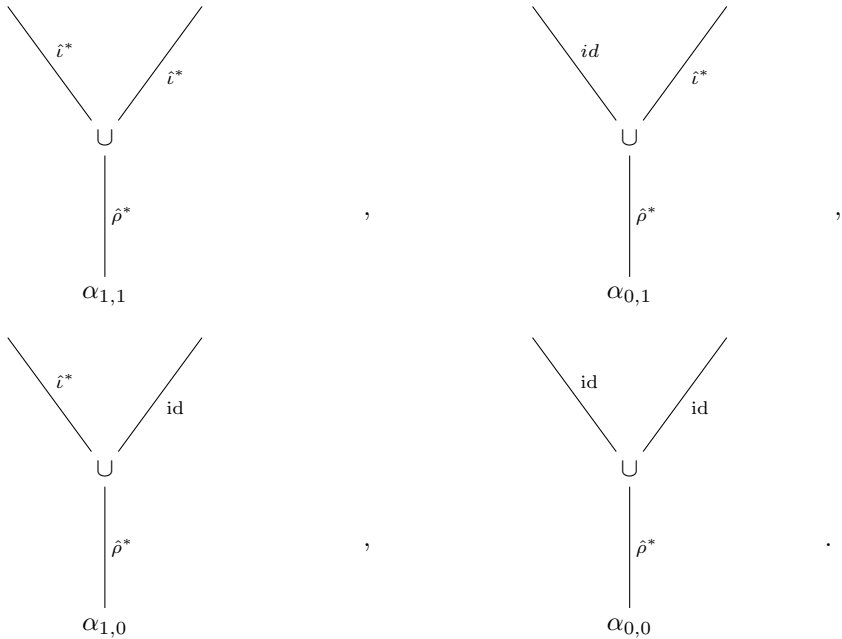
\begin{tikzcd}
	& & {} \arrow[rdd, phantom] \arrow[rdd, "\hat{\iota}^*", no head] & & {} \arrow[ldd, "\hat{\iota}^*", no head] & & & & {} \arrow[rdd, "id", no head] & & {} \arrow[ldd, "\hat{\iota}^*", no head] & \quad \\
	& & & & & & & & & & & \quad \\
	& & & \cup \arrow[dd, "\hat{\rho}^*", no head] & & & & & & \cup \arrow[dd, "\hat{\rho}^*", no head] & & \quad \\
	& & & & & & {,} & & & & &, \\
	& & & {\alpha_{1, 1}} & & & & & & {\alpha_{0, 1}} & &               
\end{tikzcd}

\begin{tikzcd}
	& & {} \arrow[rdd, phantom] \arrow[rdd, "\hat{\iota}^*", no head] & & {} \arrow[ldd, "\mathrm{id}", no head] & & & & {} \arrow[rdd, "\mathrm{id}", no head] & & {} \arrow[ldd, "\mathrm{id}",  no head] & & \quad \\
	& & & & & & & & & & & & \quad \\
	& & & \cup \arrow[dd, "\hat{\rho}^*", no head] & & & & & & \cup \arrow[dd, "\hat{\rho}^*",  no head] & & & \quad \\
	& & & & & & {, } & & & & & . \\
	& & & {\alpha_{1, 0}} &  &  &  &  &  & {\alpha_{0, 0}} & & &      
\end{tikzcd}

Similarly, if we replace position $\hat{\rho}^*$ with $\hat{s}^*$ while keeping everything else unchanged, we can then obtain $\beta_{1, 1}$, $\beta_{0, 1}$, $\beta_{1, 0}$ and $\beta_{0, 0}$.

(3) For type $\alpha$ and type $\beta$, to distinguish whether the input and output elements involved in the correspondence operation lie in cohomology or homology, we refine the notation by adding a sign $+$ and $-$:
\[
\text{add } + \text{ if the corresponding element lies in } \underset{x \in X}{\bigoplus} \widehat{C}^{\ge 0} \left(C_G(x), k\right) \text{ or } \mathcal{D}^{\ge 0}(kG, kG),
\]
\[
\text{add } - \text{ if the corresponding element lies in } \underset{x \in X}{\bigoplus} \widehat{C}^{<0}\left(C_G(x), k\right) \text{ or } \mathcal{D}^{< 0}(kG, kG).
\]
For type $\alpha$, the sign $+$ (resp. $-$) of last position means that the result obtained from type $\alpha$ operation is in $\underset{x \in X}{\bigoplus}\widehat{C}^{\ge 0}\left(C_G(x), k\right)$ (resp. $\underset{x \in X}{\bigoplus}\widehat{C}^{<{0}}\left(C_G(x), k\right)$). For type $\beta$, the sign $+$ (resp. $-$) of last position means that the result obtained from type $\alpha$ operation is in  $\mathcal{D}^{\ge 0}(kG, kG)$ (resp. $\mathcal{D}^{<0}(kG, kG)$).

According to the above classification rule, type $\alpha$ case in (2) can be further divided into $\alpha_{i\pm, j \pm , \pm }$ ($i,j\in\{0,1\}$) and type $\beta$ case in (2) can be further divided into $\beta_{i\pm, j \pm , \pm }$ ($i,j\in\{0,1\}$). For example $\alpha_{1+, 1+, +}$, $\alpha_{1+, 0-, -}$ are as follows respectively:

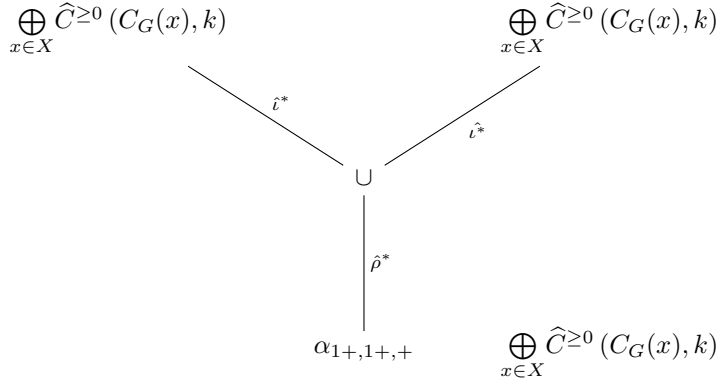
\[
\begin{tikzcd}
	{\underset{x \in X}{\bigoplus}\widehat{C}^{\ge 0}\left(C_G(x), k\right)} \arrow[rdd, phantom] \arrow[rdd, "\hat{\iota}^*", no head] & & {\underset{x \in X}{\bigoplus} \widehat{C}^{\ge 0}\left(C_G(x), k\right)} \arrow[ldd, "\hat{\iota^*}", no head]  \\
	& &  \\
	& \cup \arrow[dd, "\hat{\rho}^*", no head] \\
	& & & & {,}  \\
	& {\alpha_{1+, 1+, +}} & {\underset{x \in X}{\bigoplus}\widehat{C}^{\ge 0}\left(C_G(x), k\right)}   
\end{tikzcd}
\]
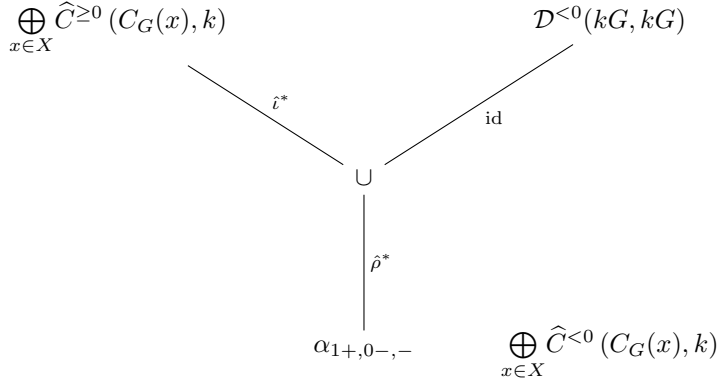
\[
\begin{tikzcd}
	{\underset{x \in X}{\bigoplus}\widehat{C}^{\ge 0}\left(C_G(x), k\right)} \arrow[rdd, phantom] \arrow[rdd, "\hat{\iota}^*", no head] & & {\mathcal{D}^{< 0}(kG, kG)} \arrow[ldd, "\mathrm{id}", no head] \\
	& & \\
	& \cup \arrow[dd, "\hat{\rho}^*", no head] \\
	& & & & {,} \\
	& {\alpha_{1+, 0-, -}} & {\underset{x \in X}{\bigoplus}\widehat{C}^{<0}\left(C_G(x), k\right)} 
\end{tikzcd}
\]

For example, $\beta_{1-, 1-, -}$, $\beta_{0-, 1+, +}$ are described as follows respectively:
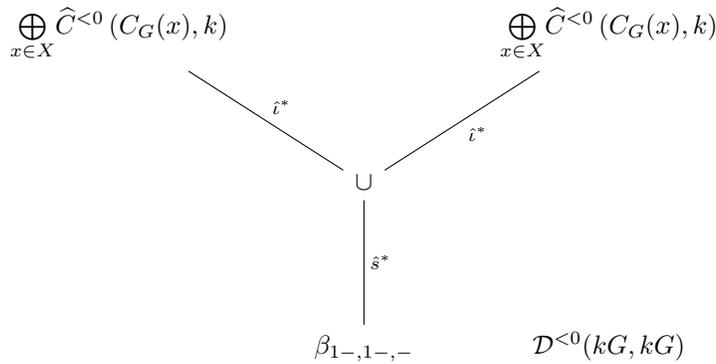
\[
\begin{tikzcd}
	{\underset{x \in X}{\bigoplus}\widehat{C}^{<0}\left(C_G(x), k\right)} \arrow[rdd, phantom] \arrow[rdd, "\hat{\iota}^*", no head] & & {\underset{x \in X}{\bigoplus} \widehat{C}^{<0} \left(C_G(x), k\right)} \arrow[ldd, "\hat{\iota}^*", no head] \\
	& & \\
	& \cup \arrow[dd, "\hat{s}^*", no head] \\
	& & & & {,} \\
	& {\beta_{1-, 1-, -}} & {\mathcal{D}^{< 0}(k G,  k G)}
\end{tikzcd}
\]
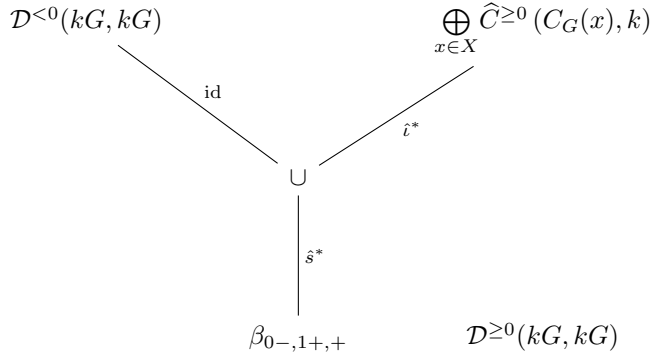
\[
\begin{tikzcd}
	{\mathcal{D}^{< 0}(kG, kG)} \arrow[rdd, phantom] \arrow[rdd, "\mathrm{id}", no head] & & {\underset{x \in X}{\bigoplus}\widehat{C}^{\ge 0}\left(C_G(x), k\right)} \arrow[ldd, "\hat{\iota}^*", no head] \\
	& & \\
	& \cup \arrow[dd, "\hat{s}^*", no head] \\
	& & & & {.} \\
	& {\beta_{0-, 1+, +}} &{\mathcal{D}^{\ge 0}(kG, kG)} 
\end{tikzcd}
\]

\subsubsection{Algorithms on locally three-branched graphs}\

we provide a classification of locally three-branched trees, along with their correspondence notations, following the approach used for locally two-branched trees.

(1) We still use the notation as last section: denote by $\alpha$-type and $\beta$-type for the trees with terminal morphism $\hat{\rho}^*$ and $\hat{s}^*$, respectively.

(2) The classification of $0$ and $1$ are also the same as before: we further refine the classification according to the morphism at each branch. Specifically, they are divided into $8$ types $\alpha_{i, j, k}$ ($i, j, k=1$ or $0$), and $\beta_{i, j, k}$ ($i, j, k=1$ or $0$). It should be noticed that the labels $i$, $j$ and $k$ indicate the branches order from left to right. For example,
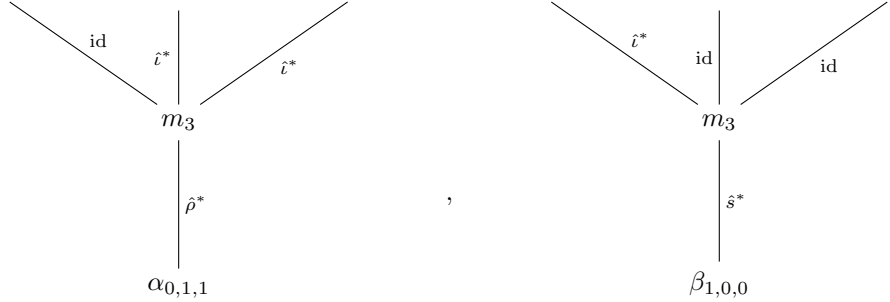
\[
\begin{tikzcd}
{} \arrow[rrdd, phantom] \arrow[rrdd, "\mathrm{id}", no head] & & {} & & {} \arrow[lldd, "\hat{\iota}^*", no head] & & {} \arrow[rrdd, "\hat{\iota}^*", no head] & & {} & & {} \arrow[lldd, "\mathrm{id}", no head] \\
& & \\
& & m_3 \arrow[dd, "\hat{\rho}^*", no head] \arrow[uu, "\hat{\iota}^*", no head] & & & &  & & m_3 \arrow[dd, "\hat{s}^*", no head] \arrow[uu, "\mathrm{id}", no head] \\
& & & & & {,} \\
& & {\alpha_{0, 1, 1}} & & & & & & {\beta_{1, 0, 0}}            
\end{tikzcd}
\]

(3) The definition of sign $+$ and $-$ are same as before: we futher refine the classification, denote as $\alpha_{i \pm, j \pm, k \pm,  \pm }$ ($i, j, k\in\{0,1\}$), and $\beta_{i \pm, j \pm, k \pm, \pm}$ ($i, j, k\in \{0,1\}$). For example,
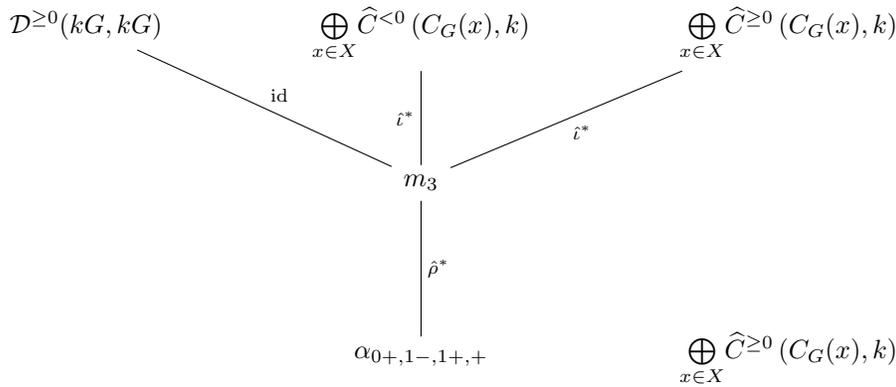
\[
\begin{tikzcd}
{{\mathcal{D}^{\ge 0}(kG, kG)}} \arrow[rrdd, phantom] \arrow[rrdd, "\mathrm{id}", no head] & & {{\underset{x \in X}{\bigoplus}\widehat{C}^{<0}\left(C_G(x), k\right)}} & & {{\underset{x \in X}{\bigoplus}\widehat{C}^{\ge 0}\left(C_G(x), k\right)}} \arrow[lldd, "\hat{\iota}^*", no head] \\
& & \\
& & m_3 \arrow[dd, "\hat{\rho}^*", no head] \arrow[uu, "\hat{\iota}^*", no head] \\
& & & & & {,} \\
& & {\alpha_{0+, 1-, 1+, +}} & & {{\underset{x \in X}{\bigoplus}\widehat{C}^{\ge 0}\left(C_G(x), k\right)}}          
\end{tikzcd}
\]
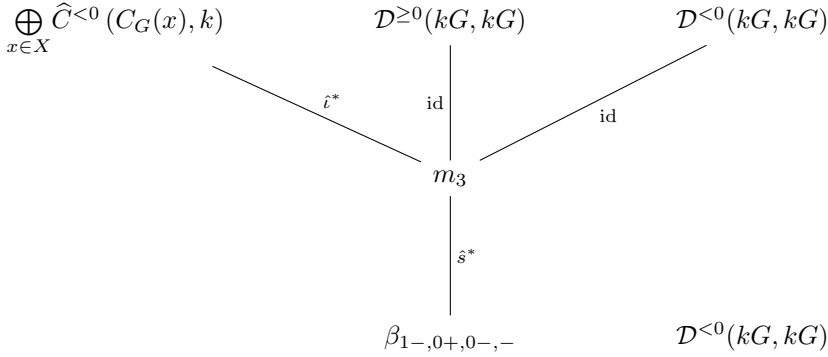
\[
\begin{tikzcd}
{{\underset{x \in X}{\bigoplus}\widehat{C}^{<0}\left(C_G(x), k\right)}} \arrow[rrdd, phantom] \arrow[rrdd, "\hat{\iota}^*", no head] & & {{\mathcal{D}^{\ge 0}(kG, kG)}} & & {{\mathcal{D}^{<0}(kG, kG)}} \arrow[lldd, "\mathrm{id}", no head] \\
& & \\
& & m_3 \arrow[dd, "\hat{s}^*", no head] \arrow[uu, "\mathrm{id}", no head] \\
& & & & & {.} \\
& & {\beta_{1-, 0+, 0-, -}} & & {{\mathcal{D}^{<0}(kG, kG)}}
\end{tikzcd}
\]

Now we recall from Section~\ref{Section: Tate-Hochschild Cohomology of a Group Algebra} the definition of $m_3$ in $\mathcal{D}^{*}(kG, kG)$, it is $0$ except for the following two cases:
\begin{itemize}
    \item [(i)] For $\phi \in C^m(kG,kG), \varphi \in C^n(kG,kG)$ and $\alpha=(g_{0}, g_{1},  \cdots, g_r) \in C_r(kG,kG)$, if $r+2\le m+n$, then $m_3(\phi, \alpha, \varphi) \in C^{m-r+n-2}(kG,kG)$ is defined by
         \begin{align*}
  		m_3(\phi, \alpha, \varphi)(h_{1}, \cdots, h_{m-r+n-2})=\sum_{g \in G} \sum_{j=\max\{1,r+2-m\}}^{\min \{n, r+1\}}(-1)^{m+r+j-1} \quad\quad\quad\quad\quad\quad\quad\quad\quad\quad\quad \\
  	\phi(h_{1, m-r+j-2}, g, g_{j, r}) g_{0} \varphi(g_{1, j-1}, g^{-1}, h_{m-r+j-1, m-r+n-2}). 
      \end{align*}

      \item [(ii)] For $\alpha=(g_0, g_{1, r}) \in C_r(kG, kG), \beta=(h_0, h_{1, s}) \in C_s(kG, kG)$ and $\phi \in C^m(kG, kG)$, if $m-1\le r+s$, 
        \begin{align*}
		m_3(\alpha, \phi, \beta)=\sum_{g \in G} \sum_{j=\max\{0,s+1-m\}}^{\min \{s,  r-m+s+1\}}(-1)^{m+r+s-j}\quad\quad\quad\quad\quad\quad\quad\\
        (g_{0}\phi(g_{1, m-s+j-1}, g, h_{j+1, s}) h_{0}, h_{1, j}, g^{-1}, g_{m-s+j, r}),
	\end{align*}
\end{itemize}

\subsubsection{\texorpdfstring{Computation for $\widehat{m}_n$}{Computation for mn}}\

In the previous section, we discussed all classification cases of locally branched graphs involved in the computation of $\widehat{m}_n$. Based on this, a planar $n$-ary tree can be decomposed into a composition of $\alpha_{i\pm, j\pm, k\pm, \pm}$, $\alpha_{i\pm, j\pm,  \pm}$, $\beta_{i\pm, j\pm, k\pm, \pm}$ and $\beta_{i\pm, j\pm, \pm}$. The corresponding algorithmic flowchart is given below.
\begin{figure}[ht]
\[
\begin{tikzcd}[row sep=1.2em, 
               column sep=0.01em,  
               cells={nodes={font=\small}}  
               ]
& & & & & a_5 \arrow[rd, "\hat{l}^*", no head] & & a_6 \arrow[ld, "\hat{l}^*", no head] & & & & & & \\
& & a_1 \arrow[rd, "\hat{l}^*", no head] & a_2 \arrow[d, "\hat{l}^*", no head] & a_3 \arrow[ld, "\hat{l}^*", no head] & a_4 \arrow[d, "\hat{l}^*", no head] & m_2 & & & & & & & \\
& & & m_3 \arrow[rd, "\hat{s}^*", no head] & & m_2 \arrow[ld, "\hat{s}^*", no head] \arrow[ru, "\hat{s}^*", no head] & & & & & & & & \\
{} \arrow[rrdd, phantom] & \widehat{m}_n & =\sum_{PT_n}\pm & & m_2 \arrow[d, "\hat{p}^*", no head] & & & & & & & & & \quad \\ 
& & & & {} & & {a_5, a_6} \arrow[d, "\text{input}", Rightarrow] & & & & & & & \quad \\
& & =\sum_{PT_n}\pm & {a_1, a_2, a_3} \arrow[d, "\text{input}", Rightarrow] & & a_4 \arrow[d, "\text{input}", Rightarrow] & {\beta_{1\pm, 1\pm, \pm}} \arrow[ld, "\text{input result}", Rightarrow] & & & & & & & \quad \\
& & & {\beta_{1\pm, 1\pm, 1\pm, \pm}} \arrow[rd, "\text{input result}", Rightarrow] & & {\beta_{1\pm, 0\pm, \pm}} \arrow[ld, Rightarrow] & & & & & & & & .\\
& & & & {\alpha_{0\pm, 0\pm, \pm}} \arrow[d, "\text{output result}", Rightarrow] & & & & & & & & & \\
& & & & {} & & & & & & & & &      
\end{tikzcd}
\]
\caption{The algorithmic flowchart corresponding to a planar $n$-ary tree.}
\label{figure:n-ary-algorithm}
\end{figure}
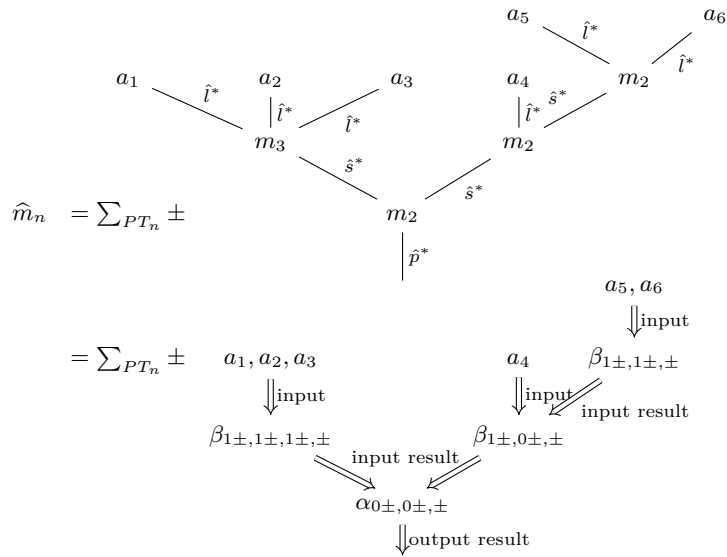

We now give a concrete example to demonstrate the above algorithm.

\begin{Ex}
For the following $PT_4$ graph, we calculate a term in the correspondence multiplication $\widehat{m}_4$, with $n>0, m>0, p>0, q>0$ and $n+m\ge p-q-2$:
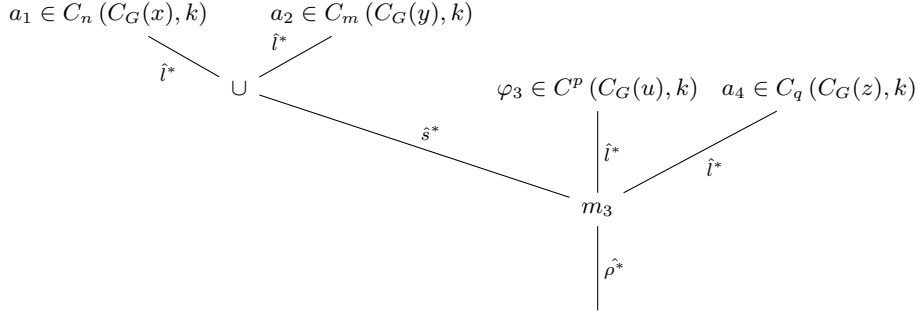
\[
\begin{tikzcd}	[	row sep=1.2em,  
	column sep=0.01em,  
	cells={nodes={font=\small}} 
	]
{a_1\in{C_n\left(C_G(x), k\right)}} & & {a_2\in{C_m\left(C_G(y), k\right)}} & & & & & & & & & & & & \\
& \cup \arrow[rrdd, phantom] \arrow[rrdd, "\hat{s}^*", no head] \arrow[lu, "\hat{l}^*", no head] \arrow[ru, "\hat{l}^*", no head] & & {{\varphi_3} \in{C^p\left(C_G(u), k\right)}} & & {a_4\in{C_q\left(C_G(z), k\right)}} \arrow[lldd, "\hat{l}^*", no head] & & & & & & & & & \quad \\
& & & & & & & & & & & & & & \quad \\
& & & m_3 \arrow[dd, "\hat{\rho^*}", no head] \arrow[uu, "\hat{l}^*"', no head] & & & & & & & & & & & \quad \\
& & & & & & & & & & & & & &. \\
& & & {} & & & & & & & & & & &   
\end{tikzcd}
\]

There are one two-branched and one three-branched structures in the above graph, we firstly compute by operation $\beta_{1-, 1-, -}$, and then by operation $\alpha_{0-, 1+, 1-, -}$. Specificly, it can be demonstrated as follows:
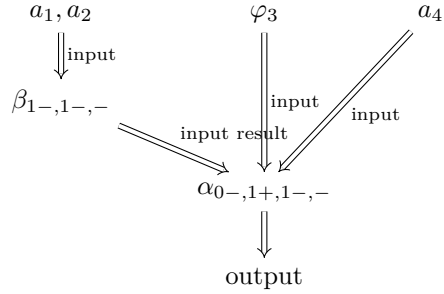
\[
\begin{tikzcd}
 {a_1, a_2} \arrow[d, "\text{input}", Rightarrow] & {\varphi_3} \arrow[dd, "\text{input}", Rightarrow] & a_4 \arrow[ldd, "\text{input}", Rightarrow] \\
{\beta_{1-, 1-, -} }\arrow[rd, "\text{input result}", Rightarrow]  \\
& {\alpha_{0-, 1+, 1-, -}} \arrow[d,  Rightarrow] \\
& \text{output}
\end{tikzcd}
\]
\end{Ex}

In fact, all the output results of operations $\alpha_{1\pm, 1\pm, \pm}$ are all possible cases of $\widehat{m}_2$. To compute $A_{\infty}$-structure, it is also necessary to consider operations $\alpha_{i\pm, j\pm, \pm}$, $\beta_{i\pm, j\pm, \pm}$, $\alpha_{i+, j-, k+, +}$, $\alpha_{i-, j+, k-, -}$, $\beta_{i+, j-, k+, +}$ and $\beta_{i-, j+, k-, -}$. In the following sections, we provide the computations for $\alpha_{1 \pm , 1 \pm ,  \pm }$, $\alpha_{i+ , j + ,  + }$, $\alpha_{i- , j - ,  -}$, $\beta_{1 \pm , 1 \pm ,  \pm }$, $\alpha_{i+ , j + ,  + }$ and $\alpha_{i- , j - ,  -}$. The reader may refer to these examples as a basis for handling the remaining cases.

\subsection{\texorpdfstring{Computations for $\widehat{m}_2$ ($=\alpha_{1 \pm, 1 \pm, \pm}$) and $\beta_{1 \pm, 1 \pm, \pm}$}{Computations for alpha and beta}}\label{Section: computation for alpha and beta}\

Recall from Section~\ref{section:additive-decomposition-TT} the homotopy retract of the additive decomposition of the Tate-Hochschild cohomology at the complex level
\begin{equation*}\xymatrix@C=0.000000001pc{
\ar@(lu, dl)_-{\hat{s}^*} & \mathcal{D}^*(kG, kG)\ar@<0.5ex>[rrrrr]^-{\hat{\rho}^*} & & & & & \underset{x \in X}{\bigoplus} \widehat{C}^*\left(C_G(x),  k\right), \ar@<0.5ex>[lllll]^-{\hat{\iota}^*}}
\end{equation*}
where for $m \ge 0$, we have
\[
\hat{\iota}^m = \rho^m,\  \hat{\iota}^{-m-1} = \iota_m ;\quad
\hat{\rho}^m = \iota^m,\  \hat{\rho}^{-m-1} = \rho_m ; \quad
\hat{s}^m = s^m,\  \hat{s}^{-m-1} = s_m.
\]

The cup product formula on $\mathcal{D}^{\ge 0}(kG, kG)$ are given in \cite{LZ16}, and the cup product formula on $\mathcal{D}^{<0}(kG, kG)$ are given in \cite{LWZ21}. We now recall the formulas frow this two papers and provide all the formulas of $\widehat{m}_2$ ($=\alpha_{1 \pm, 1 \pm, \pm}$) and $\beta_{1 \pm, 1 \pm, \pm}$.

\emph{Notation convention:} Let $X$ be a set of representatives of the conjugacy classes of elements in $G$. For each $x \in X$, $C_G(x)\subset G$ is the centralizer subgroup of $x$. Fix a decomposition of $G$ into right cosets of $C_G(x)$:
\[
G = C_G(x)\gamma_{1, x} \cup C_G(x)\gamma_{2, x} \cup \cdots \cup C_G(x)\gamma_{n_x, x},
\]
where $n_x$ is the number of elements in the conjugacy class $C_x$. Then the conjugacy class $C_x$ can be written as:
\[
C_x = \left\{ \gamma_{1, x}^{-1} x \gamma_{1, x}, \cdots, \gamma_{n_x, x}^{-1} x \gamma_{n_x, x} \right\}.
\]
We denote $x_i = \gamma_{i, x}^{-1} x \gamma_{i, x}$ and, without loss of generality, let $\gamma_{1, x} = 1$, so $x_1 = x$.

\textbf{Case 1. $\alpha_{1+, 1+, +}$ \cite{LZ16} and $\beta_{1+, 1+, +}$}

Let $\widehat{\varphi}_{x}: \overline{C_G(x)}^n\to k \in C^n(C_G(x),k)$ and $\widehat{\varphi}_{y}: \overline{C_G(y)}^m\to k \in C^n(C_G(y),k)$ with $x,y\in X$ and $m,n\ge 0$.

We first calculate $\alpha_{1+, 1+, +}$. We have $\alpha_{1+, 1+, +}\left(\widehat{\varphi}_x, \widehat{\varphi}_y\right)=\widehat{m}_2 \left(\widehat{\varphi}_x, \widehat{\varphi}_y\right)$. Assume that
\[
\widehat{m}_2 \left(\widehat{\varphi}_x, \widehat{\varphi}_y\right)=\sum_{z\in X} \widehat{m}_2 \left(\widehat{\varphi}_x, \widehat{\varphi}_y\right)_z\in \bigoplus_{z\in X} C^{n+m}(C_G(z), k),
\]
and by homotopy transfer theorem, for any $z\in X$, $\widehat{m}_2 \left(\widehat{\varphi}_x, \widehat{\varphi}_y\right)_z$ is defined as follows:
\[
\widehat{m}_2 \left(\widehat{\varphi}_x, \widehat{\varphi}_y\right)_z=\iota^{z,n+m} ( \rho^{x,n}(\widehat{\varphi}_x) \cup \rho^{y,m}(\widehat{\varphi}_y))_z.
\]
More specifically, for $h_1,\cdots,h_{n+m} \in \overline{C_G(z)}$,
\[
\widehat{m}_2 \left(\widehat{\varphi}_x, \widehat{\varphi}_y\right)_z (h_{1,n+m})= \sum_{(i, j) \in I_z} \widehat{\varphi}_x\left(h_{i,1}^x, \cdots, h_{i,n}^x\right) \widehat{\varphi}_{y}\left(h_{j,1}^y, \cdots, h_{j,m}^y\right)
\]
where 
\begin{itemize}
	\item $I_{z}=\left\{ (i, j) |\  i\in \left\{1, \cdots, n_x\right\}, j\in \left\{1, \cdots, n_y\right\} \text{ and }  x_i\left(h_1 \cdots h_n\right) y_j\left(h_1 \cdots h_n\right)^{-1}=z \right\}$,
    
	\item $h_{i,1}^x, \cdots, h_{i,n}^x \in \overline{C_G(x)}$ are determined by $\{h_1, \cdots h_n\}$, $x$, and $i$ from process $\spadesuit$, that is,
	\[
    \gamma_{i, x} h_1=h_{i,1}^x \gamma_{s_i^1, x}, \quad  \gamma_{s_i^1, x} h_2=h_{i,2}^x \gamma_{s_i^2, x}, \cdots, \gamma_{s_i^{n-1}, x} h_n=h_{i,n}^x \gamma_{s_i^n, x},
    \]
    we write $\{h_{i,1}^x, \cdots, h_{i,n}^x\}=\spadesuit_{x,i} \{h_1, \cdots, h_n\}$.
    
	\item	$h_{j,1}^y, \cdots, h_{j,m}^y \in \overline{C_G(y)}$ and $\{h_{j,1}^y, \cdots, h_{j,m}^y \}=\spadesuit_{y,j}\{h_{n+1},\cdots,h_{n+m} \}$.
\end{itemize}

\vspace{2em}  %加大行距

Now we calculate $\beta_{1+, 1+, +}$, we have 
\[
\beta_{1+, 1+, +} \left(\widehat{\varphi}_x, \widehat{\varphi}_y\right)=\sum_{z\in X} s^{z,n+m} ( \rho^{x,n}(\widehat{\varphi}_x) \cup \rho^{y,m}(\widehat{\varphi}_y))_z,
\]
by definition of $s^z$, for $g_1,\cdots,g_{m+n-1}\in \overline{G}$,
\begin{eqnarray*}
&s^{z,n+m} ( \rho^{x,n}(\widehat{\varphi}_x) \cup \rho^{y,m}(\widehat{\varphi}_y))_z (g_{1, m+n-1})\\
=&\sum\limits_{l=0}^{m+n-1}\sum\limits_{a=1}^{n_z}(-1)^{l}e_{a, l}^1z_a g_1\cdots   g_{m+n-1}, 
\end{eqnarray*}
where $e_{a, l}^1$ is, for each $l\in \{0,1,\ldots,m+n-1\}$ and $a\in \{1,\ldots,n_z\}$, is determined by 
\[
(  \rho^{x,n}(\widehat{\varphi}_x) \cup \rho^{y,m}(\widehat{\varphi}_y))_z(h_{a,1}^z, \dots , h_{a,l}^z, \gamma_{s_a^l, z}, g_{l+1}, \dots, g_{m+n-1})g_{m+n-1}^{-1}\cdots  g_1^{-1}\gamma_{a, z}^{-1}=\sum\limits_{t=1}^{n_z}e_{a, l}^t z_t,
\]
with $(h_{a, 1}^z, \dots , h_{a, m+n-1}^z)=\spadesuit_{z, a}(g_1, \dots , g_{m+n-1})$, and by computation of $\alpha_{1+, 1+, +}$, $e_{a,l}^1$ can be obtained by using operation $\alpha_{1+, 1+, +}$ on $h_{a, 1}^z, \dots , h_{a, l}^z, \gamma_{s_a^l, z}, g_{l+1}, \dots , g_{m+n-1}$, that is,
\begin{eqnarray*}
e_{a, l}^1
&=& \iota^z {(\rho^x  \left(\widehat{\varphi}_x\right) \cup \rho^y \left(\widehat{\varphi}_y \right))_z} 
(h_{a, 1}^z, \dots , h_{a, l}^z, \gamma_{s_a^l, z}, g_{l+1}, \dots , g_{m+n-1})\\
&=&\sum\limits_{(i, j)\in I_z^{a,l} }\widehat{\varphi}_x(h_{i, 1}^x, \dots , h_{i, n}^x) \widehat{\varphi}_y (h_{j, 1}^y, \dots , h_{j, m}^y), 
\end{eqnarray*}
where for each $a,l$, 
\begin{itemize}
    \item $I_z^{a,l}=\{ (i, j)|\ 1\le i\le n_x, 1\le j\le n_y \text{ and } x_i(g_1^{a,l} \cdots  g_n^{a,l} )y_j(g_1^{a,l} \cdots  g_n^{a,l})^{-1}=z \}$,  here we rewrite 
    \[
    (g_1^{a,l} , g_2^{a,l} , \cdots, g_{m+n}^{a,l} )=(h_{a, 1}^z, \dots , h_{a, l}^z, \gamma_{s_a^l, z}, g_{l+1}, \dots , g_{m+n-1}),
    \]
    and for each $(i,j)\in I_z^{a,l}$,

    \item $\{ h_{i, 1}^x, \dots , h_{i, n}^x \}=\spadesuit_{x,i}\{ g_1^{a,l}, \cdots, g_n^{a,l}\}$,

    \item $\{ h_{j, 1}^y, \dots , h_{j, m}^y \}=\spadesuit_{y,j}\{ g_{n+1}^{a,l}, \cdots, g_{n+m}^{a,l}\}$.
\end{itemize}

\vspace{2em} 

\textbf{Case 2. $\alpha_{1-, 1-, -}$ \cite{LWZ21} and $\beta_{1-, 1-, +}$ }

Let $\widehat{\alpha}_x =\left(g_{1, s}\right) \in  \widehat{C}^{-s-1} \left(C_G(x),  k\right) =C_s\left(C_G(x), k\right) =k\left[{\overline{C_G(x)}}^{\times s}\right]$ and 
$\widehat{\alpha}_y =\left(h_{1, t}\right)\in C_t\left(C_G(y), k\right)$ with $x, y \in X$, $s, t\ge 0$.

Recall from \cite{LWZ21}, we have
\[
\alpha_{1-, 1-, -}\left(\widehat{\alpha}_x, \widehat{\alpha}_y\right)=\widehat{m}_2 \left(\widehat{\alpha}_x, \widehat{\alpha}_y\right)=\sum_{z\in X} \rho_{z,s+t} \left( \iota_{x,s}(\widehat{\alpha}_x)\cup \iota_{y,t} (\widehat{\alpha}_y) \right)_z.
\]

For any $z\in X$, we define
\[
I_{z}:=\left\{g \in G | h_1 \cdots h_t g^{-1} g_s^{-1} \cdots g_1^{-1} x g_1\cdots g_s g h_t^{-1} \cdots h_1^{-1} y=\Phi(g) ^{-1} z \Phi(g) \text{ for some } \Phi (g)\in G \right\}.
\]
Then
\begin{eqnarray*}
&&\rho_{z,s+t} \left( \iota_{x,s} (\widehat{\alpha}_x)\cup \iota_{y,t} (\widehat{\alpha}_y) \right)_z\\
&=& \rho_{z,s+t} \big( \sum_{g\in G}(g h_t^{-1} \cdots h_1^{-1} y,\ h_{1,t},\ g^{-1} g_s^{-1} \cdots g_1^{-1} x,\ g_{1,s} ) \big)_z \\
&= &\rho_{z,s+t} \big( \sum_{g\in I_z} (g_s^{-1} \cdots g_1^{-1} x^{-1} g_1 \cdots g_s g h_t^{-1} \cdots h_1^{-1} \Phi(g) ^{-1} z \Phi(g),\  h_{1,t},\ g^{-1} g_s^{-1} \cdots g_1^{-1} x,\ g_{1,s}  ) \big)\\
&= &\sum_{g \in I_z}\left(k_{i_g,1}^z, \cdots, k_{i_g, s+t+1}^z\right) \in C_{s+t+1}\left(C_G(z),  k\right), 
\end{eqnarray*}
\[
\widehat{m}_2\left(\widehat{\alpha}_x,  \widehat{\alpha}_y\right) =  \sum_{z \in X} \sum_{g \in I_z}\left(k_{i_g,1}^z, \cdots, k_{i_g,s+t+1}^z\right)\in \underset{z\in Z}{\bigoplus} C_{s+t+1}\left(C_G(z), k\right),
\]
where for each $z\in X$, $k_{i_g,1}^z, \cdots, k_{i_g,s+t+1}^z \in \overline{C_G(z)}$ are uniquely determined by the following equations:
    \[
    \Phi(g) \in C_G(z) \gamma_{i_g, z},\ \gamma_{i_g, z} h_1 =k_{i_g,1}^z \gamma_{s_{i_g}^1, z},\ \gamma_{s_{i_g}^1, z} h_2=k_{i_g,2}^z \gamma_{s_{i_g}^2, z}, \cdots,  \gamma_{s_{i_g}^{t-1},  z} h_t =k_{i_g,t}^z \gamma_{s_{i_g}^t, z}, 
    \]
    \[
    \gamma_{s_{i_g}^t, z} g^{-1} g_s^{-1} \cdots g_1^{-1} x=k_{i_g,t+1}^z \gamma_{s_{i_g}^{t+1},  z},\ \gamma_{s_{i_g}^{t+1}, z}  g_1=k_{i_g,t+2}^z \gamma_{s_{i_g}^{t+2}, z}, \cdots,   \gamma_{s_{i_g}^{t+s}, z} g_s=k_{i,t+s+1}^z \gamma_{s_{i_g}^{t+s+1}, z},
    \]
    which means that $\Phi(g) \in C_G(z) \gamma_{i_g, z}$ and 
    \[
    \{ k_{i_g,1}^z, \cdots, k_{i_g,s+t+1}^z \}=\spadesuit_{z,i_g} \{h_1, \cdots, h_t, g^{-1} g_s^{-1} \cdots g_1^{-1} x, g_1, \cdots, g_s  \}.
    \]

\begin{Rmk}
    Notice that $k_{i_g,1}^z, \cdots, k_{i_g,s+t+1}^z \in \overline{C_G(z)}$ are independent to the choice of $\Phi ( g ) \in G$. If we also choose $\Phi ( g' ) \in G$ such that
    \[
    h_1 \cdots h_t g^{-1} g_s^{-1} \cdots g_1^{-1} x g_1\cdots  g_s g h_t^{-1} \cdots h_1^{-1} y=\Phi'(g)^{-1}  z \Phi'(g)=\Phi(g)^{-1}  z \Phi(g),
    \]
    then $\Phi'(g)\Phi(g)^{-1} \in C_G(z)$, $\Phi'(g)=\Phi'(g) \Phi(g)^{-1} \Phi(g) \in C_G(z) \gamma_{i_g, z}$.
\end{Rmk}

\vspace{2em}

Now we calculate $\beta_{1-, 1-, -}$. 

For any $z\in X$ and $g\in I_z$,  we rewrite $\{ h_1,\cdots,h_t, g^{-1} g_s^{-1}\cdots g_1^{-1} x, g_1,\cdots, g_s \}$ as $\{ h_{g,1},\cdots,h_{g,s+t+1} \}$. From computation of $\alpha_{1-, 1-, -} $,
\[
\{ k_{i_g,1}^z, \cdots, k_{i_g,s+t+1}^z \}=\spadesuit_{z,i_g} \{ h_{g,1},\cdots,h_{g,s+t+1} \}.
\]
By definition of $s_{z,s+t}$, we have 	 	
\begin{align*}
& s_{z,s+t}(\iota_{x,s} (\widehat{\alpha}_x ) \cup \iota_{y,t} (\widehat{\alpha}_y ) )_z \\
=& s_{z,s+t} \big( \sum_{g\in I_z} (h_{g,s+t+1}^{-1} \cdots h_{g,1}^{-1} \Phi(g) ^{-1} z \Phi(g),\  h_{g,1},\cdots,h_{g,s+t+1} ) \big)\\
=& \sum_{g\in I_z}  \sum\limits_{j=0}^{s+t+1}  (-1)^j  (h_{g,s+t+1}^{-1} \cdots h_{g,1}^{-1} \Phi(g) ^{-1} z h,\  k_{i_g,1}^z, \cdots, k_{i_g,j}^z, \gamma_{s_{i_g}^j,z}, h_{g,j+1}, \cdots,h_{g,s+t+1} ) \\
=& \sum_{g \in I_z}  \sum\limits_{j=0}^{s+t+1}  (-1)^j \big((k_{i_g,1}^z \cdots k_{i_g,j}^z \gamma_{s_{i_g}^j,z} h_{g,j+1} \cdots h_{g,s+t+1})^{-1} z , \  k_{i_g,1}^z, \cdots, k_{i_g,j}^z, \gamma_{s_{i_g}^j,z}, h_{g,j+1}, \cdots,h_{g,s+t+1} \big) .
\end{align*}	 	
Thus, $\beta_{1-, 1-, -}\left(\widehat{\alpha}_x, \widehat{\alpha}_y\right)=\sum_{z\in X}\sum_{g \in I_z}  \sum\limits_{j=0}^{s+t+1}  (-1)^j$
\[
\big((k_{i_g,1}^z \cdots k_{i_g,j}^z \gamma_{s_{i_g}^j,z} h_{g,j+1} \cdots h_{g,s+t+1})^{-1} z , \  k_{i_g,1}^z, \cdots, k_{i_g,j}^z, \gamma_{s_{i_g}^j,z}, h_{g,j+1}, \cdots,h_{g,s+t+1} \big).
\]

\vspace{2em}

\textbf{Case 3. $\alpha_{1+, 1-, -}$ and $\beta_{1+, 1-, -}$}

Let $\widehat{\varphi}_x: \overline{C_G(x)}^n\to k \in C^n(C_G(x),k)$ and $\widehat{\alpha}_y =\left(h_{1, t}\right)\in C_t\left(C_G(y), k\right)$ with $x,  y \in X$, $n, t \ge 0$ and $n-t-1 <0$. We want to compute
\begin{align*}
    \alpha_{1+, 1-, -} (\widehat{\varphi}_{x}, \widehat{\alpha}_y)= \widehat{m}_2 (\widehat{\varphi}_{x}, \widehat{\alpha}_y)
    &= \rho_{t-n} \big(\rho^{x,n} (\widehat{\varphi}_x) \cup \iota_{y,t} (\widehat{\alpha}_y) \big)\\
    &= \sum_{z\in X} \rho_{z,t-n} \big( \rho^{x,n} (\widehat{\varphi}_x) \cup \iota_{y,t} (\widehat{\alpha}_y) \big)_z.
\end{align*}

Firstly, we compute 
$  \rho^{x,n} (\widehat{\varphi}_x) \cup \iota_{y,t} (\widehat{\alpha}_y) \in C_{t-n}\left(kG, kG \right)$:
\[
\rho^{x,n} (\widehat{\varphi}_x) :=\varphi_{x}:\overline{G} ^{\times n} \rightarrow kG, \quad \varphi_x \left(g_{1, n}\right) = \sum_{i=1}^{n_x} \widehat{\varphi}_x \left(k_{i,1}^x, \cdots k_{i,n}^x \right) x_i g_1 g_2\cdots g_n,
\]
where $k_{i,1}^x, \cdots, k_{i,n}^x \in \overline{C_G(x)}$ and 
$\{k_{i,1}^x, \cdots, k_{i,n}^x\} = \spadesuit_{x,i}\{ g_1,\cdots,g_n \}$;
\[
\iota_{y,t} \left( \widehat{\alpha}_y \right) =\iota_{y,t} \left(h_{1, t}\right) =\left(h_t^{-1} \cdots h_1^{-1} y, h_{1, t} \right) \in \mathcal{H}_{y,t};
\]
\begin{align*}
\rho^{x,n} (\widehat{\varphi}_x) \cup \iota_{y,t} (\widehat{\alpha}_y)
&=\left( \varphi_x \left(h_{t-n+1, t} \right)h_t^{-1} \cdots h_1^{-1} y, h_{1, t-n}\right) \\ &=\sum_{i=1}^{n_x}\left( \widehat{\varphi}_x \left( k_{i,1}^x, \cdots, k_{i,n}^x \right) x_i h^{-1}_{t-n}\cdots h_1^{-1}y, h_{1, t-n}\right)
\\&=\sum_{i=1}^{n_x} \widehat{\varphi}_x \left( k_{i,1}^x, \cdots, k_{i,n}^x \right) \left(x_i h^{-1}_{t-n}\cdots h_1^{-1} y, h_{1, t-n}\right) \in C_{t-n} (kG, kG )
\end{align*}    
with  $\{k_{i,1}^x, \cdots, k_{i,n}^x\} = \spadesuit_{x,i}\{ h_{t-n+1},\cdots,h_t \}$; then
\begin{align*}
\rho_{t-n} \big(\rho^{x,n} (\widehat{\varphi}_x) \cup \iota_{y,t} (\widehat{\alpha}_y) \big)
&=\sum_{z\in X} \rho_{z,t-n} \big( \rho^{x,n} (\widehat{\varphi}_x) \cup \iota_{y,t} (\widehat{\alpha}_y) \big)\\
&=  \sum_{z\in X} \sum_{i\in I_z} \widehat{\varphi}_x \left( k_{i,1}^x, \cdots, k_{i,n}^x \right)  \rho_{t-n,z} ( h_{t-n}^{-1}\cdots h_1^{-1} \phi(x_i)^{-1} z \phi(x_i), h_{1,t-n} ) \\
&= \sum_{z\in X} \sum_{i\in I_z} \widehat{\varphi}_x \left( k_{i,1}^x, \cdots, k_{i,n}^x \right) ( h_{j_i,1}^z, \cdots, h_{j_i,t-n}^z )
\end{align*}
where 
\begin{itemize}
    \item for any $z\in X$, $I_z$ is defined by 
    \[
    I_z:= \left\{i\ |\ 1\le i\le n_x,\  h_1 \cdots h_{t-n} x_i  h_{t-n}^{-1} \cdots h_1^{-1} y=\phi(x_i)^{-1} z \phi(x_i) \text{ for some } \phi(x_i) \in G \right\};
    \]

    \item $ h_{j_i,1}^z, \cdots, h_{j_i,t-n}^z\in \overline{C_G(z)}$ are uniquely determined by $\Phi (x_i)=h \gamma _{j_i, z} \in \overline{C_G(z)}\gamma _{j_i, z}$ and 
    \[
    \{  h_{j_i,1}^z, \cdots, h_{j_i,t-n}^z \}=\spadesuit_{z,j_i} \{ h_1,\cdots, h_{t-n} \}.
    \]
\end{itemize}

\vspace{2em}  

Now we calculate $\beta_{1+, 1-, -}$. 
\begin{align*}
    \beta_{1+, 1-, -} (\widehat{\varphi}_{x}, \widehat{\alpha}_y)
    &= s_{t-n} \big(\rho^{x,n} (\widehat{\varphi}_x) \cup \iota_{y,t} (\widehat{\alpha}_y) \big)\\
    &= \sum_{z\in X} s_{z,t-n} \big( \rho^{x,n} (\widehat{\varphi}_x) \cup \iota_{y,t} (\widehat{\alpha}_y) \big)_z.
\end{align*}
For any $z\in X$, by deinition of $s_{z, t-n}: \mathcal{H}_{z, t-n} \rightarrow \mathcal{H}_{z, t-n+1}$, we have
\begin{align*}
  s_{z,t-n} \big( \rho^{x,n} (\widehat{\varphi}_x)& \cup \iota_{y,t} (\widehat{\alpha}_y) \big)_z
 = \sum_{i\in I_z} \sum_{u=0}^{t-n} (-1)^u \widehat{\varphi}_x \left( k_{i,1}^x, \cdots, k_{i,n}^x \right)\\
  &\big( (h_{j_i,1}^z \cdots h_{j_i,u}^z \gamma_{s_{j_i}^u,z} h_{u+1} \cdots  h_{t-n} )^{-1} z,\ h_{j_i,1}^z, \cdots, h_{j_i,u}^z, \gamma_{s_{j_i}^u,z},h_{u+1},\cdots, h_{t-n} \big).
\end{align*}
		 
\vspace{2em}

\textbf{Case 4. $\alpha_{1+, 1-, +}$ and $\beta_{1+, 1-, +}$}		

Let $\widehat{\varphi}_x: \overline{C_G(x)}^n\to k \in C^n(C_G(x),k)$ and $\widehat{\alpha}_y =\left(h_{1, t}\right)\in C_t\left(C_G(y), k\right)$ with $x,  y \in X$, $n, t \ge 0$ and $n-t-1 \ge 0$. We want to compute
\begin{align*}
    \alpha_{1+, 1-, +} (\widehat{\varphi}_{x}, \widehat{\alpha}_y)= \widehat{m}_2 (\widehat{\varphi}_{x}, \widehat{\alpha}_y)
    &= \iota^{n-t-1} \big(\rho^{x,n} (\widehat{\varphi}_x) \cup \iota_{y,t} (\widehat{\alpha}_y) \big)\\
    &= \sum_{z\in X} \iota^{z,n-t-1} \big( \rho^{x,n} (\widehat{\varphi}_x) \cup \iota_{y,t} (\widehat{\alpha}_y) \big)_z.
\end{align*}

Firstly, we compute 
\[
\rho^{x,n} (\widehat{\varphi}_x) \cup \iota_{y,t} (\widehat{\alpha}_y) \in C^{n-t-1} (kG, kG) =\bigoplus_{z\in X}\mathcal{H}^{z,n-t-1},
\]
for $(g_{1,n-t-1}) \in \overline{G}^{\times n-t-1} $,
\begin{align*}
\big( \rho^{x,n} (\widehat{\varphi}_x) \cup \iota_{y,t} (\widehat{\alpha}_y)\big) (g_{1, n-t-1})  
= &\sum_{g \in G} \varphi_x \left(g_{1, n-t-1}, g^{-1}, h_{1, t}\right)h^{-1}_t\cdots h_1^{-1} y g\\
=&\sum_{g \in G} \sum_{i=1}^{n_x} \widehat{\varphi}_x (h_{i_g,1}^x,\cdots, h_{i_g,n}^x ) x_i g_1\cdots g_{n-t-1} g^{-1} y g \\
= &\sum_{z\in X} \sum_{j=1}^{n_z}\sum_{\left(g, i\right)\in I_{z_j}} \widehat{\varphi}_x (h_{i_g,1}^x,\cdots, h_{i_g,n}^x ) z_j g_1g_2\cdots g_{n-t-1}
\end{align*}
where
\begin{itemize}
    \item for any $z\in X$, $g\in G$ and $i\in \{1,\cdots, n_x \}$, $\{h_{i_g,1}^x,\cdots, h_{i_g, n}^x \}\subset \overline{C_G(x)}$ is defined as follows
    \[
    \{h_{i_g,1}^x,\cdots, h_{i_g, n}^x \}=\spadesuit_{x,i} \{ g_1,\cdots, g_{n-t-1}, g^{-1}, h_1,\cdots, h_t \},
    \]

    \item $I_{z_j}:= \{(g, i)\ |\ g \in G , 1\le i\le n_x, x_i g_1 \cdots g_{n-t-1} g^{-1} yg (g_1\cdots g_{n-t-1})^{-1}=z_j \}$.
\end{itemize}    
Then for $(g_{1,n-t-1})\in \overline{C_G(z)}^{\times n-t-1} $, we obtain 
\[
\iota^{z,n-t-1} \big( \rho^{x,n} (\widehat{\varphi}_x) \cup \iota_{y,t} (\widehat{\alpha}_y) \big)_z (g_{1,n-t-1})= \sum_{\left(g, i\right)\in I_z} \widehat{\varphi}_x (h_{i_g,1}^x,\cdots, h_{i_g,n}^x )
\]
with $I_z=I_{z_1}$.

\vspace{2em}  

Now we calculate $\beta_{1+, 1-, +}$. 
\begin{align*}
\beta_{1+, 1-, +} (\widehat{\varphi}_{x}, \widehat{\alpha}_y)
&= s^{n-t-1} \big(\rho^{x,n} (\widehat{\varphi}_x) \cup \iota_{y,t} (\widehat{\alpha}_y) \big)\\
&= \sum_{z\in X} s^{z,n-t-1} \big( \rho^{x,n} (\widehat{\varphi}_x) \cup \iota_{y,t} (\widehat{\alpha}_y) \big)_z.
\end{align*}

For any $z\in X$ and $(g_{1,n-t-2})\in \overline{G}^{n-t-2} $, by definition of $s^{z,n-t-1} $
\begin{align*}
s^{z,n-t-1} \big( \rho^{x,n} (\widehat{\varphi}_x) \cup \iota_{y,t} (\widehat{\alpha}_y) \big)_z (g_{1,n-t-2})
&= \sum\limits_{l=0}^{n-t-2}\sum\limits_{a=1}^{n_z}(-1)^l e_{a,l}^1 z_a g_1\cdots g_{n-t-2},
\end{align*}
where $e_{a,l}^1$ is dertermined by the following equation
\[
\big(  \rho^{x,n} (\widehat{\varphi}_x) \cup \iota_{y,t} (\widehat{\alpha}_y) \big) _z (h_{a,1}^z,\cdots,h_{a,l}^z,\gamma_{s_a^l,z}, g_{l+1},\cdots,g_{n-t-2})  
g_{n-t-2}^{-1} \cdots g_1^{-1} \gamma_{a,z}^{-1}=\sum\limits_{t=1}^{n_z} e_{a,l}^t z_t,
\]
and from the computation of $\alpha_{1+, 1-, +}$, we have
\begin{align*}
    e_{a,l}^1 
    &= \iota^{z,n-t-1} \big( \rho^{x,n} (\widehat{\varphi}_x) \cup \iota_{y,t} (\widehat{\alpha}_y) \big)_z  (h_{a,1}^z,\cdots,h_{a,l}^z,\gamma_{s_a^l,z}, g_{l+1},\cdots,g_{n-t-2})\\
    &=  \sum_{\left(g, i\right)\in I_z} \widehat{\varphi}_x (h_{i_g,1}^{x,a,l},\cdots, h_{i_g,n}^{x,a,l} ) 
\end{align*}
where for each $a\in \{1,\cdots,n_z \}$ and $l\in \{1,\cdots, n-t-2\}$,

\begin{itemize}
    \item $\{ h_{a,1}^z,\cdots, h_{a,n-t-2}^z \} =\spadesuit_{z,a} \{ g_1,\cdots, g_{n-t-2} \}$; 
  
    \item $I_z := \{(g, i)\ |\ g \in G , 1\le i\le n_x, x_i \gamma_{a,l} g_1 \cdots g_{n-t-2} g^{-1} yg (\gamma_{a,l} g_1\cdots g_{n-t-2})^{-1}=z \}$;

    \item for any  $(g, i)\in I_z$,
    \[
    \{h_{i_g,1}^{x,a,l},\cdots, h_{i_g, n}^{x,a,l} \}=\spadesuit_{x,i}\{ h_{a,1}^z,\cdots,h_{a,l}^z,\gamma_{s_a^l,z}, g_{l+1},\cdots,g_{n-t-2}, g^{-1}, h_1, \cdots, h_t\}.
    \]
\end{itemize}

\vspace{2em}

\textbf{Case 5. $\alpha_{1-, 1+, -}$ and $\beta_{1-, 1+, -}$}

Let $\widehat{\alpha}_x =\left(g_{1, s}\right)\in C_s\left(C_G(x), k\right)$ and $\widehat{\varphi}_y: \overline{C_G(y)}^m\to k \in C^m(C_G(y),k)$ with $x,  y \in X$, $s, m \ge 0$ and $m-s-1 <0$. We want to compute
\begin{align*}
    \alpha_{1-, 1+, -} (\widehat{\alpha}_{x}, \widehat{\varphi}_y)= \widehat{m}_2 (\widehat{\alpha}_{x}, \widehat{\varphi}_y)
    &= \rho_{s-m} \big(\iota_{x,s} (\widehat{\alpha}_x) \cup \rho^{y,m} (\widehat{\varphi}_y) \big)\\
    &= \sum_{z\in X} \rho_{z,s-m} \big( \iota_{x,s} (\widehat{\alpha}_x) \cup \rho^{y,m} (\widehat{\varphi}_y) \big)_z.
\end{align*}

Firstly, we compute 
$  \iota_{x,s} (\widehat{\alpha}_x) \cup \rho^{y,m} (\widehat{\varphi}_y) \in C_{s-m}\left(kG, kG \right)$:
\[
\iota_{x,s} \left( \widehat{\alpha}_x \right) =\iota_{x,s} \left(g_{1, s}\right) =\left(g_s^{-1} \cdots g_1^{-1} x, g_{1, s} \right) \in \mathcal{H}_{x,s};
\]
\[
\rho^{y,m} (\widehat{\varphi}_y) :=\varphi_y:\overline{G} ^{\times m} \rightarrow kG, \quad \varphi_y \left(h_{1, m}\right) = \sum_{i=1}^{n_y} \widehat{\varphi}_y \left(k_{i,1}^y, \cdots k_{i,m}^y \right) y_i h_1 h_2\cdots h_m,
\]
where $k_{i,1}^y, \cdots, k_{i,m}^y \in \overline{C_G(y)}$ and 
$\{k_{i,1}^y, \cdots, k_{i,m}^y\} = \spadesuit_{y,i}\{ h_1,\cdots,h_m \}$;
\begin{align*}
\iota_{x,s} (\widehat{\alpha}_x) \cup \rho^{y,m} (\widehat{\varphi}_y)
&=\left(g_s^{-1} \cdots g_1^{-1} x\ \varphi_y (g_{1,m}),g_{m+1,s} \right) \\ 
&=\sum_{j=1}^{n_y}\left(g_s^{-1} \cdots g_1^{-1} x\ \widehat{\varphi}_y 
\left( h_{j,1}^y, \cdots, h_{j,m}^y \right) y_j g_1\cdots g_m, g_{m+1, s}\right)\\
&=\sum_{z\in X}\sum_{j\in I_z}\widehat{\varphi}_y \left( h_{j,1}^y, \cdots, h_{j,m}^y \right) \left(g_s^{-1} \cdots g_1^{-1} \phi(y_j)^{-1} z \phi(y_j) g_1\cdots g_m, g_{m+1, s}\right)
\end{align*}
with  $\{h_{j,1}^y, \cdots, h_{j,m}^y\} = \spadesuit_{y,j}\{ g_1,\cdots,g_m \}$; then
\begin{align*}
\rho_{s-m} \big(\iota_{x,s} (\widehat{\alpha}_x) \cup \rho^{y,m} (\widehat{\varphi}_y) \big)
&=\sum_{z\in X} \rho_{z,s-m} \big(\iota_{x,s} (\widehat{\alpha}_x) \cup \rho^{y,m} (\widehat{\varphi}_y) \big)\\
&=  \sum_{z\in X} \sum_{j\in I_z} \widehat{\varphi}_y \left( h_{j,1}^y, \cdots, h_{j,m}^y \right)  \rho_{s-m,z} ( g_s^{-1} \cdots g_1^{-1} \phi(y_j)^{-1} z \phi(y_j) g_1\cdots g_m, g_{m+1, s} ) \\
&=  \sum_{z\in X} \sum_{j\in I_z} \widehat{\varphi}_y \left( h_{j,1}^y, \cdots, h_{j,m}^y \right) ( h_{i_j,1}^z, \cdots, h_{i_j,s-m}^z )
\end{align*}
where 
\begin{itemize}
    \item for any $z\in X$, $I_z$ is defined by 
    \[
    I_z:= \left\{j\ |\ 1\le j\le n_y,\  xy_j=\phi(y_j)^{-1} z \phi(y_j) \text{ for some } \phi(y_j) \in G \right\};
    \]

    \item $ h_{i_j,1}^z, \cdots, h_{i_j,s-m}^z\in \overline{C_G(z)}$ are uniquely determined by $\Phi (y_j)g_1\cdots g_m=h \gamma _{i_j, z} \in \overline{C_G(z)}\gamma _{i_j, z}$ and 
    \[
    \{  h_{i_j,1}^z, \cdots, h_{i_j,s-m}^z \}=\spadesuit_{z,i_j} \{ g_{m+1},\cdots, g_s \}.
    \]
\end{itemize}

\vspace{2em}  

Now we calculate $\beta_{1-, 1+, -}$. 
\begin{align*}
    \beta_{1-, 1+, -} (\widehat{\alpha}_{x}, \widehat{\varphi}_y)
    &= s_{s-m} \big(\iota^{x,s} (\widehat{\alpha}_x) \cup \rho_{y,m} (\widehat{\varphi}_y) \big)\\
    &= \sum_{z\in X} s_{z,s-m} \big(\iota^{x,s} (\widehat{\alpha}_x) \cup \rho_{y,m} (\widehat{\varphi}_y) \big)_z.
\end{align*}
For any $z\in X$, by deinition of $s_{z, s-m}: \mathcal{H}_{z, s-m} \rightarrow \mathcal{H}_{z, s-m+1}$, we have

\begin{align*}
  s_{z,s-m} \big(\iota^{x,s} (\widehat{\alpha}_x) &\cup \rho_{y,m} (\widehat{\varphi}_y) \big)_z
 =\sum_{j\in I_z} \sum_{u=0}^{s-m} (-1)^u \widehat{\varphi}_y \left( h_{j,1}^y, \cdots, h_{j,m}^y \right)\\
  &\big( (h_{i_j,1}^z \cdots h_{i_j,u}^z \gamma_{s_{i_j}^u,z}\ g_{m+u+1} \cdots  g_s )^{-1} z,\ h_{i_j,1}^z, \cdots, h_{i_j,u}^z, \gamma_{s_{i_j}^u,z},g_{m+u+1},\cdots, g_s \big).
\end{align*}

\vspace{2em}

\textbf{Case 6. $\alpha_{1-, 1+, +}$ and $\beta_{1-, 1+, +}$}		

Let $\widehat{\alpha}_x =\left(g_{1, s}\right)\in C_s\left(C_G(x), k\right)$ and $\widehat{\varphi}_y: \overline{C_G(y)}^m\to k \in C^m(C_G(y),k)$ with $x,  y \in X$, $s, m \ge 0$ and $m-s-1 \ge 0$. We want to compute
\begin{align*}
    \alpha_{1-, 1+, +} (\widehat{\alpha}_{x}, \widehat{\varphi}_y)= \widehat{m}_2 (\widehat{\alpha}_{x}, \widehat{\varphi}_y)
    &= \iota^{m-s-1} \big(\iota_{x,s} (\widehat{\alpha}_x) \cup \rho^{y,m} (\widehat{\varphi}_y) \big)\\
    &= \sum_{z\in X} \iota^{z,m-s-1} \big( \iota_{x,s} (\widehat{\alpha}_x) \cup \rho^{y,m} (\widehat{\varphi}_y) \big)_z.
\end{align*}

Firstly, we compute 
\[
\iota_{x,s} (\widehat{\alpha}_x) \cup \rho^{y,m} (\widehat{\varphi}_y) \in C^{m-s-1} (kG, kG) =\bigoplus_{z\in X}\mathcal{H}^{z,m-s-1},
\]
for $(h_{1,m-s-1}) \in \overline{G}^{\times m-s-1} $,
\begin{align*}
\big(\iota_{x,s} (\widehat{\alpha}_x) \cup \rho^{y,m} (\widehat{\varphi}_y)\big) (h_{1, m-s-1})  
= &\sum_{g \in G} g g_s^{-1}\cdots g_1^{-1} x\ \varphi_y \left(g_{1, s}, g^{-1}, h_{1, m-s-1}\right) \\
=&\sum_{g \in G} \sum_{j=1}^{n_y} \widehat{\varphi}_y (h_{j_g,1}^y,\cdots, h_{j_g,m}^y )  g g_s^{-1}\cdots g_1^{-1} x y_j g_1\cdots g_s g^{-1} h_1\cdots h_{m-s-1} \\
= &\sum_{z\in X} \sum_{i=1}^{n_z}\sum_{\left(g, j\right)\in I_{z_i}}  \widehat{\varphi}_y (h_{j_g,1}^y,\cdots, h_{j_g,m}^y ) z_i h_1 h_2\cdots h_{m-s-1}
\end{align*}
where
\begin{itemize}
    \item for any $z\in X$, $g\in G$ and $j\in \{1,\cdots, n_x \}$, $\{h_{j_g,1}^y,\cdots, h_{j_g, m}^y \}\subset \overline{C_G(y)}$ is defined as follows
    \[
    \{h_{j_g,1}^y,\cdots, h_{j_g, n}^y \}=\spadesuit_{y,j} \{ g_1,\cdots, g_s, g^{-1}, h_1,\cdots, h_{m-s-1} \},
    \]

    \item $I_{z_i}:= \{(g, j)\ |\ g \in G , 1\le j\le n_y,\ g g_s^{-1}\cdots g_1^{-1} x y_j g_1\cdots g_s g^{-1}=z_i \}$.
\end{itemize}    
Then for $(h_{1,m-s-1})\in \overline{C_G(z)}^{\times m-s-1} $, we obtain 
\[
\iota^{z,m-s-1} \big(\iota_{x,s} (\widehat{\alpha}_x) \cup \rho^{y,m} (\widehat{\varphi}_y)\big) _z (h_{1, m-s-1}) = \sum_{\left(g, j\right)\in I_z} \widehat{\varphi}_y (h_{j_g,1}^y,\cdots, h_{j_g,m}^y )
\]
with $I_z=I_{z_1}$.

\vspace{2em}  

Now we calculate $\beta_{1-, 1+, +}$. 
\begin{align*}
\beta_{1-, 1+, +} (\widehat{\alpha}_{x}, \widehat{\varphi}_y)
&= s^{m-s-1} \big(\iota_{x,s} (\widehat{\alpha}_x) \cup \rho^{y,m} (\widehat{\varphi}_y) \big)\\
&= \sum_{z\in X} s^{z,m-s-1} \big(\iota_{x,s} (\widehat{\alpha}_x) \cup \rho^{y,m} (\widehat{\varphi}_y) \big)_z.
\end{align*}

For any $z\in X$ and $(h_{1,m-s-2})\in \overline{G}^{m-s-2} $, by definition of $s^{z,m-s-1} $
\begin{align*}
s^{z,m-s-1} \big( \iota_{x,s} (\widehat{\alpha}_x) \cup \rho^{y,m} (\widehat{\varphi}_y) \big)_z (h_{1,m-s-2})
&= \sum\limits_{l=0}^{m-s-2}\sum\limits_{a=1}^{n_z}(-1)^l e_{a,l}^1 z_a h_1\cdots h_{n-t-2},
\end{align*}
where $e_{a,l}^1$ is dertermined by the following equation
\[
\big(  \iota_{x,s} (\widehat{\alpha}_x) \cup \rho^{y,m} (\widehat{\varphi}_y) \big) _z (h_{a,1}^z,\cdots,h_{a,l}^z,\gamma_{s_a^l,z}, h_{l+1},\cdots, h_{m-s-2})  
h_{m-s-2}^{-1} \cdots h_1^{-1} \gamma_{a,z}^{-1}=\sum\limits_{t=1}^{n_z} e_{a,l}^t z_t,
\]
and from the computation of $\alpha_{1-, 1+, +}$, we have
\begin{align*}
    e_{a,l}^1 
    &= \iota^{z,m-s-1} \big( \iota_{x,s} (\widehat{\alpha}_x) \cup \rho^{y,m} (\widehat{\varphi}_y) \big)_z  (h_{a,1}^z,\cdots,h_{a,l}^z,\gamma_{s_a^l,z}, h_{l+1},\cdots, h_{m-s-2})\\
    &=  \sum_{\left(g, j\right)\in I_z} \widehat{\varphi}_y (h_{j_g,1}^{y,a,l},\cdots, h_{j_g,n}^{y,a,l} ) 
\end{align*}
where for each $a\in \{1,\cdots,n_z \}$ and $l\in \{1,\cdots, m-s-2\}$,

\begin{itemize}
    \item $\{ h_{a,1}^z,\cdots, h_{a,m-s-2}^z \} =\spadesuit_{z,a} \{ h_1,\cdots, h_{n-t-2} \}$; 
  
    \item $I_z:= \{(g, j)\ |\ g \in G , 1\le j\le n_y,\ g g_s^{-1}\cdots g_1^{-1} x y_j g_1\cdots g_s g^{-1}=z \}$;

    \item for any  $(g, j)\in I_z$,
    \[
    \{h_{j_g,1}^{y,a,l},\cdots, h_{j_g, m}^{y,a,l} \}=\spadesuit_{y,j}\{ h_{a,1}^z,\cdots,h_{a,l}^z,\gamma_{s_a^l,z}, h_{l+1},\cdots, h_{m-s-2}, g^{-1}, g_1, \cdots, g_s\}.
    \]
\end{itemize}

In order easier to calculate and easier for readers to read, based on the above calculation, we provide the following index graph for the calculation of $\widehat{m}_2$. We obtain the first main theorem of this paper.

\begin{Thm}\label{Thm: Main 1}
    Let $G$ be a finite group, the cup product $\widehat{m}_2$ on $\underset{x \in X}{\bigoplus} \widehat{C}^* (C_G(x), k)$ can be divided into following $6$ cases:
   \tikzset{every picture/.style={line width=0.75pt}} %set default line width to 0.75pt        
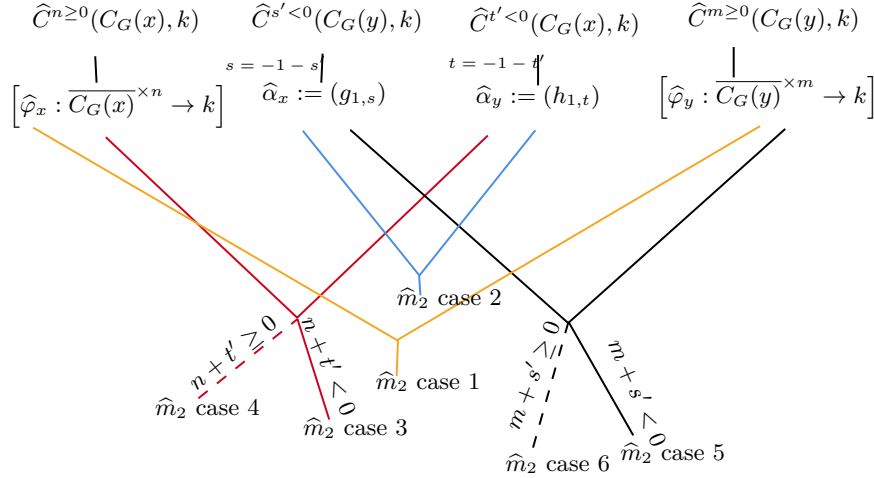
\begin{center}
\begin{tikzpicture}[
   x=0.5pt,  
   y=0.5pt,  
   yscale=-1, 
   xscale=1, 
   every node/.style={font=\small} 
   ]
	\draw  [color={rgb,  255:red,  208; green,  2; blue,  27 }  , draw opacity=1 ] (392, 117.59) -- (247.77, 255.49) -- (102.96, 118.85) ;
	\draw [color={rgb,  255:red,  208; green,  2; blue,  27 }  , draw opacity=1 ]   (248, 256.59) -- (272, 332.59) ;
	\draw   (617, 112.59) -- (452.78, 259.55) -- (287.95, 113.28) ;
	\draw  [color={rgb,  255:red,  74; green,  144; blue,  226 }  , draw opacity=1 ] (428, 113.59) -- (340.23, 223.58) -- (252, 113.96) ; 
	\draw  [dash pattern={on 4.5pt off 4.5pt}]  (453, 258.59) -- (426, 354.59) ;
	\draw  [color={rgb,  255:red,  245; green,  166; blue,  35 }  , draw opacity=1 ] (598, 110.44) -- (323.34, 272.53) -- (48, 111.59) ; 
	\draw [color={rgb,  255:red,  74; green,  144; blue,  226 }  , draw opacity=1 ]   (340, 223.59) -- (341, 238.59) ;
	\draw [color={rgb,  255:red,  245; green,  166; blue,  35 }  , draw opacity=1 ]   (324, 272.59) -- (323, 299.59) ;
	\draw [color={rgb,  255:red,  208; green,  2; blue,  27 }  , draw opacity=1 ] [dash pattern={on 4.5pt off 4.5pt}]  (173, 316.59) -- (248, 256.59) ;
	\draw    (453, 258.59) -- (502, 344.59) ;
	\draw    (95, 57.6) -- (96, 77.6) ;
	\draw    (267, 55.6) -- (266, 77.6) ;
	\draw    (430, 56.6) -- (430, 79.6) ;
	\draw    (578, 50.6) -- (578, 73.6) ;
\draw (219, 75) node [anchor=north west][inner sep=0.75pt] {$\widehat{\alpha }_x :=( g_{1,s})$};
\draw (380.67, 77) node [anchor=north west][inner sep=0.75pt]{$\widehat{\alpha }_y:=(h_{1,t})$};
\draw (27, 77) node [anchor=north west][inner sep=0.75pt]    {$\left[\widehat{\varphi }_{x} :\overline{C_{G} (x)}^{\times n}\rightarrow k\right]$};
\draw (517, 70) node [anchor=north west][inner sep=0.75pt]    {$\left[\widehat{\varphi }_{y} :\overline{C_{G} (y)}^{\times m}\rightarrow k\right]$};
\draw (45, 16) node [anchor=north west][inner sep=0.75pt] {$\widehat{C}^{n\ge 0}( C_G (x), k)$};
\draw (376, 18) node [anchor=north west][inner sep=0.75pt] {$\widehat{C}^{t'< 0}( C_G (x), k)$};
\draw (211, 15) node [anchor=north west][inner sep=0.75pt]  {$\widehat{C}^{s'< 0}( C_G (y), k)$};
\draw (541, 15) node [anchor=north west][inner sep=0.75pt]  {$\widehat{C}^{m\ge 0}( C_G (y), k)$};
\draw (305, 294) node [anchor=north west][inner sep=0.75pt]    {$\widehat{m}_2$ case 1};
\draw (321, 231) node [anchor=north west][inner sep=0.75pt]    {$\widehat{m}_2$ case 2};
\draw (159.25, 299.55) node [anchor=north west][inner sep=0.75pt] [rotate=-321.3]  {$n+t'\ge 0$};
\draw (265.2, 249.02) node [anchor=north west][inner sep=0.75pt]  [rotate=-66.9]  {$n+t'< 0$};
\draw (399.32, 337.49) node [anchor=north west][inner sep=0.75pt] [rotate=-291.28]  {$m+s'\ge 0$};
\draw (497.2, 266.02) node [anchor=north west][inner sep=0.75pt]  [rotate=-66.9]  {$m+s'< 0$};
\draw (250, 329) node [anchor=north west][inner sep=0.75pt]    {$\widehat{m}_2$ case 3};
\draw (138, 314) node [anchor=north west][inner sep=0.75pt]    {$\widehat{m}_2$ case 4};
\draw (490, 348) node [anchor=north west][inner sep=0.75pt]    {$\widehat{m}_2$ case 5};
\draw (403, 356) node [anchor=north west][inner sep=0.75pt]    {$\widehat{m}_2$ case 6};
\draw (360, 54) node [anchor=north west][inner sep=0.75pt]  [font=\tiny]  {$t=-1-t'$};
\draw (191, 55) node [anchor=north west][inner sep=0.75pt]  [font=\tiny]  {$s=-1-s'$};
\end{tikzpicture}
\end{center}
Specifically, the corresponding results are as follows:

$\bf{\widehat{m}_2}$ \textbf{case 1}
\[
\widehat{m}_2 \left(\widehat{\varphi}_x, \widehat{\varphi}_y\right)=\sum_{z\in X} \widehat{m}_2 \left(\widehat{\varphi}_x, \widehat{\varphi}_y\right)_z\in \bigoplus_{z\in X} C^{n+m}(C_G(z), k),
\]
with
\[
\widehat{m}_2 \left(\widehat{\varphi}_x, \widehat{\varphi}_y\right)_z (h_{1,n+m})= \sum_{(i, j) \in I_z} \widehat{\varphi}_x\left(h_{i,1}^x, \cdots, h_{i,n}^x\right) \widehat{\varphi}_{y}\left(h_{j,1}^y, \cdots, h_{j,m}^y\right),
\]
for $h_1,\cdots,h_{n+m}\in C_G(z) $, where 
\begin{itemize}
	\item $I_{z}=\left\{ (i, j) |\ 1\le i\le n_x,\ 1\le j\le n_y \text{ and }  x_i\left(h_1 \cdots h_n\right) y_j\left(h_1 \cdots h_n\right)^{-1}=z \right\}$,
    
	\item $h_{i,1}^x, \cdots, h_{i,n}^x \in \overline{C_G(x)}$ are determined by $\{h_1, \cdots h_n\}$, $x$, and $i$ from process $\spadesuit$, that is,
	\[
    \gamma_{i, x} h_1=h_{i,1}^x \gamma_{s_i^1, x}, \quad  \gamma_{s_i^1, x} h_2=h_{i,2}^x \gamma_{s_i^2, x}, \cdots, \gamma_{s_i^{n-1}, x} h_n=h_{i,n}^x \gamma_{s_i^n, x},
    \]
    we write $\{h_{i,1}^x, \cdots, h_{i,n}^x\}=\spadesuit_{x,i} \{h_1, \cdots h_n\}$.
    
	\item	$h_{j,1}^y, \cdots, h_{j,m}^y \in \overline{C_G(y)}$ and $\{h_{j,1}^y, \cdots, h_{j,m}^y \}=\spadesuit_{y,j}\{h_{n+1},\cdots,h_{n+m} \}$.
\end{itemize}

\vspace{1.5em}

$\bf{\widehat{m}_2}$ \textbf{case 2}
\[
\widehat{m}_2\left(\widehat{\alpha}_x,  \widehat{\alpha}_y\right) =  \sum_{z \in X} \sum_{g \in I_z}\left(k_{i_g,1}^z, \cdots, k_{i_g,s+t+1}^z\right)\in \underset{z\in Z}{\bigoplus} C_{s+t+1}\left(C_G(z), k\right),
\]
where
\begin{itemize}
    \item $I_{z}:=\left\{g \in G | h_1 \cdots h_t g^{-1} g_s^{-1} \cdots g_1^{-1} x g_1\cdots g_s g h_t^{-1} \cdots h_1^{-1} y=\Phi(g) ^{-1} z \Phi(g) \text{ for some } \Phi (g)\in G \right\}.$

    \item for any $g\in I_z$, $\Phi(g) \in C_G(z) \gamma_{i_g, z}$ and 
    \[
    \{ k_{i_g,1}^z, \cdots, k_{i_g,s+t+1}^z \}=\spadesuit_{z,i_g} \{h_1, \cdots, h_t, g^{-1} g_s^{-1} \cdots g_1^{-1} x, g_1, \cdots, g_s  \}.
    \]
\end{itemize}

\vspace{1.5em}

$\bf{\widehat{m}_2}$ \textbf{case 3}
\[
\widehat{m}_2 (\widehat{\varphi}_{x}, \widehat{\alpha}_y)
= \sum_{z\in X} \sum_{i\in I_z} \widehat{\varphi}_x \left( k_{i,1}^x, \cdots, k_{i,n}^x \right) ( h_{j_i,1}^z, \cdots, h_{j_i,t-n}^z )
\]
where 
\begin{itemize}
    \item $\{k_{i,1}^x, \cdots, k_{i,n}^x\} = \spadesuit_{x,i}\{ h_{t-n+1},\cdots,h_t \}$;

    \item for any $z\in X$, $I_z$ is defined by 
    \[
    I_z:= \left\{i\ |\ 1\le i\le n_x,\  h_1 \cdots h_{t-n} x_i  h_{t-n}^{-1} \cdots h_1^{-1} y=\phi(x_i)^{-1} z \phi(x_i) \text{ for some } \phi(x_i) \in G \right\};
    \]

    \item $ h_{j_i,1}^z, \cdots, h_{j_i,t-n}^z\in \overline{C_G(z)}$ are uniquely determined by $\Phi (x_i)=h \gamma _{j_i, z} \in \overline{C_G(z)}\gamma _{j_i, z}$ and 
    \[
    \{  h_{j_i,1}^z, \cdots, h_{j_i,t-n}^z \}=\spadesuit_{z,j_i} \{ h_1,\cdots, h_{t-n} \}.
    \]
\end{itemize}

\vspace{1.5em}

$\bf{\widehat{m}_2}$ \textbf{case 4}
\[
\widehat{m}_2 (\widehat{\varphi}_{x}, \widehat{\alpha}_y)
= \sum_{z\in X} \iota^{z,n-t-1} \big( \rho^{x,n} (\widehat{\varphi}_x) \cup \iota_{y,t} (\widehat{\alpha}_y) \big)_z \in \bigoplus_{z\in X} C^{n-t-1}(C_G(z), k),
\]
and for $(g_{1,n-t-1})\in \overline{C_G(z)}^{\times n-t-1} $, we obtain
\[
\iota^{z,n-t-1} \big( \rho^{x,n} (\widehat{\varphi}_x) \cup \iota_{y,t} (\widehat{\alpha}_y) \big)_z (g_{1,n-t-1})= \sum_{\left(g, i\right)\in I_z} \widehat{\varphi}_x (h_{i_g,1}^x,\cdots, h_{i_g,n}^x )
\]
where
\begin{itemize}
    \item $I_z:= \{(g, i)\ |\ g \in G , 1\le i\le n_x, x_i g_1 \cdots g_{n-t-1} g^{-1} yg (g_1\cdots g_{n-t-1})^{-1}=z \}$,

    \item for any $z\in X$, $(g,i)\in I_z$, $\{h_{i_g,1}^x,\cdots, h_{i_g, n}^x \}\subset \overline{C_G(x)}$ is defined as follows
    \[
    \{h_{i_g,1}^x,\cdots, h_{i_g, n}^x \}=\spadesuit_{x,i} \{ g_1,\cdots, g_{n-t-1}, g^{-1}, h_1,\cdots, h_t \}.
    \]   
\end{itemize}    

\vspace{1.5em}

$\bf{\widehat{m}_2}$ \textbf{case 5}
\[
\widehat{m}_2 (\widehat{\alpha}_{x}, \widehat{\varphi}_y)=  \sum_{z\in X} \sum_{j\in I_z} \widehat{\varphi}_y \left( h_{j,1}^y, \cdots, h_{j,m}^y \right) ( h_{i_j,1}^z, \cdots, h_{i_j,s-m}^z ),
\]
where 
\begin{itemize}
    \item for any $z\in X$, $I_z$ is defined by 
    \[
    I_z:= \left\{j\ |\ 1\le j\le n_y,\  xy_j=\phi(y_j)^{-1} z \phi(y_j) \text{ for some } \phi(y_j) \in G \right\};
    \]

    \item for any $j\in I_z$, $\{h_{j,1}^y, \cdots, h_{j,m}^y\} = \spadesuit_{y,j}\{ g_1,\cdots,g_m \}$;

    \item $ h_{i_j,1}^z, \cdots, h_{i_j,s-m}^z\in \overline{C_G(z)}$ are uniquely determined by $\Phi (y_j)g_1\cdots g_m=h \gamma _{i_j, z} \in \overline{C_G(z)}\gamma _{i_j, z}$ and 
    \[
    \{  h_{i_j,1}^z, \cdots, h_{i_j,s-m}^z \}=\spadesuit_{z,i_j} \{ g_{m+1},\cdots, g_s \}.
    \]
\end{itemize}

\vspace{1.5em}

$\bf{\widehat{m}_2}$ \textbf{case 6}
\[
\widehat{m}_2 (\widehat{\alpha}_{x}, \widehat{\varphi}_y)
= \sum_{z\in X} \iota^{z,m-s-1} \big( \iota_{x,s} (\widehat{\alpha}_x) \cup \rho^{y,m} (\widehat{\varphi}_y) \big)_z,
\]
and for any $z\in X$ and $(h_{1,m-s-1})\in \overline{C_G(z)}^{\times m-s-1} $, 
\[
\iota^{z,m-s-1} \big(\iota_{x,s} (\widehat{\alpha}_x) \cup \rho^{y,m} (\widehat{\varphi}_y)\big) _z (h_{1, m-s-1}) = \sum_{\left(g, j\right)\in I_z} \widehat{\varphi}_y (h_{j_g,1}^y,\cdots, h_{j_g,m}^y )
\]
where
\begin{itemize}
    \item $I_z:= \{(g, j)\ |\ g \in G , 1\le j\le n_y,\ g g_s^{-1}\cdots g_1^{-1} x y_j g_1\cdots g_s g^{-1}=z \}$,

    \item for any $(g,j)\in I_z$, $\{h_{j_g,1}^y,\cdots, h_{j_g, m}^y \}\subset \overline{C_G(y)}$ is defined as follows
    \[
    \{h_{j_g,1}^y,\cdots, h_{j_g, n}^y \}=\spadesuit_{y,j} \{ g_1,\cdots, g_s, g^{-1}, h_1,\cdots, h_{m-s-1} \},
    \]
\end{itemize}   
\end{Thm}

\subsection{\texorpdfstring{$A_\infty$-structures on the addtive decomposition of the Hochschild cohomology at the complex level }{A-infinity on the addtive decomposition of the Hochschild cohomology at the complex level}}\label{section:Ainfinity-on-the-additive-decomposition}\

According to the homotopy transfer theorem, the operation associated with any planar binary tree with $n$ leaves in $\mathrm{PBT_n}$ is defined by assigning $\hat{\iota}^{*\ge 0}=\rho^*$ to the leaves, $\hat{s}^{*\ge 0}=s^* $ to the internal edges, $\widehat{m}_2$ to the internal vertices, and $\hat{\rho}^{*\ge 0} =\iota^*$ to the root. It follows that every binary subtree of the planar binary tree correspond to either $\alpha_{i+,j+,+} $ operation or $\beta_{i+,j+,+} $ ($i,j\in\{ 0,1 \}$). 

In particular, only the sub-binary trees at the lowest layer correspond to $\alpha$-type operation, while all other subtrees associated with $\beta$-type operations. Moreover, when $n=2$, the bottem layer subtree corresponds to the operation $\alpha_{1+,1++}$; otherwise, it corresponds to one of $\alpha_{0+,1+,+} $, $\alpha_{1+,0+,+} $, or $\alpha_{0+,0+,+} $. Furthermore, only the top layer corresponds to the operation $\beta_{1+,1+,+}$ (for $n\ge 3$), while all other layers between top and bottem layers correspond to one of the operation $\beta_{0+,1+,+} $, $\beta_{1+,0+,+}$ or $\beta_{0+,0+,+}$. Therefore, the operation associated with any n-leaf planar binary tree consisits of a sequence of $\beta$-operations, followed by a final $\alpha$-operation.

In this section, we discuss all the results of $\alpha_{i+, j+, +}$ and $\beta_{i+, j+, +}$ ($i, j\in \{0, 1\}$), then we prove the following main results in this section.

\begin{Thm}\label{Theorem: main theorem for cohomology}
The additive decomposition $\underset{x\in X}{\bigoplus}C^*(C_G(x),k)$ of the Hochschild cohomology complex $C^*(kG,kG)$ of group algebra $kG$ has an $A_{\infty}$-algebra structure. The formulas are, up to a sign, given by:
\begin{eqnarray*}
			m_1=\partial,\qquad\qquad\qquad\qquad\qquad\qquad \\
			m_n(\widehat{{\varphi}_1},\dots ,\widehat{{\varphi}_n})(g_{1,1}\dots ,g_{1,j_1},g_{2,1},\dots ,g_{2,j_2},\dots ,g_{n,1},\dots ,g_{n,j_n}) \\
			=\sum\pm\widehat{{\varphi}_1}(h_{1,1},\dots ,h_{1,j_1})\cdots \widehat{{\varphi}_n}(h_{n,1}\dots ,h_{n,j_n})
\end{eqnarray*}
for $n\geq 2$ and for $\widehat{\varphi}_s:{\overline{C_G(x_s)}}^{\times j_s}\rightarrow k$, $x_s\in X$, $s=1,2,\cdots,n$, and the sequence 
\[
(h_{1,1},\dots ,h_{1,j_1},\cdots,h_{n,1}\dots ,h_{n,j_n})
\]
is determined from the sequence (which lies in $\overline{C_G(v)}$ for some given $v\in X$) 
\[
(g_{1,1}\dots ,g_{1,j_1},\dots ,g_{n,1},\dots ,g_{n,j_n})
\]			
by combing a sequence of combined transformations of type $\alpha$ or type $\beta$ with the action $\iota^{*,v}$.
\end{Thm}

We will prove this theorem in Section~\ref{Section: alpha in cohomology}, and to prove this theorem, we need the following claim.
    
\begin{Claim}\label{Claim: beta result}
    Let $\widehat{\varphi}_s$, for $s=1,\cdots,n$, be the elements in $C^{j_s}(C_G(x_s), k) $, with $j_s\ge 0$ and $x_s\in X$. The result of performing $n-1$ times $\beta$-operations on them is, up to a sign, given by
    \[
    \beta(\widehat{{\varphi}_1},\dots ,\widehat{{\varphi}_n})(g_1,\dots ,g_{j_1+\cdots+j_n-n+1})
    \]
    \[
    =\sum_{z\in X}\sum_{a=1}^{n_z} \pm\widehat{{\varphi}_1}(h_{1,1},\dots ,h_{1,j_1})\cdots \widehat{{\varphi}_n}(h_{n,1}\dots ,h_{n,j_n})z_a g_1\dots g_{j_1+\cdots+j_n-n+1},
    \]
    and the suquence 
    \[
    (h_{1,1},\dots ,h_{1,j_1},\cdots,h_{n,1}\dots ,h_{n,j_n})
    \]
    is determined from the sequence 
    \[
    (g_1,\dots ,g_{j_1+\cdots+j_n-n+1})\in \overline{G}^{\times j_1+\cdots+ j_n-n+1 }
    \]
    by combing a sequence of combined transformation of type $\beta$ about $z,a$.
\end{Claim}

During our computation, we used process $\spadesuit$ many times. Recall that $\{ h_{i,1}^x, \cdots,  h_{i,n}^x\}=\spadesuit_{x,i}\{ g_1,\cdots, g_n\} $ means that the sequence $\{ h_{i,1}^x, \cdots,  h_{i,n}^x\}$ is determined by the sequence $\{ g_1,\cdots, g_n\} $, $x\in X$ and $i\in \{1,\cdots,n_x\}$ from the following process:
\[
\gamma_{i, x}g_1=h_{i,1}^x \gamma_{s_i^1, x}, \quad \gamma_{s_i^1, x}g_2=h_{i,2}^x \gamma_{s_i^2, x}, \ \cdots\ ,  
\gamma_{s_i^{n-1}, x}g_n=h_{i,n}^x \gamma_{s_i^n, x}.
\]

\subsubsection{\texorpdfstring{$\alpha_{i+,j+,+}$}{alpha (i+, j+,+) }}\label{Section: alpha in cohomology}\

(1) Consider the planar binary tree $\alpha_{1+, 1+, +}$,
\[
\begin{tikzcd}[every matrix/.style={minimum width=0.3em,  minimum height=0.2em}]
{\underset{x \in X}{\bigoplus}\widehat{C}^{\ge 0} (C_G(x), k)} \arrow[rdd,  phantom] \arrow[rdd,  "\hat{\iota}^*", no head] & & {\underset{x \in X}{\bigoplus}\widehat{C}^{\ge 0} (C_G(x), k)} \arrow[ldd, "\hat{\iota}^*", no head]  \\
& & & & \\
& m_2 \arrow[dd,  "\hat{\rho}^*",  no head] & {\mathcal{D}^{\ge 0}(k G,  k G)}  \\
& & &  {,} \\
& \alpha_{1+, 1+, +}                             
\end{tikzcd}
\]
we write this graph as $\alpha_{1+, 1+, +}(\widehat{\varphi}_x, \widehat{\varphi}_y)= \widehat{m}_2 (\widehat{\varphi}_x, \widehat{\varphi}_y)$.

Fix $z \in X$, by Theorem~\ref{Thm: Main 1}, for $(h_{1,n+m}) \in \overline{C_G(z) }^{\times (n+m)}$, we have
\[
\widehat{m}_2 \left(\widehat{\varphi}_x, \widehat{\varphi}_y\right)_z (h_{1,n+m})= \sum_{(i, j) \in I_z} \widehat{\varphi}_x\left(h_{i,1}^x, \cdots, h_{i,n}^x\right) \widehat{\varphi}_{y}\left(h_{j,1}^y, \cdots, h_{j,m}^y\right).
\]
We observe that the key step to obtain the result from this planar binary tree is to obtain sequence 
\[
\{ h_{i,1}^x, \cdots, h_{i,n}^x, h_{j,1}^y, \cdots, h_{j,m}^y \}
\]
from the sequence $\{ h_1, \cdots, h_{n+m} \}$.

Therefore, we can regard the above calculation process as, for the planar binary tree $\alpha_{1+, 1+, +}$, we input $n+m$ elements $h_1,\cdots,h_{n+m}\in \overline{C_G(z)} $, and output $n+m$ elements $h_{i,1}^x,\cdots, h_{i,n}^x\in \overline{C_G(x)} $ and $h_{j,1}^y,\cdots, h_{j,m}^y\in \overline{C_G(y)} $. We simplify the calculation process as follows:
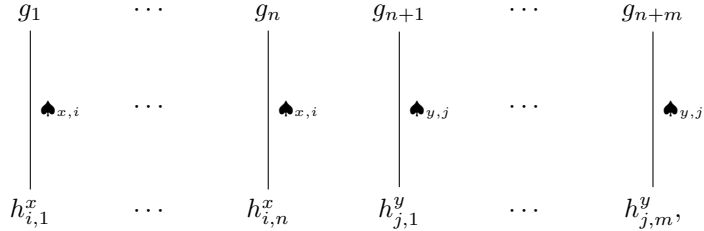
\[
\xymatrix{
	g_1\ar@{-}[dd]^{\spadesuit_{x, i}} & \cdots & g_n\ar@{-}[dd]^{\spadesuit_{x, i}}  &
	g_{n+1}\ar@{-}[dd] ^{\spadesuit_{y, j}} & \cdots & g_{n+m}\ar@{-}[dd]^{\spadesuit_{y, j}}\\
	\quad & \cdots & \quad & \quad & \cdots & \quad\\
	h_{i,1}^x & \cdots & h_{i,n}^x & h_{j,1}^y & \cdots & h_{j,m}^y, }
\]
the second row is determined by the first row by doing $\spadesuit$ process twice: $\{ h_{i,1}^x, \cdots, h_{i,n}^x\}=\spadesuit_{x,i}\{h_1, \cdots,h_n\}\subset \overline{C_G(x)} $ and $\{ h_{j,1}^y,\cdots, h_{j,m}^y \}= \spadesuit_{y,j}\{ h_{n+1},\cdots,h_{n+m}\} \subset \overline{C_G(y)}$, for fixed $(i,j)\in I_z$.

Thus the planar binary tree marked as type $\alpha_{1+, 1+, +} $ makes the sequence $\{ h_1,\cdots, h_n, h_{n+1},\cdots, h_{n+m} \}$ into the sequence $\{ h_{i,1}^x, \cdots, h_{i,n}^x, h_{j,1}^y, \cdots, h_{j,m}^y \}$, and we call this process $\alpha_{1+, 1+, +} $ transformation.

\vspace{1.5em}

(2) Now we consider the planar binary trees marked as type $\alpha_{0+, 1+, +}$, $\alpha_{1+, 0+, +}$ and $\alpha_{0+, 0+, +}$.

We now present a detailed analysis using case $\alpha_{0+,1+,+}$ as an example, while the other two cases can be derived analogously. Let $\varphi_x\in \mathcal{H}^{x,n} $ and $\widehat{\varphi}_y\in C^m(C_G(y),k)$, for $x,y\in X$ and $m,n\ge 0$,
\[
\alpha_{0+,1+,+}(\varphi_x, \widehat{\varphi}_y)= \iota^{n+m}( \varphi_x\cup \rho^{y,m}(\widehat{\varphi}_y ) ),
\]
here, the element $\varphi_x$ at the $0$-end is the result of a preceding sequence of $\beta$-type operations. According to  Claim~\ref{Claim: beta result}, we assume that the output of the previous $\beta$-type operation is 
\[
  \varphi_x(g_1,\cdots,g_n)  =\sum_{i=1}^{n_x} \pm\widehat{{\varphi}_1}(h_1,\dots ,h_{i_1}) \widehat{{\varphi}_2}(h_{i_1+1},\dots ,h_{i_2}) \cdots \widehat{{\varphi}_k}(h_{i_{k-1}+1}\dots ,h_{n+k-1})x_i g_1\cdots g_n.
\]

For $z\in X$ and sequence $\{ g_1,\cdots,g_{n+m} \}\in \overline{C_G(z) }^{n+m}$, we have
\[
\iota^{z,n+m}( \varphi_x\cup \rho^{y,m}(\widehat{\varphi}_y ) )=\sum_{(i,j)\in I_z}\varphi_x (g_1,\cdots, g_n) \widehat{\varphi}_y( h_{j,1}^y,\cdots,h_{j,m}^y )y_j g_n^{-1}\cdots g_1^{-1}z^{-1}.
\]
Thus, using the formula for $\varphi_x$, we obtain
\[
\iota^{z,n+m}( \varphi_x\cup \rho^{y,m}(\widehat{\varphi}_y ) )=\sum_{(i,j)\in I_z} \pm\widehat{{\varphi}_1}(h_1,\dots ,h_{i_1}) \widehat{{\varphi}_2}(h_{i_1+1},\dots ,h_{i_2}) \cdots \widehat{{\varphi}_k}(h_{i_{k-1}+1}\dots ,h_{n+k-1}),
\]
where $I_z=\{ (i,j)\ |\ 1\le i\le n_x, 1\le j\le n_y, x_i g_1\cdots g_n y_j g_n^{-1}\cdots g_1^{-1}=z \}$ and 
\[
\{ h_{j,1}^y,\cdots,h_{j,m}^y \}=\spadesuit_{y,j}\{ g_{n+1}, \cdots,g_{n+m} .\} 
\]
The other two cases are similar, then we prove Theorem~\ref{Theorem: main theorem for cohomology}.

Similar to the transformation $\alpha_{1+,1+,+}$, we now present the transformation of $\alpha_{0+,1+,+}$. For the sake of convenience in the following discussion, we denote the input sequence at the 0-end as $\{ g_1,\cdots,g_n \}$, and the input sequence at the 1-end as $\{ g_{n+1},\cdots,g_{n+m} \}$. For a fixed $z\in X$ and $(i,j)\in I_z$, we abbreviate this computation process as ($\alpha_{0+,1+,+}$-transformation):
\[
\xymatrix{
	g_1\ar@{-}[dd]^{\mathrm{id}} & \cdots & g_n\ar@{-}[dd]^{\mathrm{id}}  &
	g_{n+1}\ar@{-}[dd] ^{\spadesuit_{y, j}} & \cdots & g_{n+m}\ar@{-}[dd]^{\spadesuit_{y, j}}\\
	\quad & \cdots & \quad & \quad & \cdots & \quad\\
	g_1 & \cdots & g_n & h_{j,1}^y & \cdots & h_{j,m}^y, }
\]
the second row is determined by the first: elements on the 0-end remain unchanged, while those on the 1-end are obtained via $\spadesuit_{y,j}$-transformation, i.e., $\{ h_{j,1}^y,\cdots,h_{j,m}^y \}=\spadesuit_{y,j}\{ g_{n+1}, \cdots,g_{n+m} \} $.

In the same manner, we define the $\alpha_{1+,0+,+}$-transformation:

Let $\{g_1,\dots,g_n\}$ be the input at the 1-end and $\{g_{n+1},\dots,g_{n+m}\}$ the input at the 0-end. For a fixed $z\in X$ and $(i,j)\in I_z$, we abbreviate this computation process as ($\alpha_{1+,0+,+}$-transformation):
\[
\xymatrix{
	g_1\ar@{-}[dd]^{\spadesuit_{x,i}} & \cdots & g_n\ar@{-}[dd]^{\spadesuit_{x,i}}  &
	g_{n+1}\ar@{-}[dd] ^{\mathrm{id}} & \cdots & g_{n+m}\ar@{-}[dd]^{\mathrm{id}}\\
	\quad & \cdots & \quad & \quad & \cdots & \quad\\
	h_{i,1}^x & \cdots & h_{i,n}^x & g_{n+1} & \cdots & g_{n+m}. }
\]

\subsubsection{\texorpdfstring{$\beta_{i+,j+,+}$}{beta (i+,j+,+)}}\

(1) Consider the planar binary tree marked as $\beta_{1+, 1+, +}$. 

Recall the result in Section~\ref{Section: computation for alpha and beta}, for fixed $z\in X$ and $( g_1,\cdots, g_{m+n-1} )\in \overline{G}^{\times m+n-1} $
\begin{align*}
    \beta_{1+, 1+, +} \left(\widehat{\varphi}_x, \widehat{\varphi}_y\right)_z (g_{1, n+m-1})
    &= s^{z,n+m} ( \rho^{x,n}(\widehat{\varphi}_x) \cup \rho^{y,m}(\widehat{\varphi}_y))_z (g_{1, n+m-1})\\
    &= \sum\limits_{l=0}^{m+n-1}\sum\limits_{a=1}^{n_z}(-1)^{l}e_{a, l}^1 z_a g_1\cdots   g_{n+m-1},
\end{align*}
with $e_{a, l}^1$ is, for each $l\in \{0,1,\ldots,m+n-1\}$ and $a\in \{1,\ldots,n_z\}$, is given by
\begin{eqnarray*}
e_{a, l}^1
&=& \iota^{z,n+m} {(\rho^{x,n}  \left(\widehat{\varphi}_x\right) \cup \rho^{y,m} \left(\widehat{\varphi}_y \right))_z} 
(h_{a, 1}^z, \dots , h_{a, l}^z, \gamma_{s_a^l, z}, g_{l+1}, \dots , g_{n+m-1})\\
&=& \alpha_{1+,1+,+}(\widehat{\varphi}_x, \widehat{\varphi}_y)_z (h_{a, 1}^z, \dots , h_{a, l}^z, \gamma_{s_a^l, z}, g_{l+1}, \dots , g_{n+m-1}) \\
&=&\sum\limits_{(i, j)\in I_z^{a,l} }\widehat{\varphi}_x(h_{i, 1}^x, \dots , h_{i, n}^x) \widehat{\varphi}_y (h_{j, 1}^y, \dots , h_{j, m}^y), 
\end{eqnarray*}
where
\begin{itemize}
    \item $(h_{a, 1}^z, \dots , h_{a, n+m-1}^z)=\spadesuit_{z, a}(g_1, \dots , g_{n+m-1})$,

    \item $I_z^{a,l}=\{ (i, j)|\ 1\le i\le n_x, 1\le j\le n_y \text{ and } x_i(g_1^{a,l} \cdots  g_n^{a,l} )y_j(g_1^{a,l} \cdots  g_n^{a,l})^{-1}=z \}$,  here we rewrite 
    \[
    (g_1^{a,l} , g_2^{a,l} , \cdots, g_{n+m}^{a,l} )=(h_{a, 1}^z, \dots , h_{a, l}^z, \gamma_{s_a^l, z}, g_{l+1}, \dots , g_{n+m-1}),
    \]
    and for each $(i,j)\in I_z^{a,l}$,

    \item $\{ h_{i, 1}^x, \dots , h_{i, n}^x \}=\spadesuit_{x,i}\{ g_1^{a,l}, \cdots, g_n^{a,l}\}$,

    \item $\{ h_{j, 1}^y, \dots , h_{j, m}^y \}=\spadesuit_{y,j}\{ g_{n+1}^{a,l}, \cdots, g_{n+m}^{a,l}\}$.
\end{itemize}
From the analysis of $\alpha_{1+1+,+}$, it is evident that the computation of $e_{a,l}^1$ corresponds to applying transformation $\alpha_{1+,1+,+}$ to an input of $\{ h_{a, 1}^z, \dots , h_{a, l}^z, \gamma_{s_a^l, z}, g_{l+1}, \dots , g_{n+m-1} \}$. Hence, the computation of $\beta_{1+1+,+} $ can be concisely illustrated by the following diagram:
\[
\xymatrix@R=1.0pc@C=1.0pc{
g_1\ar@{-}[dd] ^{\spadesuit_{z, a}} &
\cdots &
g_l\ar@{-}^{\spadesuit_{z, a}}[dd] \ar@{-}[ddr] &
g_{l+1}\ar@{-}[ddr]^{id} &
\cdots &
g_{n+m-1}\ar@{-}[rdd]^{id} &
\quad\\
\quad & \cdots & \quad & \quad & \cdots & \cdots & \quad \\
h_{a, 1}^z & \cdots & h_{a, l}^z & \gamma_{s_a^l,z} & g_{l+1} & \cdots & g_{n+m-1}\\
g_1^{a,l} & \cdots & g_l^{a,l} & g_{l+1}^{a,l} & g_{l+2}^{a,l} & \cdots & g_{n+m}^{a,l} \\
}
\]
\[
\xymatrix@R=1.0pc@C=1.0pc{
g_1^{a,l} \ar@{-}[dd] ^{\spadesuit_{x, i}} &
\cdots &
g_{n}^{a,l} \ar@{-}[dd] ^{\spadesuit_{x, i}}&
g_{n+1}^{a,l} \ar@{-}[dd] ^{\spadesuit_{y, j}}&
\cdots &
g_{n+m}^{a,l} \ar@{-}[dd]^{\spadesuit_{y, j}}\\
\quad & \cdots & \quad & \quad & \cdots & \quad \\
h_{i, 1}^x &
\cdots &
h_{i, n}^x &
h_{j,1}^y &
\cdots &
h_{j,m}^y
}
\]
This diagram can be understood in three steps:

\textbf{Step 1.} The second row is determined by the first row, in accordance with the explicit formula of the homotopy $s^*$. Concretly, $\{ h_{a,1}^z,\cdots, h_{a,l}^z \}=\spadesuit_{z,a}\{ g_1, \cdots,  g_l\} $, and an element $\gamma_{s_a^l, z} $ at the $l+1$-th position. We denote this transformation as an $s^l$-transformation.

\textbf{Step 2.} We relabel the elements in the second row as $\{ g_1^{a,l},\cdots,g_{n+m}^{a,l} \}$. This relabeling does not change the elements themselves, but only their notation. The purpose of this step is to avoid having to distinguish between different cases depending on the value of $l$ and $m$.

\textbf{Step 3.} The third row is derived from the second row by applying the transformation $ \alpha_{1+,1+,+} $ to the sequence $\{ g_1^{a,l},\cdots,g_{n+m}^{a,l} \}$. Concretely, $\{ h_{i, 1}^x, \dots , h_{i, n}^x \}=\spadesuit_{x,i}\{ g_1^{a,l}, \cdots, g_n^{a,l}\}$
and $\{ h_{j, 1}^y, \dots , h_{j, m}^y \}=\spadesuit_{y,j}\{ g_{n+1}^{a,l}, \cdots, g_{n+m}^{a,l}\}$.

Hence, this transformation maps sequence
\[
\{ g_1,\cdots,g_{n+m-1} \}
\]
to the sequence 
\[
\{ h_{i, 1}^x, \dots , h_{i, n}^x,  h_{j, 1}^y, \dots , h_{j, m}^y \},
\]
and we denote it as \textbf{$\beta_{1+,1+,+}$-transformation}.

(2) Now we compute $ \beta_{0+,1+,+} $, $\beta_{1+,0+,+}$ and $\beta_{0+,0+,+}$.

For fixed $z\in X$ and $( g_1,\cdots, g_{m+n-1} )\in \overline{G}^{\times m+n-1} $
\begin{align*}
    \beta_{0+, 1+, +} \left(\varphi_x, \widehat{\varphi}_y\right)_z (g_{1, n+m-1})
    &= s^{z,n+m} ( \varphi_x \cup \rho^{y,m}(\widehat{\varphi}_y))_z (g_{1, n+m-1})\\
    &= \sum\limits_{l=0}^{m+n-1}\sum\limits_{a=1}^{n_z}(-1)^{l}e_{a, l}^1 z_a g_1\cdots   g_{n+m-1},
\end{align*}
with $e_{a, l}^1$ is, for each $l\in \{0,1,\ldots,m+n-1\}$ and $a\in \{1,\ldots,n_z\}$, is given by
\begin{eqnarray*}
e_{a, l}^1
&=& \iota^{z,n+m} {(\varphi_x \cup \rho^{y,m} \left(\widehat{\varphi}_y \right))_z}
(h_{a, 1}^z, \dots , h_{a, l}^z, \gamma_{s_a^l, z}, g_{l+1}, \dots , g_{n+m-1})\\
&=&\alpha_{0+,1+,+} (\varphi_x,\widehat{\varphi}_y)_z (h_{a, 1}^z, \dots , h_{a, l}^z, \gamma_{s_a^l, z}, g_{l+1}, \dots , g_{n+m-1})\\
&=&\sum\limits_{(i, j)\in I_z^{a,l}} \varphi_x(g_1^{a, l}, \dots , g_n^{a, l}) \widehat{\varphi}_y (h_{j, 1}^y, \dots , h_{j, m}^y)y_j(g_1^{a,l}\cdots g_n^{a,l})^{-1} z^{-1}.
\end{eqnarray*}
Similar to $\beta_{1+,1+,+}$, we can illustrate $\beta_{0+1+,+}$-transformation as follows:
\[
\xymatrix@R=1.0pc@C=1.0pc{
g_1\ar@{-}[dd] ^{\spadesuit_{z, a}} &
\cdots &
g_l\ar@{-}^{\spadesuit_{z, a}}[dd] \ar@{-}[ddr] &
g_{l+1}\ar@{-}[ddr]^{id} &
\cdots &
g_{n+m-1}\ar@{-}[rdd]^{id} &
\quad\\
\quad & \cdots & \quad & \quad & \cdots & \cdots & \quad \\
h_{a, 1}^z & \cdots & h_{a, l}^z & \gamma_{s_a^l,z} & g_{l+1} & \cdots & g_{n+m-1}\\
g_1^{a,l} & \cdots & g_l^{a,l} & g_{l+1}^{a,l} & g_{l+2}^{a,l} & \cdots & g_{n+m}^{a,l} \\
}
\]
\[
\xymatrix@R=1.0pc@C=1.0pc{
g_1^{a,l} \ar@{-}[dd] ^{\mathrm{id}} &
\cdots &
g_n^{a,l} \ar@{-}[dd] ^{\mathrm{id}}&
g_{n+1}^{a,l} \ar@{-}[dd] ^{\spadesuit_{y, j}}&
\cdots &
g_{n+m}^{a,l} \ar@{-}[dd]^{\spadesuit_{y, j}}\\
\quad & \cdots & \quad & \quad & \cdots & \quad \\
g_1^{a,l} &
\cdots &
g_n^{a,l} &
h_{j,1}^y &
\cdots &
h_{j,m}^y
}
\]
This diagram can also interpreted in three steps. The first two are the same as in the $\beta_{1+1+,+}$-transformation, while in the third step, the operation used is replaced by $\alpha_{0+,1+,+}$.

Similarly, we define $\beta_{1+,0+,+}$-transformation:
\[
\xymatrix@R=1.0pc@C=1.0pc{
g_1\ar@{-}[dd] ^{\spadesuit_{z, a}} &
\cdots &
g_l\ar@{-}^{\spadesuit_{z, a}}[dd] \ar@{-}[ddr] &
g_{l+1}\ar@{-}[ddr]^{id} &
\cdots &
g_{n+m-1}\ar@{-}[rdd]^{id} &
\quad\\
\quad & \cdots & \quad & \quad & \cdots & \cdots & \quad \\
h_{a, 1}^z & \cdots & h_{a, l}^z & \gamma_{s_a^l,z} & g_{l+1} & \cdots & g_{n+m-1}\\
g_1^{a,l} & \cdots & g_l^{a,l} & g_{l+1}^{a,l} & g_{l+2}^{a,l} & \cdots & g_{n+m}^{a,l} \\
}
\]
\[
\xymatrix@R=1.0pc@C=1.0pc{
g_1^{a,l} \ar@{-}[dd] ^{\spadesuit_{x,i}} &
\cdots &
g_n^{a,l} \ar@{-}[dd] ^{\spadesuit_{x,i}}&
g_{n+1}^{a,l} \ar@{-}[dd] ^{\mathrm{id}}&
\cdots &
g_{n+m}^{a,l} \ar@{-}[dd]^{\mathrm{id}}\\
\quad & \cdots & \quad & \quad & \cdots & \quad \\
h_{i, 1}^x &
\cdots &
h_{i, n}^x &
g_{n+1}^{a,l} &
\cdots &
g_{n+m}^{a,l}
}
\]
this transformation means that we use $s^l$-transformation firstly, then $\alpha_{1+,0+,+}$-transformation.

Finally, we define $\beta_{0+,0+,+}$-transformation:
\[
\xymatrix@R=1.0pc@C=1.0pc{
g_1\ar@{-}[dd] ^{\spadesuit_{z, a}} &
\cdots &
g_l\ar@{-}^{\spadesuit_{z, a}}[dd] \ar@{-}[ddr] &
g_{l+1}\ar@{-}[ddr]^{id} &
\cdots &
g_{n+m-1}\ar@{-}[rdd]^{id} &
\quad\\
\quad & \cdots & \quad & \quad & \cdots & \cdots & \quad \\
h_{a, 1}^z & \cdots & h_{a, l}^z & \gamma_{s_a^l,z} & g_{l+1} & \cdots & g_{n+m-1}\\
g_1^{a,l} & \cdots & g_l^{a,l} & g_{l+1}^{a,l} & g_{l+2}^{a,l} & \cdots & g_{n+m}^{a,l} \\
}
\]
\[
\xymatrix@R=1.0pc@C=1.0pc{
g_1^{a,l} \ar@{-}[dd] ^{\mathrm{id}} &
\cdots &
g_n^{a,l} \ar@{-}[dd] ^{\mathrm{id}}&
g_{n+1}^{a,l} \ar@{-}[dd] ^{\mathrm{id}}&
\cdots &
g_{n+m}^{a,l} \ar@{-}[dd]^{\mathrm{id}}\\
\quad & \cdots & \quad & \quad & \cdots & \quad \\
g_1^{a,l} &
\cdots &
g_n^{a,l} &
g_{n+1}^{a,l} &
\cdots &
g_{n+m}^{a,l}
}
\]

\subsubsection{\texorpdfstring{Computation of $\widehat{m}_n$}{Computation of mn 1}}\

According to the analysis at the beginning of this section, we have Theorem~\ref{Theorem: main theorem for cohomology} and proved it in Section~\ref{Section: alpha in cohomology}. Now, we present a concrete example to help understand its content.

\begin{Ex}
    Let $\widehat{\varphi}_1: \overline{C_G(x)}^{\times m}\rightarrow k$,  $\widehat{\varphi}_2: \overline{C_G(y)}^{\times n}\rightarrow k$,  $\widehat{\varphi}_3: \overline{C_G(z)}^{\times p}\rightarrow k$,  $\widehat{\varphi}_4: \overline{C_G(u)}^{\times q}\rightarrow k$ for $x,y,z,u\in X$. Consider the following planar binary tree in $PBT_4$:
    
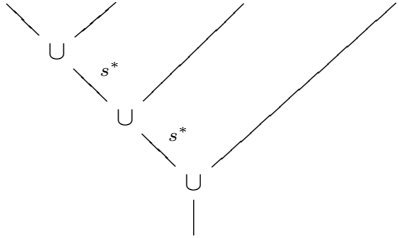
\[
    \xymatrix@R=1.0pc@C=1.0pc{
		\ar@{-}[dr]& & \ar@{-}[dl]& & \ar@{-}[ddll]& & & \ar@{-}[dddllll] \\
		& \cup\ar@{-}[dr]^{s^*} & & & & & & \\
		& & \cup\ar@{-}[dr]^{s^*} & & & & &\\
		& & & \cup\ar@{-}[d] & & & &\\
		& & & & & & & }
    \]
    The operation corresponding to this planar binary tree is a summation term of 
    \[
    \widehat{m}_4(\widehat{\varphi}_1, \widehat{\varphi}_2, \widehat{\varphi}_3, \widehat{\varphi}_4)
    (g_1,  \cdots,  g_{m-1},  g_m,  \cdots,  g_{m+n-1},  g_{m+n},  \cdots,  g_{m+n+p-1},  \cdots,  g_{m+n+p+q-2}),
    \]
    specifically, this operation involves performing the following transformations sequence:
    $\beta_{1+, 1+, +}\times \mathrm{id}\times \mathrm{id} \rightarrow \beta_{0+, 1+, +}\times \mathrm{id} \rightarrow \alpha_{0+, 1+, +}$, that is, the operation is 
    \[
    \alpha_{0+, 1+, +}( \beta_{0+, 1+, +}(\beta_{1+, 1+, +}(\widehat{\varphi}_1, \widehat{\varphi}_2), \widehat{\varphi}_3), \widehat{\varphi}_4).
    \]
    To compute this operation precisely, we inpute the oringinal sequence 
    \[
    (g_1,  \cdots,  g_{m-1},  g_m,  \cdots,  g_{m+n-1},  g_{m+n},  \cdots, g_{m+n+p-2},  g_{m+n+p-1},  \cdots,  g_{m+n+p+q-2}),
    \]
    the sequence changes as follows in order:
    \[
    (g_1,  \cdots,  g_{m-1},  g_m,  \cdots,  g_{m+n-1},  g_{m+n},  \cdots, g_{m+n+p-2},  h_{d,1}^u,\cdots,h_{d,q}^u ),
    \]
    after the first transformation $\alpha_{0+,1+,+}$, and 
    \[
    (h_{a, 1},  \cdots,  h_{a, l},  \gamma, g_{l+1}, \cdots,  g_{m+n-2},  h_{c, 1}^z,  \cdots,  h_{c, p}^z,  h_{d, 1}^u,  \cdots,  h_{d, q}^u)
    \]
    for $1\le l\le m+n-2$, after the second transformation $\beta_{0+,1+,+}$, and 
    \[
    (h_{a, 1}^x,  \cdots,  h_{a, m}^x,  h_{b, 1}^y,  \cdots, h_{b, n}^y,  h_{c, 1}^z,  \cdots,  h_{c, p}^z,  h_{d, 1}^u,  \cdots,  h_{d, q}^u).
    \]
   after the last transformation $\beta_{1+,1+,+}$, then we have the summation term correspond to this planar binary tree is 
   \[
   \sum \pm \widehat{\varphi}_1( h_{a,1}^x,\cdots,h_{a,m}^x ) \widehat{\varphi}_2( h_{b,1}^y,\cdots,h_{b,n}^y ) \widehat{\varphi}_3( h_{c,1}^z,\cdots,h_{c,p}^z ) \widehat{\varphi}_4( h_{d,1}^u,\cdots,h_{d,q}^u ).
   \]
\end{Ex}

\subsection{\texorpdfstring{$A_\infty$-structures on the addtive decomposition of the Hochschild homology at the complex level }{A-infinity on the addtive decomposition of the Hochschild homology at the complex level}}\label{section:Ainfinity-on-Hoch-homo}\

It is observed that during the computation of $\widehat{m}_n$, $\beta$ transformation appears multiple times whereas $\alpha$ transformation appears only once. To facilitate the computation, we begin by computating $\beta_{i-,j-,-}$.

\subsubsection{\texorpdfstring{$\beta_{i-,j-,-}$}{beta(i-,j-,-)}}\

We first calculate $\beta_{1-,1-,-}$.

Let $\widehat{\alpha}_x=( g_{1,s} )\in C_s( C_G(x),k )$ and $\widehat{\alpha}_y=( h_{1,t} )\in C_t( C_G(y),k )$, we have
\[  
\hat{\iota}^* \left(\widehat{\alpha}_x \right)=  \iota_{x,s}\left(\widehat{\alpha}_x \right) =\left(g_s^{-1} \cdots g_1 ^{-1} x,  g_{1,  s}\right) \in \mathcal{H}_{x,s} ;    
\]
\[ 
\hat{\iota}^* \left(\widehat{\alpha}_y\right)=\iota_{y,t} \left(\widehat{\alpha}_y\right)=\left(h_t^{-1} \cdots h_1^{-1} y,  h_{1,  t}\right) \in \mathcal{H}_{y,t}.
\]
By Case 2 in Section~\ref{Section: Tate-Hochschild Cohomology of a Group Algebra}, we know
\[
\hat{\iota}^* \left(\widehat{\alpha}_x \right)\cup \hat{\iota}^* \left(\widehat{\alpha}_y\right)= \sum_{g\in G} ( g h_t^{-1} \cdots h_1^{-1} y, h_{1,t}, g^{-1} g_s^{-1} \cdots g_1 ^{-1} x, g_{1,s} ).
\]
For each $z\in X$, we define 
\[
I_z:=\{ g\in G |\ h_{1} \cdots h_{t} g^{-1} g_{s}^{-1} \cdots g_{1}^{-1}  x  g_1\cdots  g_s g h_{t}^{-1} \cdots h_{1}^{-1} y=\phi(g)^{-1} z \phi(g) \text{ for some } \phi(g) \in G\},
\]
then we can write $\hat{\iota}^* \left(\widehat{\alpha}_x \right)\cup \hat{\iota}^* \left(\widehat{\alpha}_y\right)$ as 
\[
\sum_{z\in X} \sum_{g\in I_z} (g_s^{-1}\cdots g_1^{-1} x^{-1} g_1\cdots g_s g h_t^{-1} \cdots h_1^{-1} \phi(g)^{-1} z \phi(g), h_{1,t}, g^{-1} g_s^{-1} \cdots g_1 ^{-1} x, g_{1,s} ).
\]
For any $z\in X$ and $g\in I_z$, there exist $i_g$ such that $\phi(g)\in \overline{C_G(z)}\gamma_{i_g,z}$, assume
\[
\{ k_{i_g,1}^z,\cdots,k_{i_g,s+t+1}^z \}=\spadesuit_{z,i_g}\{ h_1,\cdots,h_t,g^{-1} g_s^{-1} \cdots g_1 ^{-1} x, g_1,\cdots, g_s \},
\]
 then we obtain
\begin{align*}
    &\hat{s}^*( \hat{\iota}^* \left(\widehat{\alpha}_x \right)\cup \hat{\iota}^* \left(\widehat{\alpha}_y\right) )\\
   = &\sum_{z\in X} \sum_{g\in I_z} s_{z,s+t+1}(g_s^{-1}\cdots g_1^{-1} x^{-1} g_1\cdots g_s g h_t^{-1} \cdots h_1^{-1} \phi(g)^{-1} z \phi(g), h_{1,t}, g^{-1} g_s^{-1} \cdots g_1 ^{-1} x, g_{1,s} )\\
    =& \sum_{z\in X} \sum_{g\in I_z} \sum_{j=0}^{s+t+1} (-1)^j ( (k_{i_g,1}^z\cdots k_{i_g,j}^z \gamma_{s_{i_g}^j,z} g_{j+1}'\cdots g_{s+t+1}')^{-1} z, k_{i_g,1}^z,\cdots,k_{i_g,j}^z, \gamma_{s_{i_g}^j,z}, g_{j+1}',\cdots,g_{s+t+1}' ),
\end{align*}
where we re-labeled $\{ h_1,\cdots,h_t,g^{-1} g_s^{-1} \cdots g_1 ^{-1} x, g_1,\cdots, g_s \}$ as $\{ g_1',\cdots,g_{s+t+1}' \}$ in the last step. It follows that, with respect to the 0-end, it suffices to consider the input element $\alpha_x:=( g_s^{-1}\cdots g_1^{-1} x, g_{1,s} )$. Following the computation of $\beta_{0-,1-,-}$, $ \beta_{1-,0-,-} $ and $ \beta_{0-,0-,-} $, we obtain that: for
\[  
\alpha_x=\left(g_s^{-1} \cdots g_1^{-1} x,  g_{1, s}\right) \in \mathcal{H}_{x,s};  
\]
\[ 
\alpha_y=\left(h_t^{-1} \cdots h_1^{-1} y, h_{1, t}\right) \in \mathcal{H}_{y,t};
\]
\[  
\widehat{\alpha}_x:=\left(g_{1, s}\right) \in C_s\left(C_G(x), k\right);
\]
\[ 
\widehat{\alpha}_y:=\left(h_{1, t}\right) \in   C_t\left(C_G(y), k\right),  
\]
we have 
\begin{align*}
	& \hat{s}^*( \alpha_{x} \cup \hat{\iota}^* \left(\widehat{\alpha}_y \right) )= \hat{s}^* ( \hat{\iota}^* \left(\widehat{\alpha}_x\right) \cup \alpha_y )= \hat{s}^*(\alpha_x   \cup \alpha_y  )\\
	=& \sum_{z\in X} \sum_{g\in I_z} \sum_{j=0}^{s+t+1} (-1)^j ( (k_{i_g,1}^z\cdots k_{i_g,j}^z \gamma_{s_{i_g}^j,z} g_{j+1}'\cdots g_{s+t+1}')^{-1} z, k_{i_g,1}^z,\cdots,k_{i_g,j}^z, \gamma_{s_{i_g}^j,z}, g_{j+1}',\cdots,g_{s+t+1}' ).
\end{align*}

It is evident that the key point of the above computation lies in how to obtain sequence 
\[
\{ k_{i_g,1}^z,\cdots,k_{i_g,j}^z, \gamma_{s_{i_g}^j,z}, g_{j+1}',\cdots,g_{s+t+1}' \}
\]
from the sequence $\{ g_1,\cdots,g_s,h_1, \cdots, h_t \}$. To facilitate the description, we define $\clubsuit_{s,t}^{x,g}$ as 
\[
( g_1',\cdots, g_{s+t+1}' )=\clubsuit_{s,t}^{x,g}( g_1,\cdots, g_s,h_1,\cdots,h_t )=( h_1,\cdots,h_t,g^{-1}g_s^{-1} \cdots g_1^{-1}x,g_1.\cdots,g_s )
\]

Therefore, for a plane binary tree of type $\beta_{i-, j-, -}$, the input consists of
\[
g_1, \cdots, g_s \in \overline{C_G(x)}, \quad h_1, \cdots, h_t \in \overline{C_G(y)}.
\]
Given a fixed $z \in X$, $g \in I_z$, and $\phi(g) \in \overline{C_G(z)} \gamma_{i_g, z}$, we fix $j \in \{1, 2, \cdots, s + t + 1\}$ and denote the sequence of transformations as:

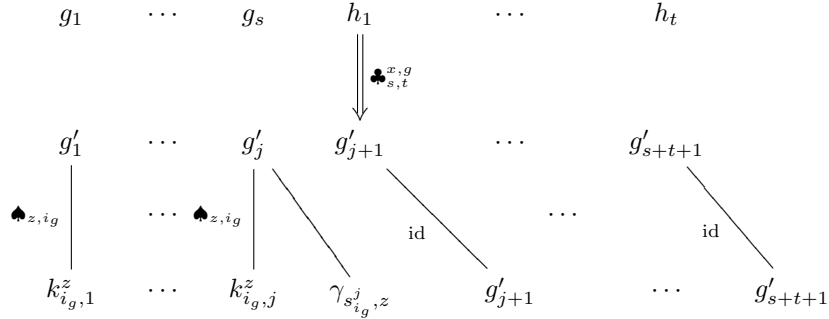
\[
\xymatrix@R=1.0pc@C=1.0pc{
& &	g_1 & \cdots & g_s & h_1 \ar@{=>}[dd] ^{\clubsuit_{s,t}^{x,g}} & \cdots & h_t \\
\quad & & \quad & \quad & \quad & \quad & \quad & \quad & & \quad \\
& & g'_1\ar@{-}[dd]_{\spadesuit_{z, i_g}} & \cdots & g'_j\ar@{-}_{\spadesuit_{z, i_g}}[dd] \ar@{-}[ddr] & g'_{j+1}\ar@{-}[ddr]_{\mathrm{id}} & \cdots & g'_{s+t+1}\ar@{-}[rdd]_{\mathrm{id}} & \quad\\
& & \quad & \cdots & \quad & \quad & \quad \quad\quad\quad \cdots & \quad\\
& & k_{i_g,1}^z & \cdots & k_{i_g,j}^z & \gamma_{s_{i_g}^j,  z} & g'_{j+1} & \cdots & g'_{s+t+1} }
\]
Specifically, the second-row sequence is obtained by applying the operation $\clubsuit_{s,t}^{x,g}$ to the entire first-row sequence.  
Following the notation from the previous section, the third-row sequence is obtained from the second-row sequence by first applying the transformation $\spadesuit_{z, i}$ to the first $j$ elements,  
then inserting an additional element $\gamma_{s_{i_g}^{j}, z}$, while keeping the last $s+t-j+1$ elements unchanged.

Moreover, for the transformation type $\beta_{i-, j-, -}$, the sequence transformation remains the same regardless of whether $i, j = 0$ or $1$.  
Therefore, we refer to the above sequence of transformations collectively as the $\beta_{-, -}$ transformation.

\subsubsection{\texorpdfstring{$\alpha_{i-,j-,-}$}{alpha(i-,j-,-)}}\

We calculate $\alpha_{1-,1-,-}$ first, that is, the following operation
\[
\begin{tikzcd}[every matrix/.style={minimum width=0.3em,  minimum height=0.2em}]
{\underset{x \in X}{\bigoplus}\widehat{C}^{<0}\left(C_G(x), k\right)} \arrow[rdd,  phantom] \arrow[rdd,  "\hat{\iota}^*",  no head] & & {\underset{x \in X}{\bigoplus} \widehat{C}^{<0}\left(C_G(y), k\right)} \arrow[ldd,  "\hat{\iota}^*",  no head] \\
\\
& \cup \arrow[dd,  "\hat{\rho^*}",  no head] & {\ \mathcal{D}^{<0}(k G, k G)} \\
& & & {, } \\
& {\alpha_{1-, 1-, -}}                     
\end{tikzcd}
\]
for any $x,y\in X$, $s,t\ge 0$. Let
\[
\widehat{\alpha}_x:=\left(g_{1,  s}\right) \in  \widehat{C}^{-s-1}\left(C_G(x), k\right)=k\left[{\overline{C_G(x)}}^{\times s}\right]\text{ and } \widehat{\alpha}_y:=\left(h_{1, t}\right) \in \widehat{C}^{-t-1}\left(C_G(y),  k\right)=k\left[{\overline{C_G(y)}}^{\times t}\right],
\]
we have
\[
\rho_{z,s+t+1}(\iota_{x,s}\left(\widehat{\alpha}_x\right)  \cup  \iota_{y,t}\left(\widehat{\alpha}_y\right)) 
=\sum_{g \in I_{z}}\left(k_{i_g,1}^z,  \cdots,  k_{i_g,s+t+1}^z\right), 
\]
\[
m_2\left(\widehat{\alpha}_x, \widehat{\alpha}_y\right) = \sum_{z \in X}\sum_{g \in I_z}\left(k_{i_g,1}^z, \cdots, k_{i_g,s+t+1}^z\right), 
\]
where 
\[
I_z:=\{ g\in G |\ h_{1} \cdots h_{t} g^{-1} g_{s}^{-1} \cdots g_{1}^{-1}  x  g_1\cdots  g_s g h_{t}^{-1} \cdots h_{1}^{-1} y=\phi(g)^{-1} z \phi(g) \text{ for some } \phi(g) \in G\},
\]
for any $z\in X$ and $g\in I_z$, $\phi(g)\in \overline{C_G(z)}\gamma_{i_g,z}$, and
\[
\{ k_{i_g,1}^z,\cdots,k_{i_g,s+t+1}^z \}=\spadesuit_{z,i_g}\{ h_1,\cdots,h_t,g^{-1} g_s^{-1} \cdots g_1 ^{-1} x, g_1,\cdots, g_s \}.
\]

By the result of $\beta_{-,-}$ transformation, the input element of 0-end is $( g_s^{-1}\cdots g_1^{-1}x, g_{1,s} )$ for $g_1,\ldots,g_s\in \overline{G}$. By computation of $\alpha_{0-,1-,-}$, $\alpha_{1-,0-,-}$ and $\alpha_{0-,0-,-}$, for 
\[  
\alpha_x=\left(g_s^{-1} \cdots g_1^{-1} x,  g_{1, s}\right) \in \mathcal{H}_{x,s};  
\]
\[ 
\alpha_y=\left(h_t^{-1} \cdots h_1^{-1} y, h_{1, t}\right) \in \mathcal{H}_{y,t};
\]
\[  
\widehat{\alpha}_x:=\left(g_{1, s}\right) \in C_s\left(C_G(x), k\right);
\]
\[ 
\widehat{\alpha}_y:=\left(h_{1, t}\right) \in   C_t\left(C_G(y), k\right),  
\]
then
\[
\rho_{z,s+t+1}(\alpha_{x} \cup \iota_{y,t}\left(\widehat{\alpha}_y\right))= \rho_{z,s+t+1} (\iota_{x,s}\left(\widehat{\alpha}_x\right) \cup \alpha_y)=\rho_{z,s+t+1}(\alpha_x \cup \alpha_y ) 
=\sum_{g \in I_z}\left(k_{i_g,1}^z, \cdots, k_{i_g,s+t+1}^z\right).  
\]

We observe that regardless of whether $i, j = 0$ or $1$,  
the key result obtained from the binary tree of type $\alpha_{i-, j-, -}$ focuses on:
\[
\text{how to derive }\{ k_{i_g,1}^z, \cdots, k_{i_g,s + t + 1}^z \}\text{ from } \{g_1, \cdots, g_s, h_1, \cdots, h_t\}.
\]
For a planar binary tree of type $\alpha_{i-, j-, -}$, the input consists of
\[
g_1, \cdots, g_s \in \overline{C_G(x)}, \quad h_1, \cdots, h_t \in \overline{C_G(y)}.
\]

Given a fixed $z \in X$, $g \in I_z$, and $\phi(g) \in \overline{C_G(z)} \gamma_{i_g, z}$, we fix
\[
j \in \{1, \cdots, s + t+1\},
\]
and denote the sequence of transformations as the $\alpha_{-, -}$ transformation.  
The specific transformation process is as follows:
\[
\xymatrix@R=0.6pc@C=0.8pc{
& g_1 & \cdots & g_s & h_1\ar@2[dd]^{\clubsuit_{s,t}^{x,g}} & \cdots & & h_t & & &\quad\\
\quad & & \quad & \quad & \quad & \quad & & & \quad \\
& h_1\ar@{-}[dd] ^{\spadesuit_{z, i}} & \cdots & h_t\ar@{-}^{\spadesuit_{z, i}}[dd] & g^{-1} g_{s}^{-1} \cdots g_{1}^{-1}x\ar@{-}[dd] ^{\spadesuit_{z, i}}& g_{1}\ar@{-}[dd] ^{\spadesuit_{z, i}}& \cdots & g_{s}\ar@{-}[dd]^{\spadesuit_{z, i}} & \\
& \quad & \cdots & \quad & \quad & & \quad & \quad & & & \quad \\
& k_{i_g,1}^z & \cdots & k_{i_g,j}^z& k_{i_g,j+1}^z & k_{i_g,j+2}^z & \cdots & k_{i_g,s+t+1}^z  
}
\]
This sequence transformation proceeds as follows:  
the second row is obtained by applying the transformation $\clubsuit_{s,t}^{x,g}$ to the first row, and the third row is obtained from the second row by applying the transformation $\spadesuit_{z, i}$.

\subsubsection{\texorpdfstring{Computation of $\widehat{m}_n$}{Computation of mn 2}}\

The multiplication $\widehat{m}_n$ is the combination of all the branching diagrams of type $\alpha_{-, -}$ and $\beta_{-, -}$ described above.  
We explain how to compose these sequence transformations in order to compute $\widehat{m}_n$ in the context of Hochschild homology, leading to the following theorem.

\begin{Thm}\label{thm:mn-on-homo}
The expansion terms in the $A_\infty$-multiplication formula on the Hochschild homology of the group algebra $kG$ at the complex level, up to a sign difference, are determined as follows:

Given 
\[
\widehat{\alpha}_1=(g_{1, 1},\cdots , g_{1, i_1} )  \in {\overline{C_G(y_1)}}^{\times i_1},
\]
\[
\widehat{\alpha}_2=(g_{2,1},\cdots , g_{2, i_2})  \in {\overline{C_G(y_2)}}^{\times i_2},
\]
\[
\vdots
\]
\[
\widehat{\alpha}_n=(g_{n, 1}\dots , g_{n, i_n})  \in {\overline{C_G(y_n)}}^{\times i_n},
\]
where $y_p\in X$ for $p=1,\cdots,n$. To compute $\widehat{m}_n( {\widehat{\alpha}_1}, \cdots, {\widehat{\alpha}_n})$, we need to see the sequence
\[
(g_{1, 1}\dots , g_{1, i_1}, g_{2, 1}, \dots , g_{2, i_2}, \dots , g_{n, 1}, \dots , g_{n, i_n})
\] 
as the input elments of the transformation $\beta_{-,-}$. As we analyzed above, this sequence will perform $n-2$ $\beta_{-,-}$ transformations, and one $\alpha_{-,-}$ transformation at last, we obtain the output sequence
\[
(k_{j, 1}^z, k_{j, 2}^z, \cdots, k_{j, \sum\limits_{p=1}^{n}i_p+2n-3 }^z),
\]
for any $z\in X$ and $j$ such that $1\le j\le n_x$ and $j$ satisfies some condition. Therefore, we have
\[
\widehat{m}_n(\widehat{\alpha}_1, \cdots, \widehat{\alpha}_n) 
=\sum_{z\in X} \sum_{j} \pm (k_{j, 1}^z, k_{j, 2}^z, \cdots, k_{j, \sum\limits_{p=1}^{n}i_p+2n-3 }^z).
\]
\end{Thm}

\section{\texorpdfstring{$A_{\infty}$-structures in the abelian group case}{A-infinity structures in the abelian group case}}\label{section:abelian-group-case}

In this section, we always let $G$ be an abelian group.

\subsection{The homotopy transfer theorem for the Tate-Hochschild cohomology of G}\ 

\textbf{(1) The additive decomposition of the Hochschild cohomology of $G$ at the complex level}

As an abelian group, $G$ is the set of representatives of the conjugacy classes of elements in $G$. For each $x\in G$, as the conjugacy of $x$, $C_x=\{ x \}$, and the centralizer subgroup $C_G(x)=G$. Define
\[
\mathcal{H}^{x,0}=k [x],\text{ and for } n\ge 1,
\]
\[
\mathcal{H}^{x,n}=\{ \varphi: \overline{G}^{\times n} \longrightarrow kG \ |\ \varphi(g_1, \dots, g_n) \in k[g_1 \cdots g_n x] \subset kG, \ \forall g_1, \dots, g_n \in \overline{G}  \}.
\]

To simplify our computations, we begin by introducing the $\spadesuit$ process we used a lot before, that is, for any $\{ g_1, \cdots, g_n\}\in \overline{C_G(x)}^{\times n}=\overline{G}^{\times n}$,
\[
\{g_1,\cdots,g_n\}=\spadesuit_{x}\{g_1,\cdots,g_n\}.
\]

A more concrete version of Lemma~\ref{Lemma: add Hochschild cohomology} can be formulated in the setting of abelian groups:

The additive decomposition $\mathrm{HH}^*(k G, k G) \simeq \bigoplus_{x \in X} \mathrm{H}^*(G, k)$ can lift to a homotopy deformation retract of complexes
\begin{equation*}\xymatrix@C=0.000000001pc{
		\ar@(lu, dl)_-{s^*}&C^*(kG, kG) \ar@<0.5ex>[rrrrr]^-{\iota^*}&&&&&\bigoplus\limits_{x\in G}  C^*(G, k). \ar@<0.5ex>[lllll]^-{\rho^*}}
\end{equation*}
where $\iota^n=\sum_{x\in G} \iota^{x,n}, \ \rho^n=\sum_{x\in G} \rho^{x,n}$, and $s^n=\sum_{x\in G}s^{x,n}$, for $n\ge 0$.

The map $\iota^{x,n}$ is given by
\begin{equation*}
	\begin{split}
		\iota^{x,n}: \mathcal{H}^{x,n} \rightarrow C^n(G,  k),  \quad
		[\varphi_x: \overline{G}^{\times n}\rightarrow kG]\mapsto [\widehat{\varphi}_x:
		\overline{G}^{\times n}\rightarrow k],
	\end{split}
\end{equation*}
\[
\text{with }\widehat{\varphi}_x(g_{1, n})=a_x \text{ where }\varphi_x(g_{1, n})g_n^{-1}\cdots
g_1^{-1}=a_x x\in k [x].
\]
The map $\rho^{x,n}$ is given by
\begin{equation*}
	\begin{split}
		\rho^{x,n}:  C^n(G,  k)&\rightarrow \mathcal H^{x,n} , \quad
		[\widehat{\varphi}_x:
		\overline{G}^{\times n}\rightarrow k]\mapsto
		[\varphi_x: \overline{G}^{\times n}\rightarrow kG], \end{split}
\end{equation*}
\[
\text{with }\varphi_x\in\mathcal{H}^{x,n} , \text{ and }\varphi_x(g_{1,n})=\widehat{\varphi}_x(g_{1,n})x g_1\cdots g_n.
\]
The homotopy $s^{x,n}$ is given by: for $(\varphi_x: \overline{G}^{\times n} \rightarrow kG)\in \mathcal{H}^{x,n}$, we define $s^{x,n}(\varphi_x) \in \mathcal{H}^{x,n-1}$ as
\[
s^{x,n}(\varphi_x)(g_{1,  n-1})=\sum_{j=0}^{n-1} (-1)^j\widehat{\varphi}_x( g_1,\cdots,g_j,1,g_{j+1},\cdots,g_{n-1} ) x g_1\cdots g_{n-1}=0.
\]
Note that $s^{x,n}(\varphi_x)(g_{1,n-1})=0$ since $1$ appears in each term.

\textbf{(2) The additive decomposition of the Hochschild homology of $G$ at the complex level}

For any $x\in G$, write 
\[
\mathcal{H}_{x, 0}=k [x ], \text{ and for } s\ge 1,
\]  
\[
\mathcal{H}_{x, s}=k\left[\left(g_s^{-1} \cdots g_1^{-1} x, g_{1, s}\right) |\ g_1, \cdots, g_s \in \overline{G}\right].
\]
Let $\mathcal{H}_{x, *}=\bigoplus_{s \ge 0} \mathcal{H}_{x, s}$. It is easy to verify that $C_*(k G, k G)=\bigoplus_{x\in G} \mathcal{H}_{x, *}$.

Similar to the Hochschild cohomology case, Lemma~\ref{Lemma: add Hochschild homology} can also be formulated in the setting of abelian groups:

The additive decomposition $\mathrm{HH}_*(k G,  k G) \simeq \bigoplus_{x \in X} \mathrm{H}_*\left(G, k\right)$ can lift to a homotopy deformation retract of complexes
\begin{equation*}
\xymatrix@C=0.000000001pc{
		\ar@(lu, dl)_-{s_*}&C_*(kG,  kG)=\bigoplus_{x \in G} \mathcal{H}_{x,  *} \ar@<0.5ex>[rrrrr]^-{\rho_*}&&&&&\bigoplus\limits_{x\in G}  C_*(G,  k), \ar@<0.5ex>[lllll]^-{\iota_*}}
\end{equation*}
where $\iota_n=\sum_{x\in G} \iota_{x,n}, \ \rho_n=\sum_{x\in G} \rho_{x,n}$ and $s_n=\sum_{x\in G}s_{x,n}$, for $n\ge 0$.

The injection $\iota_{x,n}$ is given by
\[
\iota_{x,n}: C_n\left(G,  k\right)  \stackrel{\sim}{\longrightarrow} \mathcal{H}_{x,  n},\quad\quad\quad\quad\quad\quad 
\]
\[
{\quad\quad \widehat{\alpha}_{x}=\left(g_{1},  \cdots,  g_{n}\right) \in  {\overline{G}}^{\times n}}  \longmapsto \alpha_{x}=\left(g_n^{-1} \cdots g_1^{-1}x,  g_{1,  n}\right) \in \mathcal{H}_{x,  n} ,
\]
and the surjection $\rho_{x,n}$ is given by
\[
\ \quad\quad\rho_{x,n}: \mathcal{H}_{x, n} \longrightarrow C_n\left(G, k\right), 
\]
\[
{ \alpha_x=\left(g_n^{-1} \cdots g_1^{-1} x, g_{1, n}\right) \in \mathcal{H}_{x, n}  }  \longmapsto \widehat{\alpha}_x=\left(g_1, \cdots, g_n\right) \in \overline{G} ^{\times n}.
\]
The homotopy $s_{x,n}$ is given as follows: for $\alpha_x=\left(g_n^{-1} \cdots g_1^{-1} x, g_{1, n}\right) \in \mathcal{H}_{x, n}$, 
\[
s_{x,n}\left(\alpha_x\right)=\sum_{j=0}^n(-1)^j\left(g_n^{-1} \cdots g_1^{-1} x, g_1, \cdots, g_j, 1, g_{j+1}, \cdots, g_n\right)=0.
\]
This means that $s_{x,n}(\alpha_x)$ is empty since $1$ appears in each term.

From the above analysis and Theorem~\ref{Thm: add tate Hochschild} of abelian group version, we obtain the following Theorem.
\begin{Thm}
Define $\widehat{\mathcal{H}}_x^*$ such that for $m\ge 0$,
\[
\widehat{\mathcal{H}}_x^m=\mathcal{H}^{x,m}, \text{ and }\widehat{\mathcal{H}}_x^{-m-1}=\mathcal{H}_{x,m}, 
\]
then as complex, we have $\widehat{\mathcal{H}}_x^*\simeq \widehat{C}^*(G,k)$. 
%More precisely, $\widehat{\mathcal{H}}_x^*\simeq \widehat{C}^*(G,k)$ also as $A_{\infty}$ algebras.
\end{Thm}
\begin{proof}
It is easy to verify that  $\hat{\iota}^*\hat{\rho}^*=\mathrm{id}$ and $\hat{\rho}^*\hat{\iota}^*=\mathrm{id}$.
\end{proof}

\subsection{\texorpdfstring{The operation $\widehat{m}_n$ for abelian groups}{The operation mn for abelian groups}}\

Recall that we have a $A_{\infty}$-algebra structure $(m_1,m_2,m_3,\cdots)$ on $\mathcal{D}^*(kG,kG)$, with $m_1=d^*$, $m_2=\cup$ and $m_i=0$ for $i> 3$. Since $\hat{s}^*=0$ for abelian group, when we calculate $A_{\infty}$-structure $( \widehat{m}_1, \widehat{m}_2, \widehat{m}_3, \cdots )$, we have $\widehat{m}_n=\pm \hat{\rho}^* ( m_n(\hat{\iota}^{*\otimes n}) )$, so $\widehat{m}_i=0$ for $i>3$. Furthermore, we recall the following result in \cite{LWZ21}.

\begin{Thm}\cite[Corollary 4.11]{LWZ21}\label{theorem:Ainfity-finite-group}
    Let $G$ be a finite abelian group. Then we have 
    \[
    \bigoplus_{x\in X} \widehat{G}^*(C_G(x),k)=\bigoplus_{x\in G}\widehat{C}^*(G,k).
    \]
    Assume that we already have $A_{\infty}$-structure $(\widehat{m}_1', \widehat{m}_2', \cdots)$ on $\widehat{C}^*(G, k)$, then we have the following isomorphisms as $A_{\infty}$ algebras:
    \[
    \phi:\;	D^*(kG,kG)\simeq\bigoplus_{x\in G}  \widehat{C}^*(C_G(x), k)\simeq kG\otimes  \widehat{C}^*(G, k).
    \]
    More precisely, the $A_{\infty}$-structure $(\widehat{m}_1, \widehat{m}_2, \cdots)$ on $kG\otimes  \widehat{C}^*(G, k)$ is given as follows.
    \[
    \widehat{m}_p((g_1\otimes \widehat{\alpha}_1), \cdots, (g_p\otimes \widehat{\alpha}_p)) =g_1\cdots g_p\otimes \widehat{m}'_p(\widehat{\alpha}_1, \cdots,  \widehat{\alpha}_p)
    \]
    for any $g_i \otimes \alpha_i\in kG\otimes  \widehat{C}^*(G, k)$ ($i=1,\cdots,p$).
\end{Thm}
From this theorem, to compute the $A_{\infty}$-structure on $\bigoplus\limits_{x\in X}\widehat{C}^*(C_G(x),k)$, it is enough to compute the $A_{\infty}$-structure on $\widehat{C}^*(G,k)$. We now proceed to compute the $A_{\infty}$-structure $(\widehat{m}_1', \widehat{m}_2', \cdots)$ on $\widehat{C}^*(G,k)$.

\textbf{(1) The differential on the complex $(\widehat{C}^*(G,k), \delta_*)$ (i.e. $\widehat{m}_1'$):}
\[
\widehat{C}^*(G, k)=\left(\cdots \xrightarrow{\partial_{2}} C_{1} (G, k) \xrightarrow{\partial_{1}} C_{0} (G, k) \xrightarrow{\tau} C^{0} (G, k) \xrightarrow{\delta^{0}} C^{1} (G, k) \xrightarrow{\delta^{1}} \cdots\right),
\]
\begin{itemize}
    \item [(i)] for $n\ge 0$,
    \[
    \widehat{C}^n (G, k)=C^n (G, k)= \mathrm{Map}(\overline{G}^{\times n}, k),
    \]
    and the differential is given by $\delta_n =\delta^n$, where $\delta^n(\varphi)$ sends $g_{1, n+1} \in \overline{G}^{\times n+1} $ to
    \[
    g_1 \varphi\left(g_{2, n+1}\right)+\sum_{i=1}^n(-1)^i \varphi\left(g_{1, i-1}, g_i g_{i+1}, g_{i+2, n+1}\right)+(-1)^{n+1} \varphi\left(g_{1, n}\right).
    \]

    \item [(ii)] for $n\le -1$ (let $s=-n-1\ge 0$),
    \[
    \widehat{C}^n(G, k)=C_s(G, k)= k[\overline{G} ^{\times s} ], \text{ and for } n\le -3,
    \]
    $\delta_n=\partial_s: k [\overline{G}^{\times s} ] \rightarrow  k [\overline{G}^{\times(s-1)} ]$ is defined by 
    \[
    g_{1, s} \mapsto g_{2, s}+\sum_{i=1}^{s-1}(-1)^i \left(g_{1, i-1}, g_i g_{i+1}, g_{i+2, s}\right)+(-1)^{s}  \left(g_{1, s-1}\right).
    \]
    Moreover, $\delta_{-2}=\partial_1: k[\overline{G}] \rightarrow k$ is defined by $\partial_1 \left( g_1 \right)=0$.

    \item [(iii)] $\delta_{-1}=\tau: C_0(G, k)=k \rightarrow k=C^0(G, k)$ is given by $\tau(1)  =|G|$.
\end{itemize}
Then we obtain the definition of $\widehat{m}_1'$,
    \[
    \widehat{m}_1'(\alpha)=\left\{\begin{array}{ll}
 	(-1)^{m+1} \partial_{-m-1}(\alpha) & \text{for}\ \alpha \in \widehat{C}^m(G, k)\ \text{and}\ m<-1,\\
 	\tau(\alpha)& \text{for}\ \alpha \in \widehat{C}^{-1}(G, k), \\
 	\delta^m(\alpha) & \text{for}\ \alpha \in \widehat{C}^m(G, k)\ \text{and}\ m\ge 0.
    \end{array}\right.
    \]

\textbf{(2) The definition of $\widehat{m}_2'$:}

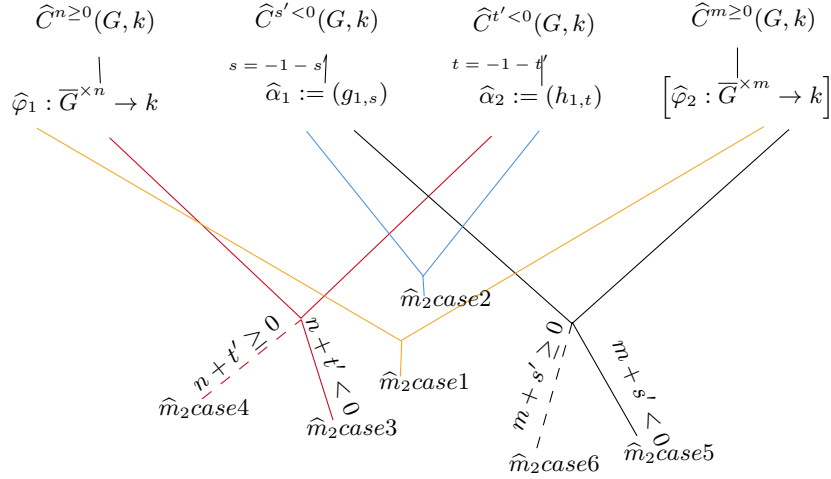
\begin{center}
	\begin{tikzpicture}[
	x=0.5pt, 
	y=0.5pt, 
	yscale=-1,
	xscale=1,
	every node/.style={font=\small}
	]
	\draw  [color={rgb, 255:red, 208; green, 2; blue, 27 }  ,draw opacity=1 ] (392,117.59) -- (247.77,255.49) -- (102.96,118.85) ;
	\draw [color={rgb, 255:red, 208; green, 2; blue, 27 }  ,draw opacity=1 ]   (248,256.59) -- (272,332.59) ;
	\draw   (617,112.59) -- (452.78,259.55) -- (287.95,113.28) ;
	\draw  [color={rgb, 255:red, 74; green, 144; blue, 226 }  ,draw opacity=1 ] (428,113.59) -- (340.23,223.58) -- (252,113.96) ;
	\draw  [dash pattern={on 4.5pt off 4.5pt}]  (453,258.59) -- (426,354.59) ;
	\draw  [color={rgb, 255:red, 245; green, 166; blue, 35 }  ,draw opacity=1 ] (598,110.44) -- (323.34,272.53) -- (48,111.59) ;
	\draw [color={rgb, 255:red, 74; green, 144; blue, 226 }  ,draw opacity=1 ]   (340,223.59) -- (341,238.59) ; 
	\draw [color={rgb, 255:red, 245; green, 166; blue, 35 }  ,draw opacity=1 ]   (324,272.59) -- (323,299.59) ; 
	\draw [color={rgb, 255:red, 208; green, 2; blue, 27 }  ,draw opacity=1 ] [dash pattern={on 4.5pt off 4.5pt}]  (173,316.59) -- (248,256.59) ; 
	\draw    (453,258.59) -- (502,344.59) ; 
	\draw    (95,57.6) -- (96,77.6) ;
	\draw    (267,55.6) -- (266,77.6) ; 
	\draw    (430,56.6) -- (430,79.6) ; 
	\draw    (578,50.6) -- (578,73.6) ;
	\draw (219,75) node [anchor=north west][inner sep=0.75pt]    {$\widehat{\alpha}_1 :=( g_{1,s})$};
	\draw (380.67,77) node [anchor=north west][inner sep=0.75pt]    {$\widehat{\alpha }_2 :=( h_{1,t})$};
	\draw (27,77) node [anchor=north west][inner sep=0.75pt]    {$\widehat{\varphi }_{1} :\overline{G}^{\times n}\rightarrow k $};
	\draw (517,70) node [anchor=north west][inner sep=0.75pt]    {$\left[\widehat{\varphi }_{2} :\overline{G}^{\times m}\rightarrow k\right]$};
	\draw (45,16) node [anchor=north west][inner sep=0.75pt]    {$\widehat{C}^{n\ge 0}( G,k)$};
	\draw (376,18) node [anchor=north west][inner sep=0.75pt]    {$\widehat{C}^{t'< 0}( G,k)$};
	\draw (211,15) node [anchor=north west][inner sep=0.75pt]    {$\widehat{C}^{s'< 0}( G,k)$};
	\draw (541,15) node [anchor=north west][inner sep=0.75pt]    {$\widehat{C}^{m\ge 0}( G,k)$};
	\draw (305,294) node [anchor=north west][inner sep=0.75pt]    {$\widehat{m}_2 case1$};
	\draw (321,231) node [anchor=north west][inner sep=0.75pt]    {$\widehat{m}_2 case2$};
	\draw (159.25,299.55) node [anchor=north west][inner sep=0.75pt]  [rotate=-321.3]  {$n+t'\ge 0$};
	\draw (265.2,249.02) node [anchor=north west][inner sep=0.75pt]  [rotate=-66.9]  {$n+t'< 0$};
	\draw (399.32,337.49) node [anchor=north west][inner sep=0.75pt]  [rotate=-291.28]  {$m+s'\ge 0$};
	\draw (497.2,266.02) node [anchor=north west][inner sep=0.75pt]  [rotate=-66.9]  {$m+s'< 0$};
	\draw (250,329) node [anchor=north west][inner sep=0.75pt] {$\widehat{m}_2 case3$};
	\draw (138,314) node [anchor=north west][inner sep=0.75pt] {$\widehat{m}_2 case4$};
	\draw (490,348) node [anchor=north west][inner sep=0.75pt] {$\widehat{m}_2 case5$};
	\draw (403,356) node [anchor=north west][inner sep=0.75pt] {$\widehat{m}_2 case6$};
	\draw (360,54) node [anchor=north west][inner sep=0.75pt]  [font=\tiny]  {$t=-1-t'$};
	\draw (191,55) node [anchor=north west][inner sep=0.75pt]  [font=\tiny]  {$s=-1-s'$};
\end{tikzpicture}
\end{center}
Percisely, the results are given as follows:

$\bf{\widehat{m}_2' case1}$
\[
\widehat{m}_2'\left(\widehat{\varphi}_1, \widehat{\varphi}_2\right)\left(h_1, \cdots, h_n, h_{n+1}, \cdots, h_{n+m}\right) = \widehat{\varphi}_1\left(h_1, \cdots, h_n\right) \widehat{\varphi}_2\left(h_{n+1}, \cdots, h_{n+m}\right).
\]

$\bf{\hat{m}_2' case2}$
\[
\widehat{m}_2'\left(\widehat{\alpha}_1, \widehat{\alpha}_2\right) =  \sum_{g \in G} \left(h_1,h_2,\cdots h_t,g^{-1} g_s^{-1} \cdots g_1^{-1}, g_1,g_2,\cdots,g_s \right).
\]

$\bf{\widehat{m}_2' case3}$
\begin{align*}
\widehat{m}_2'\left(\widehat{\varphi}_1,\widehat{\alpha}_2\right)
&= \widehat{\varphi}_1\left(h_{t-n+1},\cdots, h_t\right) \left(h_1,\cdots, h_{t-n}\right)\\
&= \left(\widehat{\varphi}_1 \left(h_{t-n+1},\cdots, h_t\right)  h_1,\cdots, h_{t-n}\right).
\end{align*}

$\bf{\widehat{m}_2' case4}$
\[
\widehat{m}_2'\left(\widehat{\varphi}_1,\widehat{\alpha}_2\right)\left( g_{1, n-t-1}\right) 
= \sum_{g \in G}  \widehat{\varphi}_1\left(g_{1,n-t-1},g^{-1}, h_{1,t}\right).
\]

$\bf{\widehat{m}_2' case5}$ 
\[
\widehat{m}_2'\left(\widehat{\alpha}_1,\widehat{\varphi}_2\right) = \widehat{\varphi}_2 \left(g_1, \cdots, g_m\right) (g_{m+1}, \cdots, g_s)= ( \widehat{\varphi}_2 \left(g_1, \cdots, g_m\right)g_{m+1}, \cdots, g_s).
\]

$\bf{\widehat{m}_2' case6}$
\[
\widehat{m}_2'\left(\widehat{\alpha}_1,\widehat{\varphi}_2\right)\left(	h_{1, m-s-1} \right)  
= \sum_{g \in G }\widehat{\varphi}_2\left(g_{1,s},g^{-1}, h_{1,m-s-1}\right).
\]

\textbf{(3) The formula for $\widehat{m}_3'$:}

$\widehat{m}_i'=0$ for $i>3$ and $\widehat{m}_3'=0$ except for the following two cases.

\begin{itemize}
    \item [(i)]  For $\widehat{\phi} \in \widehat{C}^m(G, k)$, $\widehat{\varphi }\in \widehat{C}^n(G, k)$ and $\widehat{\alpha}=\left(g_1, \cdots, g_r\right) \in \widehat{C}_r(G, k)$, if $r+2 \le m+n$, $\widehat{m}_3'(\widehat{\phi}, \widehat{\alpha}, \widehat{\varphi}) \in \widehat{C}^{m-r+n-2}(G, k)$ is defined by 
    \begin{align*}
     &\widehat{m}_3'(\widehat{\phi}, \widehat{\alpha}, \widehat{\varphi})\left(h_1, \cdots, h_{m-r+n-2}\right) \\
     &=\sum_{g \in G} \sum_{j=\max\{1,r+2-m\}}^{\min \{n, r+1\}}(-1)^{m+r+j-1} \widehat{\phi}\left(h_{1, m-r+j-2}, g, g_{j, r}\right)  \widehat{\varphi}\left(g_{1, j-1}, g^{-1},\right. 
     \left.h_{m-r+j-1, m-r+n-2}\right). 
     \end{align*}

     \item [(ii)] For $\widehat{\alpha}=\left(g_1, \cdots, g_r\right) \in \widehat{C}_r (G,k)$, $\beta =\left(h_1,\cdots,h_s\right) \in \widehat{C}_s(G,k)$ and $\widehat{\phi} \in \widehat{C}^m(G,k)$, if $m-r \le s+1 $, then
     \[ 
     \widehat{m}_3'(\widehat{\alpha}, \widehat{\phi}, \widehat{\beta})=\sum_{g \in G} \sum_{j=\max\{0,s+1-m\}}^{\min\{s,r-m+s+1\}}(-1)^{m+r+s-j}\widehat{\phi}\left(g_{1, m-s+j-1}, g, h_{j+1, s}\right) \left( h_{1, j}, g^{-1}, g_{m-s+j, r}\right).
     \]
\end{itemize}

\subsection{Examples for abelian groups}

\subsubsection{\texorpdfstring{$G=\mathbb{Z}_2$}{G Z2}}\

Assume that $G=\{g,1\}$, with $g^2=1$.

We recall the definition of $(\widehat{C}^*(G,k),\delta_*)$:
\[
\widehat{C}^*(G, k)=\left(\cdots \xrightarrow{\partial_{2}} C_{1} (G, k) \xrightarrow{\partial_{1}} C_{0} (G, k) \xrightarrow{\tau} C^{0} (G, k) \xrightarrow{\delta^{0}} C^{1} (G, k) \xrightarrow{\delta^{1}} \cdots\right),
\]
\begin{itemize}
    \item [(i)] for $n\ge 0$,
    \[
    \widehat{C}^n (G, k)=C^n (G, k)= \mathrm{Map}(\overline{G}^{\times n}, k),
    \]
    and the differential is given by $\delta_n =\delta^n$, where $\delta^n(\varphi)$ sends $g^{\times n+1}$ to
    \[
    g\varphi(g^{\times n})+(-1)^{n+1}\varphi(g^{\times n})=\varphi(g^{\times n})+(-1)^{n+1}\varphi(g^{\times n}).
    \]

    \item [(ii)] for $n\le -1$ (let $s=-n-1\ge 0$),
    \[
    \widehat{C}^n(G, k)=C_s(G, k)= k[\overline{G} ^{\times s} ], \text{ and for } n\le -3,
    \]
    $\delta_n=\partial_s: k [\overline{G}^{\times s} ] \rightarrow  k [\overline{G}^{\times(s-1)} ]$ is defined by 
    \[
   g^{\times s} \mapsto g^{\times (s-1)}+(-1)^s g^{\times (s-1)}.
    \]
    Moreover, $\delta_{-2}=\partial_1=0$.

    \item [(iii)] $\delta_{-1}=\tau: C_0(G, k)=k \rightarrow k=C^0(G, k)$ is given by $\tau(1)  =2$.
\end{itemize}
Then we obtain $\widehat{m}_1'$, we observe that if $\mathrm{char}(k)=2$, $\widehat{m}_1'=0$.

Now we define $\widehat{m}_2'$, for any $\lambda,\mu\in k$, let $\lambda^n\in C^n (G, k): g^{\times n}\to \lambda$ and $\mu_t=\mu(g^{\times t})\in C_t(G,k)$.

\textbf{case 1.} $\lambda^n\in C^n(G,k) $ and $\mu^m\in C^m(G,k) $, $\widehat{m}_2'(\lambda^n,\mu^m)=(\lambda \mu)^{n+m}$.

\textbf{case 2.} $\lambda_s\in C_s(G,k) $ and $\mu_t\in C_t(G,k)$, $\widehat{m}_2'(\lambda_s,\mu_t )=\sum_{a=1}^3 {\lambda\mu}_{s+t+1}$.

\textbf{case 3.} $\lambda^n\in C^n(G,k) $, $\mu_t\in C_t(G,k)$ and $n-t-1\le -1$, $\widehat{m}_2'(\lambda^n,\mu_t)=\widehat{m}_2'(\mu_t,\lambda^n)=(\lambda\mu)_{t-n}$.

\textbf{case 4.} $\lambda^n\in C^n(G,k) $, $\mu_t\in C_t(G,k)$ and $n-t-1\ge 0$, $\widehat{m}_2'(\lambda^n,\mu_t)=\widehat{m}_2'(\mu_t,\lambda^n)=(\lambda\mu)^{n-t-1}$.

For $\widehat{m}_3'$, we only need to discuss the following two cases,
\begin{itemize}
    \item [(1)] $\lambda^m\in C^m(G,k) $, $\mu^n\in C^n(G,k) $, $v_r\in C_r(G,k)$ and $r+2\le m+n$, then
    \[
    \widehat{m}_3'( \lambda^m, v_r, \mu^n)=\sum_{i=\text{max}\{1,r+2-m\}}^{\text{min}\{n,r+1\}} (-1)^{m+r+i-1} (\lambda\mu v)^{m-r+n-2 }.
    \]

    \item [(2)] $\lambda_r\in C_r(G,k) $, $v_s \in C_s(G,k)$, $\mu^m\in C^m(G,k) $ and $m-1\le r+s$,
    \[
    \widehat{m}_3'(\lambda_r,\mu^m,v_s )=\sum_{i=\text{max}\{0,s+1-m\}}^{\text{min}\{s,r-m+s+1\}} (-1)^{m+r+s-i} \lambda (\lambda \mu v)_{r-m+s+2}.
    \]
   
\end{itemize}

\subsubsection{\texorpdfstring{$G=\mathbb{Z}_4$}{G Z4}}\

Assume that $G=\{ g^3,g^2,g,1 \}$, with $g^4=1$. Before calculation, we first give some notations:
\begin{itemize}
    \item for any $n$, write the elements in $\overline{G}^{\times n}$ as 
    \[
    g_{j_1,\cdots,j_n}:=(g^{j_1},\cdots,g^{j_n} ) \text{ with }j_1,\cdots,j_n\in I_3=\mathbb{Z}_4/\{0\}=\{1,2,3\},
    \]

    \item define the map $\lambda^{j_1,\cdots,j_n}\in \mathrm{Map} (\overline{G}^{\times n},k)$ by 
    \[
    \lambda^{j_1,\cdots,j_n} (g_{j_1',\cdots,j_n'})=
    \left\{\begin{array}{ll}
		\lambda,&\text{if } j_1'=j_1,\cdots,j_n'=j_n;\\
		0 ,&\text{otherwise}.
	\end{array}
	\right.
    \]

    \item define the map $c_i:\ I_3^{\times n}\to I_3^{\times n+1}$ as 
    \[
      (j_1,\cdots,j_n)\mapsto 
      \left\{\begin{array}{ll}
		(j_1,\cdots,j_{i-1},2,3,j_{i+1},\cdots,j_n)+(j_1,\cdots,j_{i-1},3,2,j_{i+1},\cdots,j_n) ,&\text{if } j_i=1;\\
        (j_1,\cdots,j_{i-1},3,3,j_{i+1},\cdots,j_n) ,&\text{if } j_i=2;\\
		(j_1,\cdots,j_{i-1},1,2,j_{i+1},\cdots,j_n)+(j_1,\cdots,j_{i-1},2,1,j_{i+1},\cdots,j_n) ,&\text{if } j_i=3.
	\end{array}
	\right.
    \]
\end{itemize}

We recall the definition of $(\widehat{C}^*(G,k),\delta_*)$:
\[
\widehat{C}^*(G, k)=\left(\cdots \xrightarrow{\partial_{2}} C_{1} (G, k) \xrightarrow{\partial_{1}} C_{0} (G, k) \xrightarrow{\tau} C^{0} (G, k) \xrightarrow{\delta^{0}} C^{1} (G, k) \xrightarrow{\delta^{1}} \cdots\right),
\]
\begin{itemize}
    \item [(i)] for $n\ge 0$,
    \[
    \widehat{C}^n (G, k)=C^n (G, k)= \mathrm{Map}(\overline{G}^{\times n}, k),
    \]
    and the differential is given by $\delta_n =\delta^n$, sends $\lambda^{j_1,\cdots,j_n}$ to
    \[
    \sum_{j_0=1}^3\lambda^{j_0,j_1,\cdots,j_n}+\sum_{i=1}^n (-1)^i\lambda^{c_i(j_1,\cdots,j_n)}+(-1)^{n+1}\sum_{j_{n+1}=1}^3 \lambda^{j_1,\cdots,j_n,j_{n+1}}.
    \]

    \item [(ii)] for $n\le -1$ (let $s=-n-1\ge 0$),
    \[
    \widehat{C}^n(G, k)=C_s(G, k)= k[\overline{G} ^{\times s} ], \text{ and for } n\le -3,
    \]
    $\delta_n=\partial_s: k [\overline{G}^{\times s} ] \rightarrow  k [\overline{G}^{\times(s-1)} ]$ is defined by 
    \[
    g_{j_1,\cdots,j_s} \mapsto g_{j_2,\cdots,j_s}+\sum_{i=1}^{s-1}(-1)^i g_{j_1,\cdots,j_ {i-1}, j_i+ j_{i+1}, j_{i+2},\cdots, j_s} +(-1)^{s}  g_{j_1,\cdots, j_{s-1}}.
    \]
    Moreover, $\delta_{-2}=\partial_1=0$.

    \item [(iii)] $\delta_{-1}=\tau: C_0(G, k)=k \rightarrow k=C^0(G, k)$ is given by $\tau(1)  =4$.
\end{itemize}
Then we obtain $\widehat{m}_1'$, in the following, we define $\widehat{m}_2'$,

\textbf{case 1.} $\lambda^{j_1,\cdots,j_n}\in C^n(G,k) $ and $\mu^{k_1,\cdots,k_m}\in C^m(G,k) $,
\[
\widehat{m}_2'(\lambda^{j_1,\cdots,j_n},\mu^{k_1,\cdots,k_m})=(\lambda \mu)^{j_1,\cdots,j_n,k_1,\cdots,k_m}.
\]

\textbf{case 2.} $g_{j_1,\cdots,j_s}\in C_s(G,k) $ and $g_{k_1,\cdots,k_t}\in C_t(G,k)$,
\[
\widehat{m}_2'(g_{j_1,\cdots,j_s}, g_{k_1,\cdots,k_t} )=\sum_{a=1}^3 g_{k_1,\cdots,k_t,a,j_1,\cdots,j_s}.
\]

\textbf{case 3.} $\lambda^{j_1,\cdots,j_n}\in C^n(G,k) $, $g_{k_1,\cdots,k_t}\in C_t(G,k)$ and $n-t-1\le -1$,
\[
\widehat{m}_2'(\lambda^{j_1,\cdots,j_n}, g_{k_1,\cdots,k_t})=
\left\{\begin{array}{ll}
		\lambda g_{k_1,\cdots,k_{t-n}} ,&\text{if } k_{t-n+1}=j_1,\cdots,k_t=j_n;\\
		0 ,&\text{otherwise}.
	\end{array}
	\right.
\]

\textbf{case 4.} $\lambda^{j_1,\cdots,j_n}\in C^n(G,k) $, $g_{k_1,\cdots,k_t}\in C_t(G,k)$ and $n-t-1\ge 0$,
\[
\widehat{m}_2'(\lambda^{j_1,\cdots,j_n}, g_{k_1,\cdots,k_t})=
\left\{\begin{array}{ll}
		\lambda^{j_1,\cdots,j_{n-t-1}} ,&\text{if } k_1=j_{n-t+1},\cdots,k_t=j_n;\\
		0 ,&\text{otherwise}.
	\end{array}
	\right.
\]

\textbf{case 5.} $g_{j_1,\cdots,j_s}\in C_s(G,k)$, $\mu^{k_1,\cdots,k_m}\in C^m(G,k) $ and $m-s-1\le -1$,
\[
\widehat{m}_2'(g_{j_1,\cdots,j_s}, \mu^{k_1,\cdots,k_m})=
\left\{\begin{array}{ll}
		\mu g_{j_{m+1},\cdots,j_s} ,&\text{if } j_1=k_1,\cdots,j_m=k_m;\\
		0 ,&\text{otherwise}.
	\end{array}
	\right.
\]

\textbf{case 6.} $g_{j_1,\cdots,j_s}\in C_s(G,k)$, $\mu^{k_1,\cdots,k_m}\in C^m(G,k) $ and $m-s-1\ge 0$,
\[
\widehat{m}_2'(g_{j_1,\cdots,j_s}, \mu^{k_1,\cdots,k_m})=
\left\{\begin{array}{ll}
		\mu^{k_{s+2},\cdots,k_m} ,&\text{if } j_1=k_1,\cdots,j_s=k_s;\\
		0 ,&\text{otherwise}.
	\end{array}
	\right.
\]

For $\widehat{m}_3'$, we only need to discuss the following two cases,
\begin{itemize}
    \item [(1)] $\lambda^{j_1,\cdots,j_m}\in C^m(G,k) $, $\mu^{k_1,\cdots,k_n}\in C^n(G,k) $, $g_{l_1,\cdots,l_r}\in C_r(G,k)$ and $r+2\le m+n$, then
    \[
    \widehat{m}_3'( \lambda^{j_1,\cdots,j_m}, g_{l_1,\cdots,l_r}, \mu^{k_1,\cdots,k_n})=\sum_{i\in I} (-1)^{m+r+i-1} (\lambda\mu)^{j_1,j_2,\cdots,j_{m-r+i-2},k_{i+1},k_{i+2},\cdots,k_n },
    \]
    where $I=\{ i|\ \text{max}\{1,r+2-m\}\le i\le \text{min}\{n,r+1\}, j_{m-r+i-1}+k_i=4, l_1=k_1,l_2=k_2,\cdots,l_{i-1}=k_{i-1},l_{i+1}=j_{m-r+i},l_{i+1}=j_{m-r+i+1},\cdots,l_r=j_m \}$.

    In particular, if $r+2=m+n$,
    \[
    \widehat{m}_3'( \lambda^{j_1,\cdots,j_m}, g_{l_1,\cdots,l_r}, \mu^{k_1,\cdots,k_n})=
    \left\{\begin{array}{ll}
		-\lambda\mu ,&\text{if } j_1+k_n=4,l_1=k_1,\cdots,l_{n-1}=k_{n-1},l_n=j_2, \cdots,l_r=j_m;\\
		0 ,&\text{otherwise}.
	\end{array}
	\right.
    \]

    \item [(2)] $g_{j_1,\cdots,j_r}\in C_r(G,k) $, $g_{l_1,\cdots,l_s}\in C_s(G,k)$, $\lambda^{k_1,\cdots,k_m}\in C^m(G,k) $ and $m-1\le r+s$,
    \[
    \widehat{m}_3'(g_{j_1,\cdots,j_r},\lambda^{k_1,\cdots,k_m},g_{l_1,\cdots,l_s} )=\sum_{i\in J}(-1)^{m+r+s-i} \lambda g_{l_1,\cdots,l_i,4-k_{m-s+i},j_{m-s+i},\cdots,j_r},
    \]
    where $J=\{i|\text{ max}\{0,s+1-m\}\le i\le \text{min}\{s,r-m+s+1\},k_1=j_1,\cdots,k_{m-s+i-1}=j_{m-s+i-1}; k_{m-s+i+1}=l_{i+1},\cdots,k_m=l_s  \}$.
%     \[
% J = \left\{
%   i \,\middle|\,
%   \begin{aligned}
%     &\max\{0,\,s+1-m\} \le i \le \min\{s,\,r-m+s+1\},\\
%     &k_1 = j_1,\,\dots,\,k_{m-s+i-1} = j_{m-s+i-1},\\
%     &k_{m-s+i+1} = l_{i+1},\,\dots,\,k_m = l_s
%   \end{aligned}
% \right\}.
% \]

    In particular, if $m=1$,
    \[
    \widehat{m}_3'(g_{j_1,\cdots,j_r}, \lambda^{k_1},g_{l_1,\cdots,l_s} )=(-1)^{r+1} \lambda g_{l_1,\cdots,l_s,4-k_1,j_1,\cdots,j_r}.
    \]
\end{itemize}

\subsubsection{\texorpdfstring{$G=\mathbb{Z}_2\times \mathbb{Z}_2$}{G Z2 Z2}}\

Assume that $G=\{ g_3,g_2,g_1,1 \}$, whose multiplication table is given by
\[
\begin{array}{c|cccc}
     \cdot & 1 & g_1 & g_2 & g_3 \\
\hline
     1 & 1 & g_1 & g_2 & g_3 \\
   g_1 & g_1 & 1 & g_3 & g_2 \\
   g_2 & g_2 & g_3 & 1 & g_1 \\
   g_3 & g_3 & g_2 & g_1 & 1 \\
\end{array}
\]
Similar to the $\mathbb{Z}_4$ case, we can also introduce the following notations: for any $n\ge 1$,
\begin{itemize}

    \item define the map $\lambda^{j_1,\cdots,j_n}\in \mathrm{Map} (\overline{G}^{\times n},k)$ by 
    \[
    \lambda^{j_1,\cdots,j_n} (g_{j_1',\cdots,j_n'})=
    \left\{\begin{array}{ll}
		\lambda,&\text{if } j_1'=j_1,\cdots,j_n'=j_n;\\
		0 ,&\text{otherwise}.
	\end{array}
	\right.
    \]
    
    \item define the map $c_i:\ I_3^{\times n}\to I_3^{\times n+1}$ as 
    \[
      (j_1,\cdots,j_n)\mapsto 
      \left\{\begin{array}{ll}
		(j_1,\cdots,j_{i-1},2,3,j_{i+1},\cdots,j_n)+(j_1,\cdots,j_{i-1},3,2,j_{i+1},\cdots,j_n) ,&\text{if } j_i=1;\\
        (j_1,\cdots,j_{i-1},1,3,j_{i+1},\cdots,j_n)+(j_1,\cdots,j_{i-1},3,1,j_{i+1},\cdots,j_n) ,&\text{if } j_i=2;\\
		(j_1,\cdots,j_{i-1},1,2,j_{i+1},\cdots,j_n)+(j_1,\cdots,j_{i-1},2,1,j_{i+1},\cdots,j_n) ,&\text{if } j_i=3.
	\end{array}
	\right.
    \]

    \item define the map $d_i:\ I_3^{\times n}\to I_3^{\times n-1}$ as 
    \[
      (j_1,\cdots,j_n)\mapsto 
      \left\{\begin{array}{ll}
		(j_1,\cdots,j_{i-1},6-j_i-j_{i+1},j_{i+2},\cdots,j_n),&\text{if } j_i\neq j_{i+1};\\
		0 ,&\text{otherwise.}
	\end{array}
	\right.
    \]
\end{itemize}

The complex $(\widehat{C}^*(G,k),\delta_*)$ is defined as follows:
\[
\widehat{C}^*(G, k)=\left(\cdots \xrightarrow{\partial_{2}} C_{1} (G, k) \xrightarrow{\partial_{1}} C_{0} (G, k) \xrightarrow{\tau} C^{0} (G, k) \xrightarrow{\delta^{0}} C^{1} (G, k) \xrightarrow{\delta^{1}} \cdots\right),
\]
\begin{itemize}
    \item [(i)] for $n\ge 0$,
    \[
    \widehat{C}^n (G, k)=C^n (G, k)= \mathrm{Map}(\overline{G}^{\times n}, k),
    \]
    and the differential is given by $\delta_n =\delta^n$, sends $\lambda^{j_1,\cdots,j_n}$ to
    \[
    \sum_{j_0=1}^3\lambda^{j_0,j_1,\cdots,j_n}+\sum_{i=1}^n (-1)^i\lambda^{c_i(j_1,\cdots,j_n)}+(-1)^{n+1}\sum_{j_{n+1}=1}^3 \lambda^{j_1,\cdots,j_n,j_{n+1}}.
    \]

    \item [(ii)] for $n\le -1$ (let $s=-n-1\ge 0$),
    \[
    \widehat{C}^n(G, k)=C_s(G, k)= k[\overline{G} ^{\times s} ], \text{ and for } n\le -3,
    \]
    $\delta_n=\partial_s: k [\overline{G}^{\times s} ] \rightarrow  k [\overline{G}^{\times(s-1)} ]$ is defined by 
    \[
    g_{j_1,\cdots,j_s} \mapsto g_{j_2,\cdots,j_s}+\sum_{i=1}^{s-1}(-1)^i g_{d_i(j_1,\cdots, j_s)} +(-1)^{s}  g_{j_1,\cdots, j_{s-1}}.
    \]
    Moreover, $\delta_{-2}=\partial_1=0$.

    \item [(iii)] $\delta_{-1}=\tau: C_0(G, k)=k \rightarrow k=C^0(G, k)$ is given by $\tau(1)  =4$.
\end{itemize}
Then we obtain $\widehat{m}_1'$. The computation of $\widehat{m}_2'$ is the same as in $\mathbb{Z}_4$ case.
For $\widehat{m}_3'$, there is only a minor difference. To make the outcome clearer, we rewrite the result explicitly below.
\begin{itemize}
    \item [(1)] $\lambda^{j_1,\cdots,j_m}\in C^m(G,k) $, $\mu^{k_1,\cdots,k_n}\in C^n(G,k) $, $g_{l_1,\cdots,l_r}\in C_r(G,k)$ and $r+2\le m+n$, then
    \[
    \widehat{m}_3'( \lambda^{j_1,\cdots,j_m}, g_{l_1,\cdots,l_r}, \mu^{k_1,\cdots,k_n})=\sum_{i\in I} (-1)^{m+r+i-1} (\lambda\mu)^{j_1,j_2,\cdots,j_{m-r+i-2},k_{i+1},k_{i+2},\cdots,k_n },
    \]
    where $I=\{ i|\ \text{max}\{1,r+2-m\}\le i\le \text{min}\{n,r+1\}, j_{m-r+i-1}=k_i, l_1=k_1,l_2=k_2,\cdots,l_{i-1}=k_{i-1},l_{i+1}=j_{m-r+i},l_{i+1}=j_{m-r+i+1},\cdots,l_r=j_m \}$.

    In particular, if $r+2=m+n$,
    \[
    \widehat{m}_3'( \lambda^{j_1,\cdots,j_m}, g_{l_1,\cdots,l_r}, \mu^{k_1,\cdots,k_n})=
    \left\{\begin{array}{ll}
		-\lambda\mu ,&\text{if } j_1=k_n,l_1=k_1,\cdots,l_{n-1}=k_{n-1},l_n=j_2, \cdots,l_r=j_m;\\
		0 ,&\text{otherwise}.
	\end{array}
	\right.
    \]

    \item [(2)] $g_{j_1,\cdots,j_r}\in C_r(G,k) $, $g_{l_1,\cdots,l_s}\in C_s(G,k)$, $\lambda^{k_1,\cdots,k_m}\in C^m(G,k) $ and $m-1\le r+s$,
    \[
    \widehat{m}_3'(g_{j_1,\cdots,j_r},\lambda^{k_1,\cdots,k_m},g_{l_1,\cdots,l_s} )=\sum_{i\in J}(-1)^{m+r+s-i} \lambda g_{l_1,\cdots,l_i,k_{m-s+i},j_{m-s+i},\cdots,j_r},
    \]
    where $J=\{i|\text{ max}\{0,s+1-m\}\le i\le \text{min}\{s,r-m+s+1\},k_1=j_1,\cdots,k_{m-s+i-1}=j_{m-s+i-1}; k_{m-s+i+1}=l_{i+1},\cdots,k_m=l_s  \}$.

    In particular, if $m=1$,
    \[
    \widehat{m}_3'(g_{j_1,\cdots,j_r}, \lambda^{k_1},g_{l_1,\cdots,l_s} )=(-1)^{r+1} \lambda g_{l_1,\cdots,l_s,k_1,j_1,\cdots,j_r}.
    \]
\end{itemize}

\bigskip 

\textbf{Acknowledgements}

This work  was supported by the National Key R$\&$D Program of China (No. 2024YFA1013803), by the National Natural Science Foundation of China (No. 12031014, 13004005 and 12371043)),  by  the Fundamental Research Funds for the Central Universities (No. 020314380037),   and by Shanghai Key Laboratory of PMMP (No.
22DZ2229014).

% The authors are grateful to Jun Chen, Xiaojun Chen, Vladimir Dotsenko,  Li Guo, Yunnan Li,  Zihao Qi,  Yunhe Sheng, Rong Tang etc       for many useful comments.     
% We are very grateful to these researchers for their interests and comments.
%

\end{document}